%% file: main.tex
\documentclass{article}[11pt]
\usepackage{titling}
\usepackage{ amssymb }
\usepackage{amsfonts}
\usepackage{amsthm}
\usepackage{amsmath}
\usepackage{bbm}
\usepackage{setspace}
\usepackage[normalem]{ulem}

\usepackage{float}

\usepackage[utf8]{inputenc}

\def\showauthornotes{0}

\def\showkeys{0}
\def\showdraftbox{0}
\def\showcolorlinks{1}
\def\usemicrotype{1}
\def\showfixme{0}

\usepackage{csquotes}

\usepackage[
backend=biber,
maxcitenames=50,
maxbibnames=99,
style=alphabetic,
]{biblatex}
\title{Coboundary and cosystolic expansion without dependence on dimension or degree}

\author{Yotam Dikstein\thanks{Institute for Advanced Study, USA. email: yotam.dikstein@gmail.com.} \;and Irit Dinur\thanks{Weizmann Institute of Science, ISRAEL. email: irit.dinur@weizmann.ac.il. Both authors are supported by Irit Dinur's ERC grant 772839.}}
\date{\today}

\input{macros.tex}

\begin{document}\clearpage\thispagestyle{empty}

\maketitle
\begin{abstract}
    We give new bounds on the cosystolic expansion constants of several families of high dimensional expanders, and the known coboundary expansion constants of order complexes of homogeneous geometric lattices, including the spherical building of $SL_n(\mathbb{F}_q)$. 
    The improvement applies to the high dimensional expanders constructed by Lubotzky, Samuels and Vishne, and by Kaufman and Oppenheim. 
    
    Our new expansion constants do not depend on the degree of the complex nor on its dimension, nor on the group of coefficients. This implies improved bounds on Gromov's topological overlap constant, and on Dinur and Meshulam's cover stability, which may have applications for agreement testing. 
    
    In comparison, existing bounds decay exponentially with the ambient dimension (for spherical buildings) and in addition decay linearly with the degree (for all known bounded-degree high dimensional expanders). Our results are based on several new techniques:
    \begin{itemize}
    \item We develop a new ``color-restriction'' technique which enables proving dimension-free expansion by restricting a multi-partite complex to small random subsets of its color classes.
    \item We give a new ``spectral'' proof for Evra and Kaufman's local-to-global  theorem, deriving better bounds and getting rid of the dependence on the degree. This theorem bounds the cosystolic expansion of a complex using coboundary expansion and spectral expansion of the links.   
    \item We derive absolute bounds on the coboundary expansion of the spherical building (and any order complex of a homogeneous geometric lattice) by constructing a novel family of very short cones. 
    \end{itemize}
    \end{abstract}
\newpage
\tableofcontents
    
\newpage\pagestyle{plain}
\setcounter{page}{1}
\section{Introduction}
High dimensional expansion, which is a generalization of graph expansion to higher dimensional objects, is an active topic in recent years. The importance of graph expansion across many areas of computer science and mathematics, suggests that high dimensional expansion may also come to have significant impact. So far we have seen several exciting applications including analysis of convergence of Markov chains \cite{ALOV}, and constructions of locally testable codes and quantum LDPC codes \cite{DinurELLM2022, PanteleevK22}. 

Several notions of expansion that are equivalent in graphs, such as convergence of random walks, spectral expansion, and combinatorial expansion, turn out to diverge into two main notions in higher dimensions. 

The first is the notion of local link expansion which has to do with the expansion of the graph underlying each of the links of the complex; where a link is a sub-complex obtained by taking all faces that contain a fixed lower-dimensional face. This notion is qualitatively equivalent to convergence of random walks, it implies agreement testing, and it captures a spectral similarity between a (possibly sparse) high dimensional expander and the dense complete complex. It allows a spectral decomposition of functions on the faces of the complex in the style of Fourier analysis on the Boolean hypercube, see \cite{DiksteinDFH2018,KaufmanO-RW20, GurLL22, BafnaHKL22, GaitondeHKLZ22}. 

The second notion is coboundary and cosystolic expansion. Here we look at the complex not only as a combinatorial object but also as a sequence of linear maps, called coboundary maps, defined by the incidence relations of the complex. The $i$-th coboundary map $\coboundary_i$ maps a function on the $i$-faces to a function on the $i+1$-faces, 
\[ C^0 \stackrel{\coboundary_0}\to C^1  \stackrel{\coboundary_1}\to \cdots \stackrel{\coboundary_{d-1}}\to C^d 
\]
where $C^i = C^i(X,\F_2) = \set{f:X(i)\to\F_2}$ is the space of functions on $i$ faces with coefficients in $\F_2$ (we will consider general groups of coefficients, beyond $\F_2$). These functions are called \(i\)-chains. The coboundary map $\coboundary_i$ is defined in a very natural way: the value of $\coboundary f(s)$ for any $s\in X(i+1)$ is the sum of $f(t)$ for all $s\supset t\in X(i)$ (the precise definition is in \pref{sec:preliminaries}).

Coboundary (or cosystolic\footnote{The difference between coboundary and cosystolic expansion is just whether the cohomology is $0$ or not (i.e. whether $Ker \coboundary_{i+1} = Im \coboundary_i$). This distinction is not important for this exposition and the expansion inequality is the same in both cases.}) expansion captures how well the coboundary map tests its own kernel, in the sense of property testing. Given $f\in C^i$ such that $\coboundary f \approx 0$, coboundary expansion guarantees existence of some $g\in \ker \coboundary_i$ such that $f\approx g$. More precisely, a complex is a $\beta$ coboundary (or cosystolic) expander if
\[wt(\coboundary f)\geq \beta \cdot \min_{g\in Ker \delta}\dist(f,g)  
\]
where \(wt(\coboundary f)\) is the hamming weight of \(\coboundary f\).
We denote by $h^i(X)$ the largest value of $\beta$ that satisfies the above inequality for all $f$. 

Whereas for $i=0$ coboundary expansion coincides with the combinatorial definition of edge expansion, for larger $i$, it may appear at first glance to be quite mysterious. However, this definition is far from being a merely syntactical generalization of the $i=0$ case and turns out to provide a rich connection between topological and cohomological concepts and between several important concepts in TCS, which we describe briefly below. 

The study of coboundary and cosystolic expansion was initiated independently by Linial, Meshulam and Wallach \cite{LinialM2006},  \cite{MeshulamW09} in their study of connectivity of random complexes, and by Gromov \cite{Gromov2003} in his work on the topological overlapping property. Kaufman and Lubotzky \cite{KaufmanL2014} were the first to realize the connection between this definition and property testing.  This point of view is important in the recent breakthroughs constructing locally testable codes and quantum LDPC codes \cite{DinurELLM2022,PanteleevK22} (see also earlier works \cite{EvraKZ20}). 

Moreover, the coboundary maps come from a natural way to associate a (simplicial) complex to a constraint satisfaction problem. Attach a Boolean variable to each \(i\)-face, and view the \((i+1)\)-faces as parity constraints. The value that an assignment \(f:X(i)\to \mathbb{F}_2\)  gives on \(s \in X(i+1)\) is \(\coboundary f(s)\).
This connection to CSPs has been harnessed towards showing that the CSPs derived from certain cosystolic expanders are hard to refute for resolution and for the sum of squares hierarchy, \cite{DinurFHT2020, HopkinsL2022}. 

In addition, cosystolic expansion of $1$-chains (with non-abelian coefficients) of a complex has been connected to the stability of its topological covers \cite{DinurM2019}. Informally, a complex is cover-stable if slightly faulty simplicial covers are always ``fixable'' to valid simplicial covers. Perhaps surprisingly, this is related to agreement testing questions, particularly in the small 1\% regime, which is a basic PCP primitive and part of the initial motivation for this work. We discuss agreement testing and its relation to coboundary expansion in more detail further below in this introduction. 

In light of all of the above, we believe that cosystolic expansion is a fundamental notion that merits a deeper systematic study. Along with the aim of exploring its various implications, a more concrete research goal would be to give strong bounds, and ultimately nail down exactly, the correct expansion values for the most important and well-studied high dimensional expanders. We mention that to the best of our knowledge even for the simplest cases, such as expansion of $k$-chains in the $n$-simplex, exact expansion values are not yet completely determined.

In this work we provide new bounds for the coboundary expansion of the spherical building, and the cosystolic expansion of known bounded-degree high dimensional expanders including the complexes of \cite{LubotzkySV2005a, LubotzkySV2005b,KaufmanO2021}.

Two of the most celebrated results in this area are the works of \cite{KaufmanKL2014} and \cite{EvraK2016} showing that the bounded-degree families of Ramanujan complexes of \cite{LubotzkySV2005b} are cosystolic expanders. These works introduce an elegant local-to-global criterion, showing that if the links are coboundary expanders, and further assuming spectral expansion, then the entire complex is a cosystolic expander. 

The estimates proven by \cite{KaufmanKL2014, EvraK2016} for the coboundary expansion parameters are roughly \(h^k(X) \geq \min \left (\frac 1 Q, (d ! )^{-O(2^k)} \right ).\)
Here $X$ is a $d$ dimensional LSV complex and $Q$ is the maximal degree of a vertex which is roughly equal to $1/\lambda^{O(d^2)}$ in these complexes, where \(\lambda\) is the spectral bound on the expansion of the links. Subsequent works by Kaufman and Mass \cite{KaufmanM2018, KaufmanM2021, KaufmanM2022}, improved this bound to 
\begin{equation} \label{eq:prev-state-of-art}
    h^k(X) \geq \min \left (\frac{1}{Q}, (d!)^{-O(k)} \right ).    
\end{equation}

We completely get rid of the dependence on the ambient dimension $d$ and on the maximal degree $Q$, and prove
\begin{theorem}\label{thm:main-LSV}
 For every integer $d > 1$ and every small enough \(\lambda > 0\) let $X$ be a $d$-dimensional LSV complex whose links are $\lambda$-one-sided expanders. For every group \footnote{The theorem holds for every group $\Gamma$ for which cohomology is defined, namely, abelian groups for $k>1$ and any group for $k=1$.} $\Gamma$, every small enough \(\lambda > 0\) and every integer \(k < d-1\), $h^k(X,\Gamma)\geq \exp(-O(k^6\log k))$.
\end{theorem}
Our bounds for $h^k$ only depend on the dimension $k$ of the chains, so for $k=1$ they are absolute constants. For larger $k$ we still suffer an exponential decay. 
We do not know what the correct bound should be and whether dependence on $k$ is at all necessary.

The case of $k=1$ is interesting even in complexes whose dimension is $d \gg 1$, because $h^1$ controls the cover stability of the complex, as shown in \cite{DinurM2019}. 
Our bounds also immediately give an improvement for the topological overlap constants, when plugged into the Gromov machinery \cite{Gromov2010, DotterrerKW2018, EvraK2016}. 
We elaborate on both of these applications later below.

The result is proven by enhancing the local-to-global criterion of \cite{EvraK2016}, and introducing a variant of the local correction algorithm that makes local fixes only if they are sufficiently cost-effective. 
This is inspired by and resembles the algorithms in \cite{EvraK2016,DinurELLM2022,PanteleevK22}. 

Our analysis is novel and departs from previous proofs: instead of relying on the so-called ``fat machinery'' of \cite{EvraK2016} (and its adaptations \cite{KaufmanM2018, KaufmanM2021}), our proof is 100\% fat free and relies on the up/down averaging operators on \emph{real-valued functions}. 
Our main argument is to show that, for a function $h$ that is the indicator of the support of a (locally minimal) $k$-chain, 
\[\snorm {D\cdots D h} \gtrsim\cdots\gtrsim \snorm{DDh} \gtrsim  \snorm {Dh} \gtrsim \snorm h,\] 
where $D$ is the down averaging operator, and we write $a\gtrsim b$ whenever $a\geq \Omega(b)$. From here we easily derive a lower bound on $\snorm h$ showing that either the correction algorithm has found a nearby cocycle, or else the coboundary of our function was initially very large to begin with.

This method gives universal bounds on the cosystolic expansion of any complex whose links have both sufficient coboundary-expansion and sufficient local spectral expansion,
\begin{theorem} \torestate{\label{thm:cosystolic-expansion-clear}
    Let \(\beta, \lambda > 0\) and let \(k>0\) be an integer. Let \(X\) be a \(d\)-dimensional simplicial complex for \(d \geq k+2\) and assume that \(X\) is a \(\lambda\)-one-sided local spectral expander. Let \(\Gamma\) be any group. Assume that for every non-empty \(r \in X\), \(X_r\) is a coboundary expander and that \(h^{k+1-|r|}(X_r,\Gamma) \geq \beta\).
    Then \[h^k(X,\Gamma) \geq \frac{\beta^{k+1}}{(k+2)!\cdot 4} - e\lambda.\]
Here \(e \approx 2.71\) is Euler's number.}
\end{theorem}
Armed with an improved local-to-global connection, we derive \pref{thm:main-LSV} from \pref{thm:cosystolic-expansion-clear} by further strengthening the coboundary expansion of the links of the LSV complexes, namely spherical buildings. The best previously known bound on coboundary expansion of \(k\)-cochains in spherical buildings is due to \cite{Gromov2010} and \cite{LubotzkyMM2016}. They proved a lower bound of \( \left (\binom{d+1}{k+1} (d+2)! \right )^{-1} \). This decays exponentially with the ambient dimension \(d\), and with the cochain level $k$. We remove the dependence on $d$ by developing a new technique which we call ``color-restriction''. The $d$-dimensional spherical buildings are colored, namely, they are $d+1$-partite. For a set of \(\ell\) colors \(F\subset[d+1]\), the color restriction \(X^F\) is the complex induced on vertices whose color is contained in \(F\). The restriction to the the colors of \(F\) reduces the dimension of \(X\) from \(d\) to \(\ell-1\). We say that a color restriction \(X^F\) is a \(\beta\)-local coboundary expander, if \(X^F\) is a \(\beta\)-coboundary expander, and the same holds for the intersection of \(X^F\) with links (neighbourhoods) of faces whose color is disjoint from \(F\). We show that if a typical color-restriction is a local coboundary expander, then the entire complex is a coboundary expander, and the expansion is independent of the dimension. Namely,
\begin{theorem}\label{thm:color-restriction}
Let \(k,\ell, d\) be integers so that \(k+2\leq \ell \leq d\) and let \(\beta,p \in (0,1]\). Let \(X\) be a \((d+1)\)-partite $d$-dimensional simplicial complex so that
\[\Prob[F \in \binom{[d+1]}{\ell}]{X^F \text{ is a \(\beta\)-locally coboundary expander}} \geq p.\]
Then \(h^{k}(X) \geq \frac{p \beta^{k+1}}{e(k+2)!}\).
\end{theorem} 

Finally, to prove that the spherical building satisfies the conditions of this theorem, we need to show that a typical random color-restriction is a good coboundary expander. For this we rely on the ``cone machinery'' developed by Gromov \cite{Gromov2010}, Kozlov and Meshulam \cite{KozlovM2019}, and Kaufman and Oppenheim \cite{KaufmanO2021}. We construct in \pref{sec:geometric-lattices-coboundary-expansion}, a novel family of short cones, thus proving the following.
\begin{theorem}\label{thm:cones-intro}
    Let \(k \geq 0\). There is an absolute constant \(\beta_k =\exp(-O(k^5 \log k)) \geq 0 \) so that the following holds. Let \(X\) be the \(SL_n(\mathbb{F}_q)\)-spherical building for any integer \(n \geq k+1\) and prime power \(q\). Let \(\Gamma\) be any group. Then \(X\) is a coboundary expander with constant \(h^k(X,\Gamma) \geq \beta_k\). 
\end{theorem}
In fact, we prove a more general version of this theorem, that holds for the order complex of any homogeneous geometric lattice, see \pref{thm:lattice-k-coboundary-expansion}.\\

Most earlier works on cosystolic expansion focus on  $\mathbb{F}_2$ coefficients (see \cite{KaufmanM2018} and \cite{DinurM2019} for two exceptions). This is an important case especially in light of Gromov's result connecting $\F_2$-expansion and topological overlap. However, expansion (of $1$-chains) with respect to more general coefficients is necessary for results on topological covers and in turn for agreement testing. The theorems stated above show expansion of $k$-chains with respect to coefficients not only in $\F_2$ but in general abelian groups $\Gamma$, and when $k=1$ also for non abelian groups $\Gamma$. In other words, the theorems hold for all groups of coefficients where the cohomology is defined. 

Finally, we end with an upper bound. While most of our work is focused on lower bounds for coboundary and cosystolic expansion, we show in \pref{app:upper-bounding-cosystolic-expansion} that families of dense simplicial complexes cannot have cosystolic expansion greater than \(1+o(1)\). This implies that high degree, in some weak sense, limits cosystolic expansion. It is interesting to compare this to a result of Kozlov and Meshulam that shows upper bounds on coboundary expansion of complexes with bounded degree \cite{KozlovM2019}.

\subsection{Applications of cosystolic expansion}
We describe two applications of cosystolic expansion for deriving topological properties of simplicial complexes.
\paragraph{Topological overlap}Cosystolic expansion was studied by \cite{Gromov2010} to give a combinatorial criterion for the topological overlapping property. Let \(f:X\to \R^k\) be continuous mapping (with respect to the natural topology on \(X\)), i.e. \(f\) realizes \(X\) in \(\R^k\). A point \(p \in \R^k\) is called \(c\)-heavily covered if
\[ \Prob[s \in X(k)]{p \in f(s)} \geq c.\]
A well known result by \cite{FurediK1981} showed that for every affine map from the complete \(2\)-dimensional complex to the plane, there exists a \(\frac{1}{27}\)-heavily covered point. Gromov's greatly generalized this theorem to all \emph{continuous} functions (instead of only affine functions), all dimensions \(k\) (instead of \(k=2\)) and complexes that are cosystolic expanders (instead of the complete complex), with \(c\) that depends on the dimension of the map \(k\), as well as the cosystolic expansion constant. For a precise statement, see \pref{sec:concrete-complexes}. 

The motivation for \cite{EvraK2016} was to show that there exists families of bounded degree simplicial complexes which have this property. They use \cite{LubotzkySV2005b} complexes and achieve a lower bound of \(c \geq \min (\frac{1}{Q},(d!)^{-O(2^k)} )\), which comes from their bound on cosystolic expansion. This bound has been improved as a corollary of \cite{KaufmanM2022} to \(\min (\frac{1}{Q},(d!)^{-O(k)})\). Here again, \(d\) is the dimension of \(X\), which may be much larger than \(k\), and \(Q\) is the maximal degree of a vertex in \(X\).

Plugging in our bounds into Gromov's theorem gives  the bound \(c \geq \exp(-O(k^7 \log k))\) for the topological overlapping property. This bound is free of the ambient dimension and of the degree. 

\paragraph{Cover stability} The second author and Meshulam studied a topological locally testable property called \emph{cover stability} \cite{DinurM2019}. This property is equivalent to cosystolic expansion of \(1\)-chains. A covering map between two simplicial complexes \(X,Y\) is a surjective \(t\)-to-1 simplicial map\footnote{simplicial means that every \(i\)-face in \(Y\) is sent to an \(i\)-face in \(X\).} \(\rho:Y(0) \to X(0)\) such that for every \(\tilde{u} \in Y(0)\) and \(\rho(\tilde{u})=u \in X(0)\), it holds that the links of \(\tilde{u},u\) are isomorphic \(Y_{\tilde{u}} \cong X_u\).

Graph covers (also known as lifts) have been quite useful in construction of expander graphs. Bilu and Linial showed that random covers of Ramanujan graphs are almost Ramanujan \cite{BiluL2006}. A celebrated result by \cite{MarkusSS2015} used these techniques to construct bipartite Ramanujan graphs of every degree. Recently, \cite{Dikstein2023} showed that random covers could also be applied for constructing new simplicial complexes that are local spectral expanders.

Dinur and Meshulam \cite{DinurM2019} show that there exists a test that for any simplicial complex $X$ and an alleged cover given by a simplicial map $\rho :Y\to X$ samples \(q\) points \((u_i, \rho(u_i))\) and measures how close \(\rho\) is to an actual covering map. The query complexity of the test is \(q=3t\) points. Its soundness is affected by the cosystolic expansion of \(1\)-chains. Using our new bounds on cosystolic expansion, we show that the complexes constructed in \cite{LubotzkySV2005b} or in \cite{KaufmanO2021} are cover-stable, i.e. that there exists some universal constant \(c>0\), such that for every \(\rho:Y(0)\to X(0)\)
    \[\Prob[(u_i,\rho(u_i))_{i=1}^q]{\text{test fails}} \geq c \cdot \min \sett{\dist(\rho, \psi)}{\psi:Y(0)\to X(0) \text{ is a cover}},\]
where the distance is Hamming distance.

\paragraph{Agreement testing} Coboundary expansion found an exciting new application in agreement testing \cite{GotlibK2022,DiksteinD2023agr,BafnaM2023}. An agreement test is a consistency test that originated as a component in low degree testing \cite{RuSu96,ArSu,RaSa}, but has been extensively studied ever since (see e.g.\ \cite{GolSaf97,ImpagliazzoKW2012,DinurK2017}). This test is a crucial component in many PCP constructions \cite{Raz-parrep,GolSaf97,Dinur07, ImpagliazzoKW2012,DinurM11}. Given a set of partial functions on a set, an agreement test is a way to test whether these functions are correlated with some function that defined on the whole vertex set. The works \cite{GotlibK2022,DiksteinD2023agr,BafnaM2023} mentioned above use coboundary expansion to characterize when an agreement test is sound. Via this characterization they analyze agreement tests on high dimensional expanders. Continuing this line of works, \cite{DiksteinD2023swap,DiksteinDL2024,BafnaLM2024} use theorems and tools developed in a preliminary version of this paper, to lower bound coboundary expansion of new high dimensional expanders, and with these lower bounds they obtain new agreement tests. These include the first agreement tests where the underlying complex family is bounded degree in the so called `list decoding regime' (the regime that is relevant to high-soundness PCPs such as the parallel repetition PCP \cite{Raz-parrep,ImpagliazzoKW2012}). 

\subsection{Related work}
Coboundary and Cosystolic expansion was defined by Linial, Meshulam and Wallach \cite{LinialM2006}, \cite{MeshulamW09}, and indpendently by Gromov \cite{Gromov2010}. Gromov studied cosystolic expansion as a proxy for showing the topological overlapping property. Linial, Meshulam and Wallach were interested in analyzing high dimensional connectivity of random complexes.

Kaufman, Kazhdan and Lubotzky \cite{KaufmanKL2014} introduced an elegant local to global argument for proving cosystolic expansion of $1$-chains in the \emph{bounded-degree} Ramanujan complexes of \cite{LubotzkySV2005a,LubotzkySV2005b}. This was significantly extended by Evra and Kaufman \cite{EvraK2016} to cosystolic expansion in all dimensions, thereby resolving Gromov's conjecture about existence of bounded degree simplicial complexes with the topological overlapping property in all dimensions. Kaufman and Mass \cite{KaufmanM2018, KaufmanM2021} generalized the work of Evra and Kaufman from \(\mathbb{F}_2\) to all other groups as well, and used this to construct lattices with good distance. The best previously known bound for LSV complexes \eqref{eq:prev-state-of-art} is by Kaufman and Mass in \cite{KaufmanM2022}. In our proof we define a random walk which is actually a distributional version of the double balanced sets in \cite{KaufmanM2022}; a key component in their lower bound.

Following ideas that appeared implicitly in Gromov's work, Lubotzky Mozes and Meshulam analyzed the expansion of many ``building like'' complexes \cite{LubotzkyMM2016}. Kozlov and Meshulam \cite{KozlovM2019} abstracted the main lower bound in \cite{LubotzkyMM2016} to the definition of cones (which they call chain homotopies), in order to analyze the coboundary expansion of geometric lattices and other complexes. Their work also connects coboundary expansion to other homological notions, and gives an upper bound to the coboundary expansion of bounded degree simplicial complexes.
In \cite{KaufmanO2021}, Kaufman and Oppenheim defined the notion of cones in order to analyze the cosystolic expansion of their high dimensional expanders (see \cite{KaufmanO181}). In addition, they also come up with a criterion for showing that complexes admit short cones. They prove lower bounds on the cosystolic expansion of their complexes for \(0\)- and \(1\)-chains. The case of $k$-chains with \(k\geq 2\) is still open.

Several works tried to define quantum LDPC codes as cohomologies of simplicial complexes. Cosystolic expansion is used for analyzing the distance of the quantum code. Works by Evra, Kaufman and Zémor \cite{EvraKZ20} and by Kaufman and Tessler \cite{KaufmanT2021} used cosystolic expansion in Ramanujan complexes to construct quantum codes that beat the \(\sqrt{n}\)-distance barrier. This continued in a sequence of works \cite{PK1,HastingsHO21,BE} which culminated in the breakthrough work of \cite{PanteleevK22} that construct quantum LDPC codes with constant rate and distance. This later code is a cohomology of a certain chain complex, albeit not a simplicial complex; and it is analyzed essentially through the cosystolic expansion. Developing new techniques for cosystolic expansion can be potentially useful in this domain as well.

\subsection{Open questions}
The works by \cite{LubotzkyMM2016}, \cite{KozlovM2019} and \cite{KaufmanO2021} analyze a variety of symmetric complexes (that support a transitive group action). Could one combine our ``color restriction'' technique with the cone machinery to get lower bounds independent of degree and dimension on these complexes as well?
There are a number of concrete constructions of local spectral high dimensional expanders that have excellent local spectral properties \cite{ChapmanLP2020,LiuMY2020, Golowich2021,ODonnellP2022, Dikstein2023}. Are any of them cosystolic expanders?

Another intriguing direction of research is to develop additional techniques for analyzing coboundary or cosystolic expansion. The current techniques are limited to complexes that either have a lot of symmetry, or have excellent local expansion properties. Are there other complexes with these properties?

Our expansion bounds still have a dependence on the level ($k$) of the chains. In the complete complex, for instance, this is not necessary. The complete complex is a \(\beta = 1+o(1)\) coboundary expander for all \(k\)-chains \cite{LubotzkyMM2016}. It is not clear whether a dependence on \(k\) is necessary even in the spherical building. Which complexes have coboundary expansion that does not decay with the size of the chains? 

Finally, the notion of coboundary and cosystolic expansion is closely related to locally testable codes and quantum LDPC codes. They also have connections to agreement expanders. It is interesting to find additional applications for these expanders.

\subsection{Overview of the proof of \pref{thm:main-LSV}}
We start with a complex \(X\) that is a finite quotient of the affine building, as constructed by \cite{LubotzkySV2005b}. Our goal is to lower bound the cosystolic expansion of \(X\). The proof has three components:
\begin{itemize}
    \item (\pref{thm:cosystolic-expansion-clear}) A new local-to-global argument that derives cosystolic expansion of the complex from coboundary and spectral expansion of its links. 
    \item (\pref{thm:color-restriction}) A general color restriction technique that reduces the task of analyzing the coboundary expansion of a partite complex, to that of analyzing the local coboundary expansion of random color restrictions of it.
    \item (\pref{thm:cones-intro}) Bounds on random color restrictions of (links of) the spherical building. Towards this end we construct a novel family of short cones for the spherical building (not based on apartments as in previous works \cite{LubotzkyMM2016}).
\end{itemize}

Below we give a short overview of each of these steps. For simplicity we assume in this subsection that \(\Gamma = \mathbb{F}_2\), which captures the main ideas.

\paragraph{The local to global argument, \pref{thm:cosystolic-expansion-clear}}

Let \(X\) be our simplicial complex. We describe a correction algorithm, that takes as input a \(k\)-chain \(f:X(k) \to \mathbb{F}_2\), with small coboundary \(\prob{\coboundary f \ne 0} = \varepsilon\) and outputs a \(k\)-chain \(\tilde{f}:X(k) \to \mathbb{F}_2\) close to \(f\) that has no coboundary, i.e. \(\coboundary \tilde{f} = 0\). For this overview, we focus on \(k=1\), i.e. \(f\) is a function on edges, which already exhibits the main ideas. 

Let \(\eta > 0\) be some predetermined parameter. Our algorithm locally fixes ``stars'' of lower dimensional faces, that is, sets \(Star_k(r) = \sett{s \in X(k)}{s \supseteq r}\) for \(r \in X(j)\) (when \(j \leq k\)). The fix takes place only if it is sufficiently useful: whenever it decreases the weight of \(\coboundary f\) by at least \(\eta \prob{Star_k(r)}\). 
In the case at hand, \(k=1\), so $r$ is either a vertex or an edge, so
\begin{enumerate}
    \item If \(r\in X(1)\), \(Star_1(r) = \set{r}\) and a fix just means changing the value of \(f(r)\).
    \item If \(r\in X(0)\), \(Star_1(r)=\set{ru}_{u \sim r} \) are all edges adjacent to \(r\). Here a fix means changing the values of all \(\sett{f(ru)}{u \sim r}\) simultaneously.
\end{enumerate}

\begin{algorithm}  \label{alg:cosystolic-fixing-overview}
~
    \begin{enumerate}
        \item Set \(f_0 := f\). Set \(i=0\).
        \item While there exists a vertex or edge \(r \in X(0) \cup X(1)\) so that \(Star_k(r)\) has an assignment that satisfies a \(\eta \prob{Star_k(r)}\)-fraction of faces more than the current assignment.
        \begin{itemize}
            \item Let \(fix_r:Star_k(r) \to \Gamma\) be an optimal assignment to \(Star_k(r)\).
            \item Set \(f_{i+1}(s) = \begin{cases} f_i(t) & r \not\subseteq s \\
            fix_r(s) & r \subseteq s
            \end{cases}\).
            \item Set i:=i+1.
        \end{itemize}
        \item Output the final function \(\tilde{f}:=f_i\).
    \end{enumerate}
\end{algorithm}
The fact that we correct \(f\) locally only if the fix satisfies $\eta$ fraction more triangles will promise that \(\dist(f,\tilde{f}) \leq \frac{1}{\eta} wt(\coboundary f)\). The output of the algorithm, \(\tilde f\), is \emph{not} necessarily locally minimal in the sense of \cite{KaufmanKL2014,EvraK2016}, but it is \emph{``\(\eta\)-locally-minimal''}.

Notation: For functions \(g,h:X(\ell) \to \RR\) we denote by \(\iprod{g,h} = \Ex[r \in X(\ell)]{g(r)h(r)}\) the usual inner product. For \(\ell=1,2\), denote by \(D^{\ell}\) the \emph{down operator} that takes \(h:X(2) \to \R\) and outputs \(D^\ell h:X(2-\ell) \to \R\) via averaging. Namely \(D^\ell h(r)\) is the average of \(h(s)\) over \(s \supseteq r\), \(\Ex[s\supseteq r]{h(s)}\).

Let $h:X(2)\to\RR$ indicate the support of a \(\coboundary \tilde{f}\), so $h(t)=1$ iff $\delta\tilde f \neq 0$. Our main argument is to show 
\[\snorm {D^{3} h} \gtrsim \snorm{D^2 h} \gtrsim  \snorm {Dh} \gtrsim \snorm h.\] 
Eventually \(D^{3} h =\ex{h}^2\) is just a constant function. This shows that \((\ex{h})^2 = const\cdot \ex{h}\) which implies that either the algorithm corrected \(f\) to a cosystol, i.e. \(h=0\), or that \(h\) has large weight, which implies that \(\coboundary f\) had large weight to begin with.

Let us show for example that \(\snorm {D^{3} h} \gtrsim \snorm{D^2 h}\) given that \(\snorm{D^2 h} \gtrsim \snorm {Dh} \gtrsim \snorm{h}\). To do so, we define an auxiliary averaging operator \(N\) based on a random walk from vertices to triangles, and use the fact that in local spectral expanders, 
\begin{equation} \label{eq:approx-hdx-overview}
    \snorm {D^{3} h} \approx \iprod{N h, D^2 h}.
\end{equation}
 The operator $N:\ell_2(X(2))\to\ell_2(X(0))$ is defined by \(N h(v) = \Ex[s]{h(s)}\), where \(s\) is sampled according to the following walk: Given \(v \in X(0)\), sample some \(t \in X(3)\) such that \(v \in t\), and then go to the triangle \(s = t \setminus \set{v}\).  We mention that the concept of localizing over such a distribution has appeared in \cite{KaufmanM2022}.
The proof of \eqref{eq:approx-hdx-overview} follows by localizing the expectation to the links and relying on the link expansion as in \cite{Oppenheim2018}, \cite[Claim 8.8]{DiksteinDFH2018} and in \cite{KaufmanO-RW20}.

The key lemma in the proof shows that if there are many faces \(s' \supseteq v_0\) such that \(h(s')=1\), then there are many \(s\) such that \(v\notin s, \set{v} \cup s =t \in X(3)\), where \(h(s)=1\). More precisely, we will show that for every \(v \in X(0)\) it holds that
\begin{equation}\label{eq:overview-un-inequality}
Nh(v) \gtrsim \beta (D^2 h(v) - \eta)\footnote{We remark that without the additive error of \(\eta\) (that comes from replacing local minimality with \(\eta\)-local minimality), this inequality is similar to that proven in \cite[Theorem 5.2]{KaufmanM2022}. Relaxing this inequality with \(\eta\), is what allows us to improve upon these previous works.}.
\end{equation}

This immediately implies that
\begin{align*}
  \iprod{N h, D^2 h} &= \Ex[v]{D^2 h(v) Nh(v)}  \\
  &\overset{\eqref{eq:overview-un-inequality}}{\gtrsim} \beta (\Ex[v]{(D^2 h(v))^2} - \eta \Ex[v]{D^2h(v)}) \\
  &\gtrsim \beta \snorm{D^2 h} - \beta \eta \snorm{h} \\
  &\gtrsim \beta \snorm{D^2 h}.
\end{align*}
The second inequality follows from \(\Ex[v]{D^2 h(v)}=\Ex[s]{h(s)}=\snorm{h}\). The last inequality follows from the assumption that \(\snorm{h} = O(\snorm{D^2 h})\). Combining this with \eqref{eq:approx-hdx-overview} gives us the desired inequality.

Let us understand what is written in \eqref{eq:overview-un-inequality}. On the right-hand side, \(D^2 h(v) = \Prob[xy \in X_{v}(1)]{h(v xy) = 1}\) is the fraction of triangles \(v xy\) containing \(v\), such that \(\coboundary \tilde{f}(v xy) \ne 0\).
On the left-hand side, \(N h(v)\) is the fraction of \(s\) that complete \(v\) to some \(t = v \cup s \in X(3)\), so that \(\coboundary \tilde{f}(s) \ne 0\). For such an \(s=uxy\),
\begin{align} \label{eq:cob-cob-equality}
0 = \coboundary \coboundary \tilde{f}(v uxy) = \coboundary \tilde{f}(uxy) + (\coboundary \tilde{f}(v ux) + \coboundary \tilde{f}(v uy) + \coboundary \tilde{f}(v xy)).
\end{align}
Set \(g:X_{v}(1)\to \mathbb{F}_2\) to be \(g(xy)=\coboundary \tilde{f}(v xy)\), and note that \(g\) has the following properties:
\begin{enumerate}
    \item By \eqref{eq:cob-cob-equality}, 
\(\coboundary \tilde{f}(uxy) = 1 \iff \coboundary g(uxy) = 1\).
    \item \(\prob{g \ne 0} = \Prob[s \ni v]{\coboundary \tilde{f}(s) \ne 0}= D^2 h(v)\).
    \item $\eta$-local-minimality: \(\dist(g,B^1(X_{v})) \geq \prob{g \ne 0} - \eta\), where $B^1(X_v) = \sett{\delta \psi}{\psi:X_v(0)\to\F_2}$ is the set of coboundaries.
\end{enumerate}
We explain the third item. Assume towards contradiction that \(\dist(g,B^1(X_{v})) < \prob{g \ne 0} - \eta\) and let \(\coboundary \psi\) be a coboundary closest to \(g\). Then by changing the values of \(\tilde{f}\) on \(Star(v)\) to be \(\tilde{f}'(v u):= \tilde{f}(v u) + \psi(u)\), we have that whenever \(g(xy)= \coboundary \psi(xy)\), then the fixed function satisfies \(\coboundary \tilde{f}'(v xy) = 0\). I.e. 
\[\dist(g, \coboundary \psi) = \Prob[v xy]{\coboundary\tilde{f}'(v xy) = 0} < \Prob[v xy]{\coboundary\tilde{f}(v xy) = 0} -\eta.\]
This is a contradiction to the $\eta$-local minimality of \(\tilde{f}\) which is guaranteed by the algorithm.

Here is where the coboundary expansion of \(X_{v}\) comes into play. By coboundary expansion, we have that \(\Prob{\coboundary g(uxy)=1} \geq \beta \dist(g,B^1(X_{v}))\).
By combining the above we will get that
\[ N h(v) = \Prob[uxy \in X_{v}(2)]{\coboundary \tilde{f}(uxy) \ne 0} \geq \beta (\Prob[xy \in X_{v}(1)]{g(xy) \ne 0} - \eta) = \beta (D^2 h(v) - \eta).\]

\paragraph{The ``color restriction'' technique, \pref{thm:color-restriction}} For this overview, assume that \(k=2\) The full details are in \pref{sec:coboundary-expansion-of-partite-complexes}. 
Let \(Y\) be a \(d\)-dimensional \((d+1)\)-partite complex so that a \(p\)-fraction of its color restrictions \(Y^F\) are \(\beta\)-local-coboundary expanders. We begin with a \(2\)-chain \(f:Y(2)\to \mathbb{F}_2\) with small coboundary, namely \(\Prob[s \in Y(3)]{\coboundary f(s) \ne 0} = \varepsilon\). We need to find a \(1\)-chain \(g:Y(1) \to \mathbb{F}_2\) so that \(\dist(f,\coboundary g) \leq O(\frac{\varepsilon}{\beta^3 p})\). 

We first select a random color restriction, i.e. a set of colors so that \(Y^F\) is a local coboundary expander, that the weight of \(\coboundary f\) when restricted to triangles whose colors are in \(F\) is close to weight of \(\coboundary f\) on all \(Y\). Averaging arguments guarantee that such \(F\) exists.
Using this \(F\), we construct \(g\) in three steps. In the first step we define \(g\) on edges with both endpoints colored in \(F\), \(uv \in Y^F\). In the second step we define \(g\) on edges with one endpoint colored  in \(F\), i.e. \(uv \in Y(1)\) where \(u \in Y^F\) and \(v \notin Y^F\). In the third step we define \(g\) on edges \(uv \in X(1)\) with neither endpoints colored  in \(F\), i.e. where \(u,v \notin Y^F\). Every step uses the values of \(g\) that were constructed in the step before. For \(k>2\) the \((k-1)\)-chain is constructed following a similar sequence of $k+1$ steps.
\begin{enumerate}
    \item We start with the values of \(g\) on edges \(vu \in Y^F(1)\). By the choice of \(F\), the weight of \(\coboundary f\) inside \(Y^F\) is roughly \(\varepsilon\). Local coboundary expansion implies that there exists a \(1\)-chain \(g_0\) whose coboundary is close to \(f\) on \(Y^F\). We set \(g(uv) = g_0(uv)\) for all \(uv \in Y^F(1)\).
    \item Next we define \(g\) on edges \(vu\) so that \(v \notin Y^F\) and \(u \in Y^F\). Fix some \(v \notin Y^F\). Let \(Y^F_v = \sett{s \in Y^F}{s\cup v \in Y}\). This is the color restriction of the link of \(v\). We wish to set values for $g(vu)$ for all edges $vu$ such that $u\in Y_v^F(0)$. We describe a system of equations that we use to set the values of $g$ on the edges $vu$ so as to satisfy a maximal number of equations. 
    For every \(u_1 u_2 \in Y^F_v(1)\), the triangle \(v u_1 u_2\) defines an equation: 
    \begin{equation}\label{eq:lv-eqs}
         f(v u_1 u_2) + g(u_1 u_2) = g(u_1 v) + g(u_2 v).
    \end{equation}
    Note that the left-hand side of the equation is known since we have the values of \(f\) on all triangles, and we already constructed \(g\) for edges \(u_1 u_2 \in Y^F(1)\). So the above is an equation with two unknowns. We set \(g(vu)\) simultaneously for all \(u \in Y_v^F(1)\) to be an assignment that satisfies the largest fraction of equations (ties broken arbitrarily).
    
    The idea behind this step is the following. Obviously, we'd like that \(f(v u_1 u_2) = g(u_1 u_2) + g(u_1 v) + g(u_2 v)\) for as many triangles as possible, so it makes sense to define \(g\) to satisfy the largest amount of equations \eqref{eq:lv-eqs}. Let \(h_v:Y^F_v(1) \to \mathbb{F}_2\) be the left-hand side of \eqref{eq:lv-eqs}, i.e. \(h_v(u_1 u_2) =  f(v u_1 u_2) + g(u_1 u_2)\). We want to find an assignment \(g_v:Y^F_v(0)\to \mathbb{F}_2\) so that \(h_v(u_1 u_2) = g_v(u_1) + g_v(u_2)\) for as many equations \eqref{eq:lv-eqs} as possible (and set \(g(vu) = g_v(u)\)). Finding a solution \(g_v:Y^F_v(0)\to \mathbb{F}_2\) that satisfies \eqref{eq:lv-eqs} is equivalent finding \(g_v\) so that \(h_v(u_1 u_2) = \coboundary g_v(u_1 u_2)\). Hence, to find an assignment that satisfies most of the equations is the same showing that \(h_v\) is close to a coboundary. In the analysis we show that \(\coboundary h_v \approx 0\). This together with the local coboundary expansion of \(Y^F\) (which says that \(h^{1}(Y^F_v,\mathbb{F}_2) \geq \beta\)) will show that indeed we can find satisfying \(\set{g_v}_{v \notin Y^F}\) so that \(f\approx\coboundary g\) where the distance is over edges \(uv\) where \(v \notin Y^F, u \in Y^F\).
    
    \item Finally we need to define the values of \(g\) on edges \(vu\) so that \(v,u \notin Y^F\). Let \(vu\) be such an edge. Every triangle \(uvw\) where \(w \in Y_{vu}^F(0)\) defines a constraint on \(g(vu)\):
     \begin{equation}\label{eq:lv-eqs-2}
         f(uvw) + g(uw) + g(vw) = g(uv).
    \end{equation}
     As in the previous case, \(f(uvw)\) is known, and \(g(uw),g(vw)\) were determined in step 2.
     We set \(g(vu) = maj \sett{f(uvw) + g(uw) + g(vw)}{w \in Y_{uv}^F(0)}\). Ties are broken arbitrarily. Here we use the local coboundary expansion of \(Y^F\) in a way similar to the previous step, to show that indeed \(f\approx \partial g\).
\end{enumerate}

\begin{figure}
    \centering
    \includegraphics[scale=0.35]{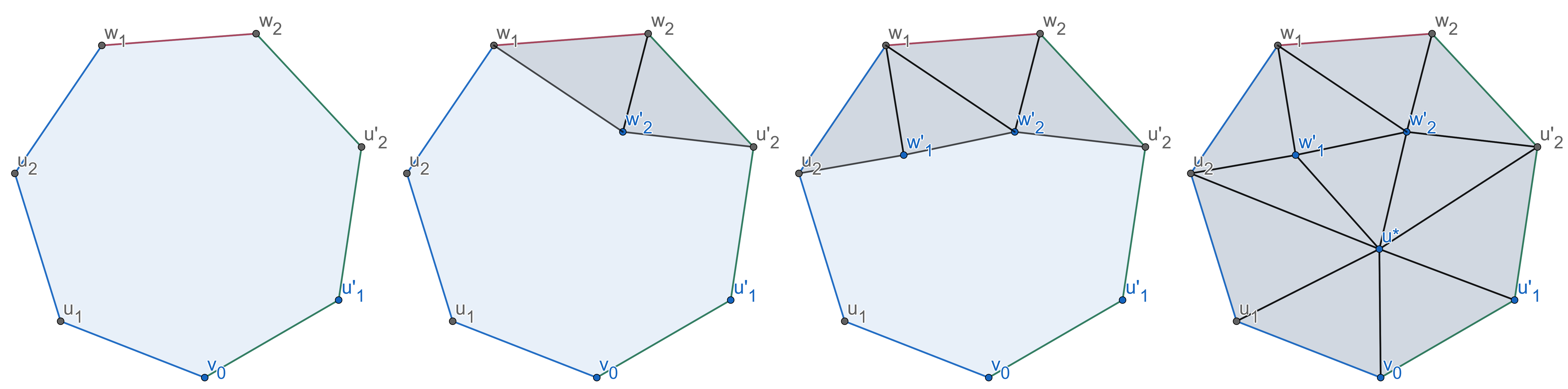}
    \caption{Tiling a cycle}
    \label{fig:general-cycle}
\end{figure}

\paragraph{New bounds on color-restrictions of the spherical building via cones, \pref{thm:cones-intro}}
In order to apply the color restriction technique we need to show that for a \(d\)-dimensional spherical building, many color restrictions are coboundary expanders\footnote{In fact, we need to show that the links of the color restrictions are also coboundary expanders, but we ignore this point in the overview for brevity.}. 
For this overview we assume that \(k=1\) and \(|F|=5\). Let us see how to bound coboundary expansion by constructing short cones.

It turns out easier to do so when the set of colors is a set of colors that are geometrically increasing (e.g. for \(k=1\) we need colors \(F= \set{i_1,i_2,...,i_5}\) so that \(i_j \geq 10 i_{j-1}\)).  The fraction of such sets of colors $F$ is a constant that doesn't depend on \(d\) (it may depend on \(k\)). For example, there is a constant probability that we select colors \(F\) so that for \(j=1,2,..,5\),  \(\frac{d}{10^{16-3j}} \leq i_j < \frac{2d}{10^{16-3j}}\) , since each of these intervals are a constant fraction of the interval \([1,2,...,d]\). When these inequalities hold then \(i_j \geq 10 i_{j-1}\).

Denote by \(Y\) the \(SL_d(\mathbb{F}_q)\)-spherical building. Let \(Y^F\) be a complex induced by the subspaces of dimensions (i.e., colors) \(F =\set{i_1,i_2,...,i_5}\) so that \(i_j \geq 10 i_{j-1}\)). Using the cone technology described in \pref{sec:geometric-lattices-coboundary-expansion}, showing the \(Y^F\) is a coboundary expander reduces to showing that there is a short \(1\)-cone on \(Y^F\). A \(1\)-cone consists of three things: 
\begin{enumerate}
    \item A vertex \(v \in X(0)\) (sometimes called the apex).
    \item For every $u$, a path \(p_u\) from the apex \(v\) to \(u\) in \(Y^F(1)\).
    \item For every edge $uw\in $, a tiling by triangles \(t_{uw}\subset Y^F(2)\)
    of the cycle that consists of the path \(p_u\) from \(v\) to \(u\), the edge \(uw\) and the path \(p_w\) from \(w\) back to \(v\). Denote this cycle by \(p_u \circ uw \circ p_w\). Here a tiling is a set of triangles whose boundary is the edges of the cycle. 
\end{enumerate}
We give a formal and general definition of cones in \pref{sec:geometric-lattices-coboundary-expansion}.
The radius of a cone is \(rad((v,\set{p_u}_{u\in Y^F(0)},\set{t_{uw}}_{uw \in Y^F(1)})) = \max_{uw \in X(1)} \abs{t_{uw}}\).

We start by choosing an apex \(v=v_0\) of dimension \(i_1\) arbitrarily. Next we choose our paths to be as short as possible, and to consist of subspaces of dimension as low as possible. Explicitly we do the following.
\begin{enumerate}
    \item For \(u\) adjacent to \(v_0\), set \(p_u=(v_0,u)\).
    \item For \(u\) of the same dimension as \(v_0\) we find some \(w\) of dimension \(i_2\) so that \(w\) is a neighbour of \(v_0\) and \(u\), and set \(p_u = (v_0,w,u)\). This is always possible since the dimension of \(u+v_0\) is at most \(2i_1\), so we can take any \(w\) of dimension \(i_2 \geq 2 i_1\) that contains the sum of spaces. (Notice how the fact that dimensions are geometrically increasing is important here).
    \item For other \(u \in Y^{F}(0)\), we first take some \(w_2 \subseteq u\) of dimension \(i_1\). Then we find some \(w_1\) who is a neighbour of \(v_0\) and of \(w_2\) and we set \(p_u = (v_0,w_1,w_2,u)\).
\end{enumerate}

Constructing \(t_{w_1 w_2}\) requires more care. Let us first consider the easier case. If \(dim(w_1),dim(w_2) \leq i_4\) then the cycle \(p_{w_1} \circ w_1 w_2 \circ p_{w_2}\) contains at most \(7\) vertices, all of dimension \(\leq i_4\). In particular, the sum of all the vertices/subspaces is of dimension at most \(7i_4 \leq i_5\), so there is a vertex \(u^*\) of dimension \(i_5\) that contains all the vertices in the cycle. The set of triangles \(u^* xy\) for all edges \(xy\) in the cycles is indeed a tiling of the cycle.

In the general case, it could be that the dimension of (say) \(w_1\) is \(i_5\). For example, assume that \(dim(w_1) = i_5, dim(w_2) = i_4\) (in particular \(w_2 \subseteq w_1\). It is useful to read this description while looking at \pref{fig:general-cycle}. In this case, we first find a tiling that ``shifts'' the cycle to a cycle of low dimension vertices. More explicitly, we find some \(w'_2 \subseteq w_2\) of dimension \(i_3\), that is also connected to \(w\)'s neighbours in the cycle. These neighbours are \(w_1\) (and any subspace of \(w_2\) is connected to it), and some \(u'_2\) of dimension \(\leq i_2\), so we can indeed find some \(w'_2\) that is connected to \(u\) and \(u'_2\) of dimension \(i_3\). We tile the cycle with \(w_2w'_2 u'_2, w_2 w'_2 w_1\). This exchanges \(w_2\) with \(w_2'\) in the untiled cycle. We perform a similar vertex-switch, for \(w_1\) as well, finding some \(w'_1\) of dimension \(i_4\) that is connected to \(w_1\) neighbours in the untiled cycle. After these two steps, we can find a \(u^*\) that is connected to all the (now low-dimensional) cycle as in the previous case.

\subsection{Organization of this paper}
\pref{sec:preliminaries} contains preliminaries. We prove \pref{thm:cosystolic-expansion-clear} that connects coboundary expansion in links to cosystolic expansion in \pref{sec:cosystolic-expansion} via the local correction algorithm. We develop the ``color restriction'' technique and prove \pref{thm:color-restriction} in \pref{sec:coboundary-expansion-of-partite-complexes}. We analyze the expansion of the spherical building and other homogeneous geometric lattices in \pref{sec:geometric-lattices-coboundary-expansion}. We tie everything up and prove \pref{thm:main-LSV} in \pref{sec:concrete-complexes}. In this section we present applications of our new bounds for better cover stability and topological overlap. In \pref{app:upper-bounding-cosystolic-expansion} we show an upper bound on the cosystolic expansion of dense complexes.

\subsection{Acknowledgment}
We are deeply grateful to Lewis Bowen for his important feedback on our paper, which greatly improved its clarity and readability.

\input{preliminaries.tex}
\input{spectral.tex}
\input{partite_complex_thm.tex}
\input{k-exp-for-spherical-building.tex}

\input{concrete-complexes.tex}
\printbibliography

\appendix
\input{appendixAnew}
\input{coboundary-expansion-never-greater-than-1.tex}
\end{document}

%% file: macros.tex
\providecommand{\RR}{\mathbb{R}}
\providecommand{\NN}{\mathbb{N}}

\usepackage[l2tabu, orthodox]{nag}

\usepackage{xspace,enumerate}

\usepackage[dvipsnames]{xcolor}

\usepackage[T1]{fontenc}
\usepackage[full]{textcomp}

\usepackage[american]{babel}

\usepackage{mathtools}

\usepackage{amsthm}
\usepackage{amsmath}

\newtheorem{theorem}{Theorem}[section]
\newtheorem*{theorem*}{Theorem}

\newtheorem{proposition}[theorem]{Proposition}
\newtheorem*{proposition*}{Proposition}
\newtheorem{lemma}[theorem]{Lemma}
\newtheorem*{lemma*}{Lemma}
\newtheorem{corollary}[theorem]{Corollary}
\newtheorem*{corollary*}{Corollary}
\newtheorem*{conjecture*}{Conjecture}

\newtheorem*{fact*}{Fact}

\newtheorem*{hypothesis*}{Hypothesis}

\theoremstyle{definition}
\newtheorem{definition}[theorem]{Definition}
\newtheorem*{definition*}{Definition}

\newtheorem{algorithm}[theorem]{Algorithm}
\newtheorem{test}[theorem]{Test}

\theoremstyle{remark}
\newtheorem{claim}[theorem]{Claim}
\newtheorem*{claim*}{Claim}
\newtheorem{remark}[theorem]{Remark}
\newtheorem*{remark*}{Remark}

\newtheorem*{observation*}{Observation}

 \usepackage[letterpaper,
 top=1in,
 bottom=1in,
 left=1in,
 right=1in]{geometry}

\usepackage[varg]{pxfonts} 

\usepackage{lmodern}

\ifnum\showkeys=1
\usepackage[color]{showkeys}
\fi

\ifnum\showcolorlinks=1
\usepackage[
colorlinks=true,
urlcolor=blue,
linkcolor=blue,
citecolor=OliveGreen,
]{hyperref}
\fi

\ifnum\showcolorlinks=0
\usepackage[
colorlinks=false,
pdfborder={0 0 0}
]{hyperref}
\fi

\usepackage{prettyref}

\newcommand{\savehyperref}[2]{\texorpdfstring{\hyperref[#1]{#2}}{#2}}

\newrefformat{eq}{\savehyperref{#1}{\textup{(\ref*{#1})}}}
\newrefformat{lem}{\savehyperref{#1}{Lemma~\ref*{#1}}}
\newrefformat{def}{\savehyperref{#1}{Definition~\ref*{#1}}}
\newrefformat{thm}{\savehyperref{#1}{Theorem~\ref*{#1}}}
\newrefformat{cor}{\savehyperref{#1}{Corollary~\ref*{#1}}}
\newrefformat{cha}{\savehyperref{#1}{Chapter~\ref*{#1}}}
\newrefformat{sec}{\savehyperref{#1}{Section~\ref*{#1}}}
\newrefformat{app}{\savehyperref{#1}{Appendix~\ref*{#1}}}
\newrefformat{tab}{\savehyperref{#1}{Table~\ref*{#1}}}
\newrefformat{fig}{\savehyperref{#1}{Figure~\ref*{#1}}}
\newrefformat{hyp}{\savehyperref{#1}{Hypothesis~\ref*{#1}}}
\newrefformat{alg}{\savehyperref{#1}{Algorithm~\ref*{#1}}}
\newrefformat{rem}{\savehyperref{#1}{Remark~\ref*{#1}}}
\newrefformat{item}{\savehyperref{#1}{Item~\ref*{#1}}}
\newrefformat{step}{\savehyperref{#1}{step~\ref*{#1}}}
\newrefformat{conj}{\savehyperref{#1}{Conjecture~\ref*{#1}}}
\newrefformat{fact}{\savehyperref{#1}{Fact~\ref*{#1}}}
\newrefformat{prop}{\savehyperref{#1}{Proposition~\ref*{#1}}}
\newrefformat{prob}{\savehyperref{#1}{Problem~\ref*{#1}}}
\newrefformat{claim}{\savehyperref{#1}{Claim~\ref*{#1}}}
\newrefformat{relax}{\savehyperref{#1}{Relaxation~\ref*{#1}}}
\newrefformat{red}{\savehyperref{#1}{Reduction~\ref*{#1}}}
\newrefformat{part}{\savehyperref{#1}{Part~\ref*{#1}}}
\newrefformat{ex}{\savehyperref{#1}{Example~\ref*{#1}}}
\newrefformat{para}{\savehyperref{#1}{Paragraph~\ref*{#1}}}

\newcommand{\Sref}[1]{\hyperref[#1]{\S\ref*{#1}}}

\usepackage{nicefrac}

\ifnum\usemicrotype=1
\usepackage{microtype}
\fi

\ifnum\showauthornotes=1
\newcommand{\Authornote}[2]{{\sffamily\small\color{red}{[#1: #2]}}}
\newcommand{\Authornotecolored}[3]{{\sffamily\small\color{#1}{[#2: #3]}}}
\newcommand{\Authorcomment}[2]{{\sffamily\small\color{gray}{[#1: #2]}}}
\newcommand{\Authorstartcomment}[1]{\sffamily\small\color{gray}[#1: }

\newcommand{\Authorfnote}[2]{\footnote{\color{red}{#1: #2}}}
\newcommand{\Authorfixme}[1]{\Authornote{#1}{\textbf{??}}}
\newcommand{\Authormarginmark}[1]{\marginpar{\textcolor{red}{\fbox{\Large #1:!}}}}
\else
\newcommand{\Authornote}[2]{}
\newcommand{\Authornotecolored}[3]{}
\newcommand{\Authorcomment}[2]{}
\newcommand{\Authorstartcomment}[1]{}

\newcommand{\Authorfnote}[2]{}
\newcommand{\Authorfixme}[1]{}
\newcommand{\Authormarginmark}[1]{}
\fi

\ifnum\showfixme=0

\fi

\usepackage{boxedminipage}

\newcommand{\brac}[1]{[#1]}
\newcommand{\Brac}[1]{\left[#1\right]}

\newcommand{\abs}[1]{\lvert#1\rvert}

\newcommand\sett[2]{\left\{ #1 \left| \; \vphantom{#1 #2} \right. #2  \right\}}
\newcommand{\set}[1]{\{#1\}}

\newcommand{\norm}[1]{\lVert#1\rVert}

\newcommand{\snorm}[1]{\norm{#1}^2}

\newcommand{\iprod}[1]{\langle#1\rangle}

\def\dim{\mathrm{ dim}}

\newcommand{\F}{\mathbb{F}}
\newcommand{\Esymb}{\mathbb{E}}
\newcommand{\Psymb}{\mathbb{P}}

\DeclareMathOperator*{\E}{\Esymb}

\DeclareMathOperator*{\ProbOp}{\Psymb}

\renewcommand{\Pr}{\ProbOp}
\def\one{{\mathbf{1}}}

\newcommand{\prob}[1]{\Pr \left[ {#1} \right] }
\newcommand{\Prob}[2][]{\Pr_{{#1}}\left[#2\right]} 
\newcommand{\cProb}[3]{\Pr_{{#1}}\left[ #2 \left| \; \vphantom{#2 #3} \right. #3  \right]} 

\newcommand{\ex}[1]{\E\brac{#1}}
\newcommand{\Ex}[2][]{\E_{{#1}}\Brac{#2}}

\newcommand{\ve}{\;\hbox{and}\;}

 \usepackage{dsfont}
\usepackage{mathrsfs}

\newcommand{\textparen}[1]{\text{(#1)}}

\ifx\because\undefined
\newcommand{\because}[1]{\textparen{because #1}}
\else
\renewcommand{\because}[1]{\textparen{because #1}}
\fi

\newcommand\bdot\bullet

\DeclareMathOperator{\supp}{supp}
\DeclareMathOperator{\dist}{dist}
\DeclareMathOperator{\sign}{sign}

\newcommand{\Z}{\mathbb Z}

\newcommand{\R}{\mathbb R}

\renewcommand{\leq}{\leqslant}

\renewcommand{\geq}{\geqslant}

\ifnum\showdraftbox=1

\else

\fi

\let\epsilon=\varepsilon

\numberwithin{equation}{section}

\newcommand{\MYstore}[2]{%
  \global\expandafter \def \csname MYMEMORY #1 \endcsname{#2}%
}

\newcommand{\MYload}[1]{%
  \csname MYMEMORY #1 \endcsname%
}

\newcommand{\MYnewlabel}[1]{%
  \newcommand\MYcurrentlabel{#1}%
  \MYoldlabel{#1}%
}

\newcommand{\MYdummylabel}[1]{}

\newcommand{\torestate}[1]{%
  \let\MYoldlabel\label%
  \let\label\MYnewlabel%
  #1%
  \MYstore{\MYcurrentlabel}{#1}%
  \let\label\MYoldlabel%
}

\newcommand{\restatetheorem}[1]{%
  \let\MYoldlabel\label
  \let\label\MYdummylabel
  \begin{theorem*}[Restatement of \prettyref{#1}]
    \MYload{#1}
  \end{theorem*}
  \let\label\MYoldlabel
}

\newcommand{\restatelemma}[1]{%
  \let\MYoldlabel\label
  \let\label\MYdummylabel
  \begin{lemma*}[Restatement of \prettyref{#1}]
    \MYload{#1}
  \end{lemma*}
  \let\label\MYoldlabel
}

\newcommand{\restateprop}[1]{%
  \let\MYoldlabel\label
  \let\label\MYdummylabel
  \begin{proposition*}[Restatement of \prettyref{#1}]
    \MYload{#1}
  \end{proposition*}
  \let\label\MYoldlabel
}

\newcommand{\restateclaim}[1]{%
  \let\MYoldlabel\label
  \let\label\MYdummylabel
  \begin{claim*}[Restatement of \prettyref{#1}]
    \MYload{#1}
  \end{claim*}
  \let\label\MYoldlabel
}

\newcommand{\restatecorollary}[1]{%
  \let\MYoldlabel\label
  \let\label\MYdummylabel
  \begin{corollary*}[Restatement of \prettyref{#1}]
    \MYload{#1}
  \end{corollary*}
  \let\label\MYoldlabel
}

\newcommand{\restatefact}[1]{%
  \let\MYoldlabel\label
  \let\label\MYdummylabel
  \begin{fact*}[Restatement of \prettyref{#1}]
    \MYload{#1}
  \end{fact*}
  \let\label\MYoldlabel
}

\newcommand{\restatedefinition}[1]{
\let\MYoldlabel\label 
\let\label\MYdummylabel 
\begin{definition*}[Restatement of \prettyref{#1}] 
    \MYload{#1} 
\end{definition*} 
\let\label\MYoldlabel 
} 

\newcommand{\restate}[1]{%
  \let\MYoldlabel\label
  \let\label\MYdummylabel
  \MYload{#1}
  \let\label\MYoldlabel
}

\let\origparagraph\paragraph
\renewcommand{\paragraph}[1]{\origparagraph{#1.}}

\allowdisplaybreaks

\newcommand{\dunion}{\mathbin{\mathaccent\cdot\cup}}

\let\pref=\prettyref

\newcommand{\dir}[1]{\overset{\to}{#1}}
\newcommand{\bproof}[1]{\begin{proof}[Proof of \pref{#1}]}
\newcommand{\eproof}{\end{proof}}
\newcommand{\coboundary}{{\delta}}
\newcommand{\boundary}{\partial}
\newcommand{\Img}{{\textrm Im}}
\newcommand{\Ker}{{\textrm Ker}}
\usepackage{cancel}

%% file: preliminaries.tex
\section{Preliminaries and notation} \label{sec:preliminaries}
\paragraph{Simplicial complexes} A pure \(d\)-dimensional simplicial complex \(X\) is a set system (or hypergraph) consisting of an arbitrary collection of sets of size \(d + 1\) together with all their subsets. The sets of size \(i+1\) in \(X\) are denoted by \(X(i)\), and in particular, the vertices of \(X\) are denoted by \(X(0)\). We will sometimes omit set brackets and write for example \(uvw\in X(2)\) instead of \(\set{u,v,w}\in X(2)\). As convention \(X(-1) = \set{\emptyset}\). Unless it is otherwise stated, we always assume that \(X\) is finite.
Let \(X\) be a \(d\)-dimensional simplicial complex. Let \(k \leq d\). We denote the set of oriented \(k\)-faces in \(X\) by \(\dir{X}(k) = \sett{(v_0,v_1,...,v_k)}{\set{v_0,v_1,...,v_k} \in X(k)}\). For \(s=(v_0,v_1,...,v_k) \in \dir{X}(k)\) we denote \(set(s) = \set{v_i}_{i=0}^k\), but when its clear from context we abuse notation and write \(s\) for its underlying set instead of \(set(s)\). 
For an oriented face \(s \in \dir{X}(k)\) and an index \(i\in \set{0,1,...,k}\), we denote by \(s_i\) the face obtained by removing the \(i\)-th vertex of \(s\).

Finally, Let \(s = (v_0,...,v_i)\), and \(t=(u_{0},...,u_{j})\). We denote by the concatenation \(s\circ t = (v_0,v_1,...,v_i,u_{0},u_1,...,u_j)\).

\paragraph{Probability over simplicial complexes}
Let \(X\) be a simplicial complex and let \(\Pr_d:X(d)\to (0,1]\) be a density function on \(X(d)\) (i.e. \(\sum_{s \in X(d)}\Pr_d(s)=1\)). This density function induces densities on lower level faces \(\Pr_k:X(k)\to (0,1]\) by \(\Pr_k(t) = \frac{1}{\binom{d+1}{k+1}}\sum_{s \in X(d),s \supset t} \Pr_d(s)\). We can also define a probability over directed faces, where we choose an ordering uniformly at random. Namely, for \(s\in \dir{X}(k)\), \(\Pr_k(s) = \frac{1}{(k+1)!}\Pr_k(set(s))\). When it's clear from the context, we omit the level of the faces, and just write \(\Pr[T]\) or \(\Prob[t \in X(k)]{T}\) for a set \(T \subseteq X(k)\).

\subsection{Coboundary and cosystolic expansion}
\paragraph{Asymmetric functions}
Let \(X\) be a \(d\)-dimensional simplicial complex. Let \(-1 \leq k \leq d\) be an integer. Let \(\Gamma\) be a group. A function \(f:\dir{X}(k)\to \Gamma\) is \emph{asymmetric} if for every \((v_0,v_1,...,v_k) \in \dir{X}(k)\), and every permutation \(\pi:[k]\to [k]\) it holds that
\[f(v_0,v_1,...,v_k) = f(v_{\pi(0)},v_{\pi(1)},...,v_{\pi(k)})^{sign(\pi)}.\]
We denote the set of these functions by \(C^k(X,\Gamma)\). We note that by fixing some order to the vertices \(X(0) = \set{v_0,v_1,...,v_n}\), there is a bijection between functions \(f:X(k)\to \Gamma\) and asymmetric functions \(\dir{f}:\dir{X}(k)\to \Gamma\). Given \(f:X(k)\to \Gamma\) and a set \(s=\set{v_{i_0},v_{i_1},...,v_{i_k}}\) so that \(i_0<i_1<...<i_k\), we set \(\dir{f}(v_{\pi(i_0)},v_{\pi(i_1)},...,v_{\pi(i_k)})=f(s)^{\sign (\pi)}\).

We record the following useful relation.
\begin{claim} \label{claim:algebraic-relation}
Let \(s \in X(j)\). For every \(x \in X_s\) and every asymmetric function \(g:X(k)\to \Gamma\) it holds that \(\sum_{i_1=0}^j \sum_{i_2=0}^{j-1}(-1)^{i_1+i_2}g((s_{i_1})_{i_2}\circ x) = 0.\)
\end{claim}

Let \(f:\dir{X}(k) \to \Gamma\). The weight of \(f\) is \(wt(f) = \Prob[t\in X(k)]{f(t)\ne 0}\). For two functions \(f,g:\dir{X}(k)\to \Gamma\) the distance between \(f\) and \(g\) is \(\dist(f,g) = wt(f-g) = \Prob[t\in X(k)]{f(t)\ne g(t)}\). 

\paragraph{Cohomology}
Let \(\Gamma\) be an abelian group. The coboundary operator \(\coboundary_k : C^k(X,\Gamma) \to C^{k+1}(X,\Gamma)\) is defined by 
\[\coboundary_k f(s) = \sum_{i=0}^k (-1)^i f(s_i).\]
It is a direct calculation to verify that for any asymmetric function \(f \in C^{k}\) the function \(\coboundary_k f\) is indeed an asymmetric function, and that \(\coboundary_{k+1} \circ \coboundary_k = 0\).

Let \(B^k(X,\Gamma) = \Img(\coboundary_{k-1})\) be the space of coboundaries. Let \(Z^k(X,\Gamma) = \Ker(\coboundary_{k})\) be the space of cosystols. As \(\coboundary_{k+1} \circ \coboundary_k = 0\), it holds that \(B^k(X,\Gamma) \subseteq Z^k(X,\Gamma)\). The \(k\)-cohomology is \(H^k(X,\Gamma) = Z^k(X,\Gamma) / B^k(X,\Gamma)\). 
\paragraph{Coboundary expansion}
For a function \(f:\dir{X}(k)\to \Gamma\) let   \(\dist(f,B^k) = \min_{g \in C^{k-1}} \dist(f,\coboundary g),\) be the minimal distance between \(f\) and a coboundary.
The \(k\)-th coboundary constant of a complex \(X\) (with respect to an abelian group \(\Gamma\)) is
\[h^k(X,\Gamma) = \min_{f  \in C^k\setminus  B^k}\frac{wt(\coboundary f)}{\dist(f,B^k)}.
\]  where $B^k = B^k(X,\Gamma)$.
Note that \(h^k(X,\Gamma) > 0\) if and only if \(H^k = 0\). 

\paragraph{Cosystolic expansion}
A very related high dimensional notion of expansion is cosystolic expansion. The \(k\)-th cosystolic expansion constant of $X$ (with respect to an abelian group \(\Gamma\)) is
\[h^k(X,\Gamma) = \min_{f \in C^k\setminus  Z^k}\frac{wt(\coboundary f)}{\dist(f,Z^k)},\] where $Z^k = Z^k(X,\Gamma)$.
Notice that when \(B^k(X,\Gamma) = Z^k(X,\Gamma)\), namely, when $H^k=0$, this coincides with the definition of coboundary expansion, and this justifies using the same notation $h^k$, where the term coboundary expansion (as opposed to cosystolic expansion) is taken to indicate $H^k=0$. 

Another useful way to understand the constant is the following.  \(h^k(X,\Gamma) \geq \beta\) if and only if for every \(f:\dir{X}(k)\to \Gamma\) there is some \(h \in Z^k(X,\Gamma)\) so that \(\beta \dist(f, h) \leq wt(\coboundary f)\). We note that in the work of \cite{EvraK2016} cosystolic expanders were also required to have no small weight \(f \in Z^k(X,\Gamma) \setminus B^k(X,\Gamma)\). We don't focus on this notion in our work.

\paragraph{Non abelian coboundary and cosystolic expansion}
For \(k=0,1\) we can define the cohomology with respect to non abelian groups as well. Let \(\Gamma\) be a non abelian group. As before, for every \(k\) we can define \(C^k(X,\Gamma)\). We define the coboundary operators as follows:
\begin{enumerate}
    \item \(\coboundary_{-1}:C^{-1}(X,\Gamma)\to C^{0}(X,\Gamma)\) is \(\coboundary_{-1} h (v) = h(\emptyset)\).
    \item \(\coboundary_{0}:C^{0}(X,\Gamma)\to C^{1}(X,\Gamma)\) is \(\coboundary_{0} h (vu) = h(v)h(u)^{-1}\).
    \item \(\coboundary_{1}:C^{1}(X,\Gamma)\to C^{2}(X,\Gamma)\) is \(\coboundary_{1} h (vuw) = h(vu)h(uw)h(wv)\).
\end{enumerate}
It is easy to check that \(\coboundary_{k+1} \circ \coboundary_k f = e\) where \(e \in \Gamma\) is the unit. The definitions for \(h^k(X,\Gamma)\) and coboundary expansion are the same as in the abelian case for \(k=0,1\).
\paragraph{Edge expander graphs}
\begin{definition}
    Let \(G = (V,E)\) be a graph and let \(\lambda > 0\). We say that \(G\) is a \(\lambda\)-edge expander if for every \(S \subseteq V\) so that \(0 < \prob{S} \leq \frac{1}{2}\), it holds that \(\prob{E(S,V \setminus S)} \geq \lambda \prob{S}\).
\end{definition}
It is well known that any \(\lambda\)-one-sided spectral expander is a \(\frac{1-\lambda}{2}\)-edge expander. The following claim shows that every \(\lambda\)-edge expander graph \(G\) also has \(h^0(G,\Gamma) \geq \frac{\lambda}{2}\) for any group \(\Gamma\). 
\begin{claim} \label{claim:edge-expander-many-sets}
Let \(G = (V,E)\) be a \(\lambda\)-edge expander. Let \(S_1,S_2,...,S_m \subseteq V\) be mutually disjoint sets so that \(V = \bigcup_{j=1}^m S_j\) and so that \(< \varepsilon\) edges cross between the sets for \(\varepsilon \leq \frac{\lambda}{2}\). Then there exists \(j\) so that \(\prob{S_j} \geq 1-\varepsilon/\lambda\).
\end{claim}

For any group \(\Gamma\) and \(h:X(0)\to \Gamma\), we take as sets \(S_g = h^{-1}(g)\). Edges so that \(\coboundary h \ne 0\) are edges that cross between \(S_{g},S_{g'}\) for some \(g\ne g'\). By \pref{claim:edge-expander-many-sets}, when there are \(\varepsilon\)-edges crossing the cut, then there exists some \(g:X(-1)\to \Gamma\) so that \(\prob{S_g} \geq 1-\frac{\varepsilon}{\lambda}\). Then \(\prob{h \ne \coboundary g} \geq \frac{\varepsilon}{\lambda}\).

\begin{proof}[Proof of \pref{claim:edge-expander-many-sets}]
Denote by \(E' \subseteq E\) the edges that cross between sets. We first show that there must exist some \(j\) so that \(\prob{S_j} > \frac{1}{2}\).

Assume that every \(S_i\) has measure less than \(\frac{1}{2}\). Then
\[\varepsilon > \prob{E} = \frac{1}{2}\sum_{j=1}^m \prob{E(S_j, V\setminus S_j)} \geq \frac{\lambda}{2}\sum_{j=1}^m \prob{S_j} \geq \frac{\lambda}{2}.\]
This contradicts the assumption that \(\varepsilon \leq \frac{\lambda}{2}\).

Hence \(\prob{S_j} >\frac{1}{2}\). Thus 
\[\varepsilon \geq \prob{E(S_j, V \setminus S_j)} \geq \lambda \prob{V \setminus S_j} = \lambda(1-\prob{S_j})\]
and \(\prob{S_j} \geq 1-\frac{\varepsilon}{\lambda}\).
\end{proof}

\subsection{Local properties of simplicial complexes}
\paragraph{Links of faces}
Let \(X\) be a \(d\)-dimensional simplicial complex. Let \(k < d\) and \(s \in X(k)\). The link of \(s\) is a \(d-k-1\)-dimensional simplicial complex defined by
\(X_s = \sett{t \setminus s}{t \in X, t \supseteq s}\). We point out that the link of the empty set is \(X_\emptyset = X\).

Let \(s \in X(k)\) for some \(k \leq d\). The density function \(\Pr_d\) on \(X\) induces on the link is \(\Pr^s_{d-k-1}:X(d-k-1)\to (0,1]\) where 
\({\Pr}_{d-k-1}^s[{t}] = \frac{\prob{t \cup s}}{\prob{s} \binom{d+1}{k+1}}\).
We usually omit \(s\) in the probability, and for \(T \subseteq X_s(k)\) we write \(\Prob[t \in X_s(k)]{T}\) instead.

\paragraph{High dimensional local spectral expanders}
Let \(X\) be a \(d\)-dimensional simplicial complex. Let \(k \leq d\). The \(k\)-skeleton of \(X\) is \(X^{\leq k} = \bigcup_{j=-1}^k X(j)\). In particular, the \(1\)-skeleton of \(X\) is a graph.

\begin{definition}[Spectral expander]
    Let \(G=(V,E)\) be a graph (that is, a \(1\)-dimensional simplicial complex). Let \(A\) be its normalized adjacency operator, i.e. for every \(f:V\to \R\), \(Af:V \to \R\) is the function \(Af(v) = \Ex[uv \in E]{f(u)}\). Let \(1=\lambda_1\geq \lambda_2 \geq ... \geq \lambda_{|V|}\geq -1\) be the eigenvalues of \(A\).
    
    Let \(\lambda \geq 0\). We say that \(G\) is a \(\lambda\)-one sided spectral expander if \(\lambda_2 \leq \lambda\). We say that \(G\) is a \(\lambda\)-two sided spectral expander if \(\lambda_2 \leq \lambda\) and \(\lambda_{|V|}\geq -\lambda\).
\end{definition}

\begin{definition}[high dimensional local spectral expander]
    Let \(X\) be a \(d\)-dimensional simplicial complex. Let \(\lambda \geq 0\). We say that \(X\) is a \(\lambda\)-one sided (two sided) local spectral expander if for every \(s \in X^{\leq d-2}\), the \(1\)-skeleton of \(X_s\) is a \(\lambda\)-one sided (two sided) spectral expander.
\end{definition}

\paragraph{Partite complexes}
A \((d+1)\)-partite simplicial complex is a \(d\)-dimensional complex that has a partition \(X(0)=V_0 \dunion V_1 \dunion ... \dunion V_d \) so that for every \(s \in X(d)\) and every \(i=0,1,...,d\) it holds that \(|s \cap V_i|=1\). Let \(X\) a \((d+1)\)-partite simplicial complex. A color of a face \(t \in X(k)\) is \(col(t) = \sett{i \in [d]}{t \cap V_i \ne \emptyset}\). Let \(F \subseteq [d]\). We denote by \(X[F] = \sett{s \in X}{col(s)=F}\), and by \(X^F = \sett{s \in X}{col(s) \subseteq F}\). A probability density on \(X\) induces a probability density on \(X^F\), \(\Pr^{F}:X^F(|F|-1) \to (0,1]\) by \(\Pr^F(s) = \sum_{t \in X(d), s\subseteq t} \Pr[t]\).

\subsection{Complexes of interest}\label{sec:complexes}
\paragraph{The \(SL_d(\mathbb{F}_q)\)-spherical building} We do not define here spherical buildings in full generality; for a general definition see e.g. \cite{Bjrner1984}. Let \(q\) be a prime power, and let \(\mathbb{F}_q\) be the field with \(q\) elements. Let \(d>1\) be integers. The \(SL_d(\mathbb{F}_q)\)-spherical building is the following \(d-1\)-partite simplicial complex
\[X(0) = \sett{W \subseteq \mathbb{F}_q^d}{W \text{ is a vector subspace} \ve W \ne \set{0}, \mathbb{F}_q^d}.\]
\[ X(d-2) = \sett{\set{W_1,W_2,...,W_{d-1}}}{W_1 \subsetneq W_2 \subsetneq ...\subsetneq W_{d-1}}.\]
The probability over \(X(d-2)\) is uniform. The color of every vertex is its dimension.

\paragraph{Geometric lattices}
Let \((P, \preceq)\) be a finite poset. The order complex of \(P\), denoted \(X_P\), is the simplicial complex on the vertex set \(P\) whose faces are all \(\set{v_0,...,v_d}\) so that \(v_0 \prec v_1 \prec ... \prec v_d\), see \cite{Kozlov2008}. A poset \((P, \preceq)\) is a lattice if any two elements \(x,y \in P\) have a unique minimal upper bound \(x \vee y\) and a unique maximal lower bound \(x \wedge y\). Let \(P\) be a lattice with minimal element \(\hat{0}\) and maximal element \(\hat{1}\). Then it has a rank function \(rk:P \to \NN\), with \(rk(\hat{0}) = 0\) and \(rk(y) = rk(x) + 1\) whenever \(y\) is a minimal element of \(\set{z : z \succ x}\). P is a \emph{geometric lattice} if \(rk(x) + rk(y) \geq rk(x \vee y) + rk(x \wedge y)\) for any \(x,y \in P\), and any element in \(P\) is a join of atoms (i.e., rank 1 elements). An example for a geometric lattice is the lattice whose elements are all subspaces of \(\mathbb{F}_q^d\), with the containment partial order. It's order complex is the \(SL_d(\mathbb{F}_Q)\)-spherical building.

Let \(P\) be a geometric lattice, we denote by \(\bar{P} = P \setminus \set{\hat{0},\hat{1}}\). It is known (see \cite{Kozlov2008}) that if \(P\) is a geometric lattice, then \(X_{\bar{P}}\) is pure and \(d\)-partite, where \(d = rk(\hat{1})-1\). The color of every vertex is its rank. A homogeneous lattice \(P\) is a lattice so that \(Aut(P)\) act transitively on \(X_{\bar{P}}(d)\) (by the action \(\pi(s) = \sett{\pi(v)}{v \in s}\)). We show in \pref{sec:geometric-lattices-coboundary-expansion} that homogeneous geometric lattices have constant coboundary expansion.

Geometric lattices have spectral expansion properties.
\begin{claim}\label{claim:geometric-lattice-is-an-edge-expander}
    Let \(P\) be a homogeneous geometric lattice of of rank \(\geq 2\). Then for every \(i\) and \(j \geq 2i\) it holds that the bipartite graph between \(X_P[i], X_P[j]\) is a \(\frac{1}{\sqrt{2}}\)-one sided spectral expander.
\end{claim}

\begin{proof}[Proof of \pref{claim:geometric-lattice-is-an-edge-expander}]
    Note that every \(u,v \in X_P[i]\) there is a path of length \(2\) of the form \((u, w, v)\) where \(w \geq u \vee v\). This is because \(rk(u \vee v) \leq 2i \leq j\), and because in a rank \(d\) geometric lattice we can always embed any chain in a chain of length \(d\) (i.e. we can always find \(w \succeq u \vee v\) of rank \(j\)). Moreover, by homogeneity, the degree of every \(w\) is some constant \(D \geq 2\). Thus the operator of taking two steps in this graph (starting at \(X_P[i]\)) is equal to \(\frac{1}{D} Id + (1-\frac{1}{D})C\) where \(C\) is the operator of the complete graph. The second largest eigenvalue of this operator is \(\frac{1}{D} \leq \frac{1}{\sqrt{2}}\). Indeed, this is because taking two steps from any vertex \(w\) has probability of \(\frac{1}{D}\) to traverse in a path of the form \((w,v,w)\), and otherwise the endpoint of is uniform among all \(w' \ne w\).
    
    As this is a \(2\)-step walk according to the bipartite graph's operator, it follows that the second largest eigenvalue of the original graph is at most \(\frac{1}{\sqrt{2}}\).
\end{proof}

\paragraph{\cite{LubotzkySV2005b} complexes}
Lubotzky, Samuels and Vishne constructed the first bounded degree high dimensional expanders. They construct them by taking quotients of Bruhat-Titz buildings.

\begin{theorem}[\cite{LubotzkySV2005b}]\torestate{\label{thm:lsv-complexes}
For any prime power \(q\) and integer \(d>1\), there is a family \(\mathcal{X}_{q,d}=\set{X_n}_{n=1}^\infty\) of connected complexes whose links are (isomorphic copies of) the \(SL_d(\mathbb{F}_q)\)-spherical building. In particular, For every \(\lambda>0\) there is some \(q_0\) so that every \(X_n\) is a \(\lambda\)-one sided high dimensional expander when \(q\geq q_0\).}
\end{theorem}

\paragraph{\cite{KaufmanO2021} complexes}

Kaufman and Oppenheim give another construction of bounded degree partite high dimensional expanders, by a technique called group development \cite{KaufmanO2021}. We state the properties in their construction that are necessary for our needs.
\begin{theorem}[\cite{KaufmanO2021}] \torestate{\label{thm:ko-complexes}
For every \(\lambda > 0\) there exists a family of \(4\)-partite complexes \(\mathcal{Y}_{\lambda} = \set{Y_n}_{n=1}^\infty\) so that
\begin{enumerate}
    \item \(Y_n\) is a \(\lambda\)-one sided high dimensional expander.
    \item There exists a constant \(\beta > 0\) (independent of \(\lambda\)) so that for every abelian group \(\Gamma\) and every \(s \in Y_n(0)\), the link of \(s\) has \(h^{1}(Y_s,\Gamma) \geq \beta\).
\end{enumerate}}
\end{theorem}

%% file: spectral.tex
\section{Cosystolic expansion} \label{sec:cosystolic-expansion}
In this section we prove that local spectral expanders whose links are coboundary expanders are cosystolic expanders.

\restatetheorem{thm:cosystolic-expansion-clear}

In fact, we prove a slightly more general statement, allowing for different coboundary expansion in every level.
\begin{theorem} \label{thm:cosystolic-expansion-different-levels}
    Let \(k >0\) be an integer and let \(\beta_0,\beta_1,\beta_2,...,\beta_{k} \in (0,1]\) and \(\lambda > 0\). Let \(X\) be a \(d\)-dimensional simplicial complex for \(d \geq k+2\) and assume that \(X\) is a \(\lambda\)-one-sided local spectral expander. Let \(\Gamma\) be any group. Assume that for every \(0\leq \ell \leq k\) and \(r \in X(\ell)\), \(X_r\) is a coboundary expander and that \(h^{k-\ell}(X_r,\Gamma) \geq \beta_{k-\ell}\).
    Then \[h^k(X,\Gamma) \geq \frac{\prod_{\ell=0}^{k}\beta_\ell}{(k+2)!\cdot 4} - e\lambda.\]
Here \(e \approx 2.71\) is Euler's number.
\end{theorem}
Obviously, \pref{thm:cosystolic-expansion-clear} follows from \pref{thm:cosystolic-expansion-different-levels} by setting \(\beta_\ell=\beta\) for every \(\ell=0,1,2,...,k\).

The following proposition, that is important for the topological overlapping property will also be proven via similar arguments.
\begin{proposition} \label{prop:heavy-cosystols}
Let \(k >0\) be an integer and let \(\beta_0,\beta_1,\beta_2,...,\beta_{k-1} \in (0,1]\) and \(\lambda > 0\). Let \(X\) be a \(d\)-dimensional simplicial complex for \(d \geq k+1\) and assume that \(X\) is a \(\lambda\)-one-sided local spectral expander. Let \(\Gamma\) be any group. Assume that for every \(0\leq \ell \leq k-1\) and \(r \in X(\ell)\), \(X_r\) is a coboundary expander and that \(h^{k-\ell}(X_r,\Gamma) \geq \beta_{k-\ell-1}\). Then every \(g \in Z^{k}(X,\Gamma) \setminus B^{k}(X,\Gamma)\), has \(wt(g) \geq \frac{\prod_{\ell=0}^{k-1}\beta_\ell}{(k+1)!} - e\lambda\).
\end{proposition}
We remark that the when \(\Gamma\) is non abelian, these statements make sense only when \(k=1\).

Turning back to \pref{thm:cosystolic-expansion-different-levels}, we present a correction algorithm. We will show that when \(f \in C^{k}(X,\Gamma)\) has a small coboundary, then the algorithm below returns some \(\tilde{f} \in Z^k(X,\Gamma)\) that is close to \(f\).
\begin{algorithm} \label{alg:cosystolic-fixing}
    Input: A function \(f:\dir{X}(k) \to \Gamma\), a parameter \(\eta \leq 1\).
    Output: A function \(\tilde{f}:\dir{X}(k) \to \Gamma\).
    \begin{enumerate}
        \item Set \(f_0 := f\). Set \(i=0\).
        \item While there exists \(\ell \leq k\), and a face \(r \in \dir{X}(\ell)\) so that \(Star_k(r) = \sett{s \in X(k)}{r \subseteq s}\) has an assignment that satisfies a \(\eta \prob{Star_k(r)}\)-fraction of faces more than the current assignment, do:
        \begin{itemize}
            \item Let \(fix_r:Star_k(r) \to \Gamma\) be an optimal assignment to \(Star_k(r)\), satisfying the maximal number of $k+1$-faces containing \(r\).
            \item Set \(f_{i+1}(s) = \begin{cases} f_i(s) & r \not \subseteq s \\
            fix_r(s) & r \subseteq s
            \end{cases}\).
            \item Set i:=i+1.
        \end{itemize}
        \item Output \(\tilde{f}:=f_i\).
    \end{enumerate}
\end{algorithm}

\subsection{Properties of \pref{alg:cosystolic-fixing}}
Before proving \pref{thm:cosystolic-expansion-different-levels} we record some properties of \pref{alg:cosystolic-fixing}.  

\begin{claim} \label{claim:algorithm-halts}
\pref{alg:cosystolic-fixing} halts on every input.
\end{claim}

\begin{claim} \label{claim:distance-is-dominated-by-coboundary-and-beta}
 Let \(f:\dir{X}(k)\to \Gamma\) and let \(\eta \leq 1\). Let \(\tilde{f}:\dir{X}(k)\to \Gamma\) be the output of \pref{alg:cosystolic-fixing} on \((f,\eta )\). Then \(\eta \dist(f,\tilde{f}) \leq wt(\coboundary f)\).
\end{claim}

\begin{proof}[Proof of \pref{claim:algorithm-halts}]
Denote by \(\varepsilon_j = \Prob[t\in X(k+1)]{\coboundary f_j(t) \ne 0}\). If we show that \(\varepsilon_{j+1} < \varepsilon_j\), then as \(X\) is finite the algorithm must halt. Indeed, fix \(j\) and let \(r\) be the face that was fixed in the \(j\)-th step. If \(t \in X(k+1)\) doesn't contain \(r\) then it holds that \(\coboundary f_j (t) = \coboundary f_{j+1}(t)\) so 
\[\cProb{t\in X(k+1)}{\coboundary f_j(t) \ne 0}{t \not \supseteq r} = \cProb{t\in X(k+1)}{\coboundary f_{j+1}(t) \ne 0}{t \not \supseteq r}.\]

For faces containing $t$, \pref{alg:cosystolic-fixing} changed the values of \(f_j\) to satisfy more of the \(k+1\)-faces containing $t$, so all in all $\varepsilon_{j+1}<\varepsilon_j$.
\end{proof}

\begin{proof}[Proof of \pref{claim:distance-is-dominated-by-coboundary-and-beta}]
Let \(i\) be  so that \pref{alg:cosystolic-fixing} returned \(\tilde{f}=f_i\). By the triangle inequality \(\dist(f,f_i) \leq \sum_{j=0}^{i-1}\dist(f_j,f_{j+1})\). Following the notation of \pref{claim:algorithm-halts}, let \(\varepsilon_j = wt(\coboundary f_j)\). If we show that 
   \[\eta \dist(f_j,f_{j+1}) \leq \varepsilon_j - \varepsilon_{j+1}\]
then
    \[\eta \dist(f,f_i) \leq \sum_{j=0}^{i-1} \varepsilon_j - \varepsilon_{j+1} = \varepsilon_0 - \varepsilon_i= wt(\coboundary f)-\varepsilon_i \leq wt(\coboundary f).\]
Indeed, fix \(j\) and let \(r\) be the \(\ell\)-face that was selected in the \(j\)-th step of the algorithm. Recall that \(Star_k(r)\) is the set of \(k\)-faces that contain \(r\). Then \(\dist(f_j,f_{j+1}) \leq \prob{Star_k(r)}\), since the change between \(f_j,f_{j+1}\) was only on faces in \(Star_k(r)\). However, the difference \(\varepsilon_j - \varepsilon_{j+1}\) is greater or equal to  \(\eta\prob{Star_k(r)}\), since otherwise the algorithm wouldn't change anything. Combine the two inequalities:
\[\eta \dist(f_j, f_{j+1}) \leq \eta \prob{Star_k(r)} \leq \varepsilon_j - \varepsilon_{j+1}.\]
\end{proof}

\subsection{Local minimality}
\begin{definition}[Restriction]
Let \(g \in C^{k}(X,\Gamma)\) and let \(r \in X(\ell)\) for some \(0\leq \ell \leq k-1\). The restriction of \(g\) to \(r\) is the function \(g_r \in C^{k-\ell-1}(X_r,\Gamma)\) is defined by \(g_r(p)=g(r \circ p)\). 
\end{definition}

\begin{definition}[Local minimality]
Let \(\eta \geq 0\) and let  \(g \in C^{k}(X,\Gamma)\). We say that \(g\) is \(\eta\)-locally minimal, if for every \(0\leq \ell \leq k-1\), every \(r \in X(\ell)\), and every \(h \in C^{k-\ell-2}(X_r,\Gamma)\) it holds that
\[wt(g_r) \leq wt(g_r + \coboundary h) + \eta.\]
\paragraph{The non-abelian case}If \(\Gamma\) is non-abelian we need the correct analogy to adding coboundaries. The definition of \(\eta\)-minimality is as follows. If \(k=1\), we say that \(g\) is \(\eta\)-locally minimal if for every \(v \in X(0)\), and every \(\gamma \in \Gamma\), it holds that 
\[wt(g_v) \leq wt(\gamma \cdot g_v) + \eta.\]

If \(k=2\), we say that \(g\) is \(\eta\)-locally minimal if:
\begin{enumerate}
    \item For every edge \(uv\) and every \(\gamma \in \Gamma\), it holds that \(wt(g_{uv})\leq wt(\gamma \cdot g_{uv}) + \eta\).
    \item For every vertex \(v\) and every function \(h:X_v(0)\to \Gamma\), it holds that \(wt(g_v) \leq wt(g_v^h) + \eta\), where
    \(g_v^h(uw) = h^{-1}(u)g_v(uw)h(w)\).
\end{enumerate}
\end{definition}

\begin{claim} \label{claim:output-is-close-to-optimal}
Let \(f:\dir{X}(k)\to \Gamma\) and let \(\eta \leq 1\). Let \(\tilde{f}:\dir{X}(k)\to \Gamma\) be the output of \pref{alg:cosystolic-fixing} on \((f,\eta )\). Then \(\coboundary \tilde{f}\) is \(\eta\)-locally minimal.
\end{claim}

\begin{proof}[Proof of \pref{claim:output-is-close-to-optimal}]
Assume towards contradiction that there is some \(r\in X(j)\) and a \(h \in C^{k-|r|-1}(X_r,\Gamma)\) so that \(wt((\coboundary \tilde{f})_r) > wt((\coboundary \tilde{f})_r + \coboundary h) + \eta\). We define \(fix_r:X(k) \to \Gamma\) to be \(fix_r(t) = h(p)\) if \(t=r \circ p\), and zero if \(r \not \subseteq t\).

By definition 
\begin{equation} \label{eq:meaning-of-weight-of-restricted-coboundary}
wt((\coboundary \tilde{f})_r) = \Prob[t \in X(k+1), t \supseteq r]{\coboundary \tilde{f}(t) \ne 0}
\end{equation}
and
\[ wt((\coboundary \tilde{f})_r + \coboundary h) = \Prob[t \in X(k+1), t \supseteq r]{\coboundary (\tilde{f} + fix_r) (t)}.\] 
But then \pref{alg:cosystolic-fixing} would have added \(fix_r\) to \(\tilde{f}\) (or another even better function), and wouldn't return \(\tilde{f}\), a contradiction.

We remark that the same idea holds in the non-abelian case where \(\coboundary \tilde{f} \in C^2(X,\Gamma)\), even though the case analysis is cumbersome. Equation \eqref{eq:meaning-of-weight-of-restricted-coboundary} is still true. Thus,
\begin{enumerate}
    \item Let \(\gamma \in \Gamma\) and \(r=uv \in X(1)\). For every triangle \(uvw \in X(2)\), the value of \(\gamma (\coboundary \tilde{f})_r(w) = (\gamma \cdot \tilde{f}(uv)) \tilde{f}(vw) \tilde{f}(wu)\). By \pref{alg:cosystolic-fixing}, changing the value of \(\tilde{f}(uv)\) to \(\gamma \cdot \tilde{f}(uv)\) cannot decrease the weight of \(\coboundary \tilde{f}\) by more than \(\eta\).
    \item If \(r \in X(0)\) and \(h:X_r(0)\to \Gamma\). Then for every triangle \(rvw\in X(2)\), it holds that
    \[(\coboundary \tilde{f})_r^h(vw) = h(v)\tilde{f}(rv) \tilde{f}(vw) \tilde{f}(wr) h(w)^{-1} = (h(v)\tilde{f}(rv)) \tilde{f}(vw) (h(w)\tilde{f}(rw))^{-1}.\]
    By \pref{alg:cosystolic-fixing}, changing the values of \(\set{\tilde{f}(rx)}\) for the edges \(rx\) adjacent to \(r\), to the values \(h(x)\tilde{f}(rx)\) cannot decrease the weight of \(\coboundary \tilde{f}\) by more than \(\eta\).
\end{enumerate}
\end{proof}

\subsection{Locally minimal cosystols are heavy}
The following lemma states that non-zero functions that are locally minimal must have large weight.

\begin{lemma} \label{lem:loc-min-inequality}
    Let \(\beta_0,...,\beta_{k-1}\) and \(\lambda\) be as in \pref{thm:cosystolic-expansion-different-levels}. Let \(X\) be such that for every \(0\leq \ell \leq k-1\) and every \(s \in X(\ell)\) it holds that \(X_s\) is a coboundary expander and \(h^{k-\ell-1}(X_s,\Gamma) \geq \beta_{k-\ell-1}\). Assume further that \(X\) is a \(\lambda\)-local spectral expander. Let \(g \in Z^{k}(X,\Gamma)\) be non-zero and \(\eta\)-locally minimal. Then
    \begin{equation} \label{eq:loc-min-inequality}
    wt(g) \geq \frac{\prod_{\ell=0}^{k-1}\beta_\ell}{(k+1)!} - e(\eta + \lambda).
    \end{equation}

    Additionally, for the case of non-abelian \(\Gamma\), when $k=2$,  \eqref{eq:loc-min-inequality} holds for \(\eta\)-locally minimal and non-zero \(g = \coboundary f\), for any \(f \in C^1(X,\Gamma)\).
\end{lemma}
The last remark regarding $k=2$ is needed since $Z^2(X,\Gamma)$ is not defined for non-abelian groups $\Gamma$.

This lemma implies \pref{thm:cosystolic-expansion-clear} and \pref{prop:heavy-cosystols} directly.

\begin{proof}[Proof of \pref{thm:cosystolic-expansion-different-levels}, given \pref{lem:loc-min-inequality}]
    Fix \(\eta = \frac{\prod_{\ell=0}^{k}\beta_\ell}{4((k+2)!)}\). Let \(\tilde{f}\) be the output of \pref{alg:cosystolic-fixing} for some function \(f\) and \(\eta\). If \(wt(\coboundary f) \geq \frac{\prod_{\ell=0}^{k}\beta_\ell}{4(k+2)!} -e\lambda\) there is nothing to prove, so we assume that \(wt(\coboundary f) < \frac{\prod_{\ell=0}^{k}\beta_\ell}{4(k+2)!} -e\lambda\). Then \(\coboundary \tilde{f} \in Z^{k+1}(X,\Gamma)\) is an \(\eta\)-locally minimal function so that \(wt(\coboundary \tilde{f}) \leq wt(\coboundary f)\). Hence by \pref{lem:loc-min-inequality} (applied with \(k+1\) instead of \(k\)), \(\coboundary \tilde{f} = 0\) and \(\tilde{f}\) is a cosystol. By \pref{claim:distance-is-dominated-by-coboundary-and-beta}, \(\eta \dist(f,\tilde{f}) \leq wt(\coboundary f)\), and we are done.
\end{proof}

\begin{proof}[Proof of \pref{prop:heavy-cosystols}, given \pref{lem:loc-min-inequality}]
    For every \(r \in X(j)\) and \(h \in C^{k-j-1}(X_r,\Gamma)\), we define \(h^{\uparrow}:X(k) \to \Gamma\) by
    \[h^{\uparrow}(s) = \begin{cases}
        h(p) & s=r\circ p \\
        0 & r \not \subseteq s.
    \end{cases}.\]
    It is easy to see that \(g_r + \coboundary h = (g+ \coboundary h^{\uparrow})_r\).
    
    Now let \(0 \ne g \in Z^{k}(X,\Gamma) \setminus B^{k}(X,\Gamma)\) be minimal among all \(Z^{k}(X,\Gamma) \setminus B^{k}(X,\Gamma)\).  
    By the above, \(g\) is also \(0\)-locally minimal (since otherwise we could have found some non-zero coboundary \(\coboundary h^{\uparrow}\) to add to \(g\) and decrease its weight). Thus \(wt(g) \geq \frac{\prod_{\ell=0}^{k-1}\beta_\ell}{(k+1)!}- e \lambda\) as required.

    We remark that the case where \(\Gamma\) is non-abelian and \(k=1\) is similar. Given \(g \in Z^{1}(X,\Gamma) \setminus B^1(X,\Gamma)\) that is non-zero and has minimal weight over all such functions. First we establish that it is locally minimal. Indeed, assume towards contradiction that there is some vertex \(v \in X(0)\) and \(\gamma \in \Gamma \) so that \(wt(g_v) < wt(\gamma g_v)\). Then the function 
    \[g'(xy) = \begin{cases}
        \gamma g(xy) & x=v \\
        g(xy)\gamma^{-1} & y=v \\
        g(xy) & otherwise
    \end{cases}.\]
    is also a cosystol. Taking some triangle \(vuw \in X(2)\) that contains \(v\), the value of \[\coboundary g'(vuw) = \gamma \coboundary g(vuw) \gamma^{-1}=e\] (the identity in \(\Gamma\)). For any triangle \(uwx\) that doesn't contain \(v\) we have that \(\coboundary g'(uwx) = \coboundary g(uwx)=e\). On the other hand, \(wt(g')<wt(g)\) so \(g'\) is trivial, which implies that \(g = \coboundary h\) where \(h(v)=\gamma\) and \(h(u)=e\).  A contradiction to the fact that \(g \notin B^1(X,\Gamma)\).
\end{proof}

The remainder of this section is devoted to proving \pref{lem:loc-min-inequality}. For this we need to define averaging operators that play a crucial role in the theory behind local-spectral expanders. We will only define what we need so for a more thorough exposition see e.g. \cite{DiksteinDFH2018}. Let \(\ell_2(X(j))\) be the Hilbert space of all functions \(f:X(j) \to \RR\) where the inner product is \(\iprod{f,g} = \Ex[r \in X(j)]{f(r)g(r)}\). Let \(D_k: \ell_2(X(k)) \to \ell_2(X(k-1))\) be the following operator
\[D_kf(s) = \Ex[t \supseteq s]{f(t)}.\]
This operator's adjoint is \(U_{k-1}: \ell_2(X(k-1)) \to \ell_2(X(k))\) that is defined by
\[U_{k-1}f(t) = \Ex[s \subseteq t]{f(s)}.\]
As a shorthand we write \(D_k^\ell = D_{k-\ell+1} D_{k-\ell+2} ... D_k\) for \(\ell \geq 1\) (and the same for \(U\)). For \(\ell=0\) \(D_k^0=U_k^0=Id\). We record that \(D_k^{\ell}f\) is a function whose domain is \(X(k-\ell)\), and that \(U_k^\ell f\) is a function whose domain is \(X(k+\ell)\).

Let \(j \leq k < d\). The operator \(N_{k\to j}:\ell_2(X(k)) \to \ell_2(X(j))\) is defined by 
\[N_{k\to j} f(r) = \Ex[t \in X(k+1), t \supseteq r]{\Ex[s \subseteq t, r\not \subseteq s]{f(s)}}.\]
Let us spell out this expression. We average over \(f(s)\) where \(s\) is chosen according to the following rule. We first sample some \(t \supseteq r\) in \(X(k+1)\), and then we sample \(s \subseteq t\) given that it does not contain \(r\).

When \(j,k\) is clear from the context we simply write \(D,U,N\).

The following is an operator norm inequality that is similar to \cite{DiksteinDFH2018}, but for the one-sided case. We prove it in the end of this section.
\begin{claim} \label{claim:psd-inequality-for-hdx}
    Let \(X\) be a \(\lambda\)-one-sided local spectral expander. Then \(U_j^{k-j}N_{k \to j} \preceq U_{j-1}^{k-j+1}D_k^{k-j+1} + \lambda Id\) for every \(j\leq k\).
\end{claim}
Here \(A \preceq B\) for self adjoint operators \(A,B\) means that \(B-A\) is positive semi-definite, that is, \(\iprod{(B-A)h,h} \geq 0\) for every function \(h\) in the domain of \(A,B\).

\begin{proof}[Proof of \pref{lem:loc-min-inequality}]
    Let \(h = \one_{g \ne 0}\). We will prove that \(wt(g)=\ex{h} \geq \frac{\prod_{\ell=0}^{k-1}\beta_\ell}{(k+1)!} - e(\eta+\lambda)\). We do this by showing that
    \begin{enumerate}
        \item \(\snorm{D_k h} \geq \frac{1}{k+1} \snorm h - \lambda \snorm{h}.\)
        \item For \(0 \leq j < k\), \( \snorm{D^{k-j+1}_k h} \geq \frac{\beta_{k-j-1}}{j+1}\cdot \snorm{D^{k-j}_k h} - \left (\frac{\beta_{k-j-1} \eta}{j+1} + \lambda \right ) \snorm{h}.\)
    \end{enumerate}
We note that \(D^{k+1}h\) is a constant (as \(\lambda\)-local spectral expansion says in particular that the complex is connected) - the average of \(h\) on all faces. Hence \(\snorm{D^{k+1}h} = \ex{h}^2\). By iteratively applying these inequalities we get that

\begin{align*}
    \ex{h}^2 &= \snorm{D^{k+1}h} \\
    &\geq \beta_{k-1} \snorm{D^k h} - \left (\beta_{k-1} \eta + \lambda \right ) \snorm{h} \\
    &\geq \frac{\beta_{k-1} \beta_{k-2}}{2} \snorm{D^{k-1} h} - \beta_{k-1} \left(\frac{\beta_{k-2} \eta}{2} + \lambda \right ) \snorm{h}  - \left (\beta_{k-1} \eta + \lambda \right ) \snorm{h} \\
    & ... \\
    &\geq \snorm{h} \cdot \left (\frac{\prod_{\ell=0}^{k-1}\beta_\ell}{(k+1)!} - \eta \sum_{j=0}^{k-1}\frac{\beta_j}{(k-j+1)!} - \lambda \left (1+\sum_{j=0}^{k-1} \frac{\beta_j}{(k-j+1)!} \right ) \right ).
\end{align*}
By assuming \(\beta_j \leq 1\), we upper bound \(\sum_{j=0}^{k-1} \frac{\beta_j}{(k-j+1)!} \leq \sum_{j=0}^\infty \frac{1}{j!} = e\), and get \(\ex{h}^2 \geq \snorm{h} \cdot \frac{\prod_{\ell=0}^{k-1}\beta_\ell}{(k+1)!} - e(\eta + \lambda)\).
As \(\snorm{h}=\ex{h}\) the lemma follows.

\bigskip

Let us begin with the first item. we call $s\in X(k)$ \emph{active} if $h(s)=1$. By assumption, \(g \in Z^{k}(X,\Gamma)\), i.e.
\[\coboundary g(t) = \sum_{i=0}^{k+1} (-1)^i g(t_i) = 0.\]
Thus if \(t \in X(k+1)\) contains an active \(s=t_{i_1}\), then it must also contain a second active \(s'=t_{i_2}\)\footnote{in the case where \(\Gamma\) is non-abelian and \(g =\coboundary f \in C^2(X,\Gamma)\),  even though \(\coboundary g(abcd)\) is not defined, one still observes that  \(\coboundary f(abc) = \coboundary f(acd) = \coboundary f(abd) = e\) implies that \(\coboundary f(bcd)=0\) so the same conclusion holds.}. This implies that \(N_{k \to k} h(s) \geq \frac{1}{k+1}h(s)\), and so
\[ \iprod{h,N_{k \to k} h} = \E_t [h(t) N_{k \to k} h(t)] \geq \frac 1 {k+1}\snorm h.\]

By \pref{claim:psd-inequality-for-hdx} \(N^{k\to k} \preceq UD + \lambda Id\), so
\[ \frac 1 {k+1}\snorm h \leq \iprod{N_{k \to k}h,h} \leq \iprod{ UD h,h} +\lambda\snorm h = \snorm{Dh} + \lambda\snorm h\]
so the first item is proven.

Next, we will prove the second item. As before, we will show that
\begin{equation}\label{eq:step1}
    \iprod{ U^{k-j} N_{k \to j} h, h} 
    \geq  \frac{\beta_{k-j-1}}{j+1} \cdot (\snorm {D^{k-j} h} - \eta \snorm h).
\end{equation}
Then we rely on \pref{claim:psd-inequality-for-hdx}
\begin{equation}\label{eq:algebra}
    \snorm{D^{k-j+1} h} \geq \iprod{ U^{k-j} N_{k \to j} h, h} - \lambda \snorm{h}.
\end{equation}
Combining these inequalities completes the proof.

We now state the following claim, which is proven using the coboundary expansion of $X_r$ where \(r\) is a \(j\)-face. 
\begin{lemma}[Key lemma]\label{lem:key} 
    Let $r\in X(j)$. Then
    \[ N_{k \to j} h(r) \geq \frac{\beta_{k-j-1}}{j+1}(D^{k-j}h(r) - \eta).\]
\end{lemma}
From this pointwise inequality, \eqref{eq:step1} follows easily:
\begin{equation}\label{eq:cbdry-bound1}
\begin{split}
    \iprod{  U^{k-j} N_{k \to j} h, h} = \iprod{N_{k \to j} h, D^{k-j} h} &\geq
    \Ex[r]{D^{k-j}h(r) \cdot \frac{\beta_{k-j-1}}{j+1}\cdot  (D^{k-j}h(r) -\eta)}
    \\
    &=  \frac {\beta_{k-j-1}} {j+1} \cdot (\snorm {D^{k-j} h} - \eta \snorm{h})
\end{split}
\end{equation}
\end{proof}

We will prove \pref{lem:key} under the assumption that \(\Gamma\) is abelian since additive notation is more convenient. For non-abelian groups, see \pref{rem:non-abelian-case}.
\begin{proof}[Proof of \pref{lem:key}]
First, let us understand the meaning of the inequality in \pref{lem:key}. Recall that \(N_{k \to j}h(r)\) is an average of \(h(s)\) over faces \(s\in X(k)\) so that \(r,s \subseteq t\) for some \(t \in X(k+1)\) and \(r \not \subseteq s\). As \(h\) is an indicator function this is the same as writing
\[N_{k \to j}h(r) = \Prob[t,s]{h(s)=1},\]
where \(t,s\) are as above.
On the other side of the inequality there is \(D^{k-j} h(r) = \Prob[s \supseteq r]{h(s) = 1}\).
Hence, we need to show that if there are many active faces that contain \(r\), there must also be many active faces that ``complete'' \(r\) to a \((k+1)\)-face.

We first note that 
\begin{equation} \label{eq:propagating-prob-1}
    N_{k \to j}h(r)=\Prob[t,s]{h(s)=1} \geq \frac{1}{j+1} \Prob[t]{\exists s \subseteq t \; h(s)=1 \ve r \not \subseteq s},
\end{equation}
so we shall actually lower bound \(\Prob[t]{\exists s \subseteq t \; h(s)=1 \ve r \not \subseteq s}\).

As \(g\in Z^{k}(X,\Gamma)\), for every \(t=r \circ p \in X(k+1)\)
\begin{align} \label{eq:coboundary-from-restriction}
    0 = \coboundary g(r \circ p) = \sum_{i=0}^{j}(-1)^i g(r_i \circ p) + (-1)^j \sum_{i=0}^{k-j}(-1)^i g(r \circ p_i).
\end{align}
And in particular 
\begin{align} \label{eq:iff-coboundary}
    \sum_{i=0}^{k-j}(-1)^i g(r \circ p_i) \ne 0 \iff \sum_{i=0}^{j}(-1)^i g(r_i \circ p) \ne 0.
\end{align}
Recall that the restriction of \(g\) is \(g_r:X_r(k-j-1) \to \Gamma\), defined by \(g_r(p)=g(r\circ p)\). As we can see, \(\coboundary g_r(p)\) is the left-hand side of \pref{eq:iff-coboundary}.
Thus
\begin{equation} \label{eq:inequality-to-coboundary}
     \Prob[t]{\exists s \subseteq t \; h(s)=1 \ve r \not \subseteq s} \geq  \Prob[t=r\circ p]{\sum_{i=0}^{k-j}(-1)^i g(r \circ p_i) \ne 0} = \Prob[p \in X_r(k-j)]{\coboundary g_r(p) \ne 0}.
\end{equation}
By assumption \(X_r\) is a \(\beta_{k-j-1}\)-coboundary expander, this is at least \(\beta_{k-j-1} \cdot \dist(g_r, B^{k-j-1}(X_r,\Gamma))\).

To conclude we need to show that 
\begin{equation}
    \dist(g_r, B^{k-j-1}(X_r,\Gamma)) \geq \Prob[s \supseteq r]{g(s) \ne 0} - \eta.
\end{equation}
But 
\begin{equation} \label{eq:dist-to-coboundary-and-local-minimality}
\dist(g_r, B^{k-j-1}(X_r,\Gamma)) = \min_{f \in C^{k-j-2}(X_r,\Gamma)} \set{wt(g_r + \coboundary f)} \geq wt(g_r) - \eta.    
\end{equation}

where the inequality follows from \(\eta\)-minimality of \(g\). As \(wt(g_r) = \Prob[s \supseteq r]{h(s) = 1}\) we have proven
\[N_{k \to j}h(r) \geq \frac{\beta_{k-j-1}}{j+1} \dist(g_r, B^{k-j-1}(X_r,\Gamma)) \geq \frac{\beta_{k-j-1}}{j+1} \left ( \Prob[s \supseteq r]{h(s) = 1} - \eta \right ).\]
\end{proof}

\begin{remark}[The non-abelian case] \label{rem:non-abelian-case}
The first place where we need to accommodate for the non-commutativity is in the derivation of \eqref{eq:inequality-to-coboundary}. Let us understand how to substitute \eqref{eq:coboundary-from-restriction} which implies \eqref{eq:iff-coboundary}, for non-abelian groups.

 If for example, if \(r \in X(0)\) and \(g \in Z^{1}(X,\Gamma)\), and \(ruw \in X(2)\) we can write
\[e=\coboundary g(ruw)=g(ru)g(uw)g(wr)\]
instead of \eqref{eq:coboundary-from-restriction}.
This implies that
\begin{equation}\label{eq:urw}
g(uw)=g(ur)g(rw) 
\end{equation}
or 
\[g(rw)g(uw)g(wr)=g(rw)\cdot(g(ur)g(rw))\cdot g(wr)=g(rw)g(ru)^{-1} = g_r(w)g_r(u)^{-1}=\coboundary g_r(wu)\]
where in the first equality we plugged in the first part of \eqref{eq:urw} and in the second to last equality we plugged in the second part of \eqref{eq:urw}. Since the left hand side is a conjugation of $g(uw)$, we deduce that \(g(uw) \ne e \iff \coboundary g_r(uw) \ne e\). This is the same conclusion as we get in \eqref{eq:coboundary-from-restriction}. The case where \(r \in X(1)\) is similar.

If \(k=2\) we cannot even define \eqref{eq:coboundary-from-restriction} since the coboundary map is not defined. Still, let us see that a similar conclusion to \eqref{eq:iff-coboundary} holds. Let \(g = \coboundary f \in C^2(X,\Gamma)\). Let \(r=ab \in X(1)\) and \(t=abcd \in X(3)\). Denote by \(\gamma = f(ab) f(bc) f(cd)\). Then
\begin{align*}
        \gamma^{-1} \coboundary g_r(cd) \gamma &= \gamma^{-1} g(rc) g(rd)^{-1} \gamma \\
        &= \gamma^{-1} g(abc)g(adb) \gamma \\
        &= f(dc)\cdot \cancel{(f(cb)f(ba)f(ab)f(bc))} \cdot f(ca)f(ad)f(db)\cdot \cancel{(f(ba)f(ab))}\cdot f(bc)f(cd) \\
        &=(f(dc)f(ca)f(ad))(f(db)f(bc)f(cd)) \\
        &= \coboundary f(dca) \coboundary f(dbc)\\
        &= g(dca)g(dcb)^{-1}.
\end{align*}
In particular, we deduce that \(g(dca)g(dcb)^{-1} \ne e \iff \coboundary g_r(cd) \ne e\), and \eqref{eq:inequality-to-coboundary} now becomes
\begin{equation} \label{eq:inequality-to-coboundary-non-abelian}
     \Prob[t]{\exists s \subseteq t \; h(s)=1 \ve r \not \subseteq s} \geq  \Prob[t=r\circ cd]{ g(dca)g(dcb)^{-1} \ne 0} = \Prob[cd \in X_r(1)]{\coboundary g_r(cd) \ne 0}.
\end{equation}
Similarly, when \(r=a \in X(0)\) and \(t=abcd\) we observe similarly that
\begin{align*}
  f(ab) g(bcd) f(ba) &= f(ab) (f(bc)f(cd)f(db)) f(ba) \\
  &= f(ab)f(bc) \cdot (f(ca) f(ac)) \cdot f(cd) \cdot (f(da) f(ad)) \cdot f(db) f(ba) \\
  &= \coboundary f(abc) \coboundary f(acd) \coboundary f(adb) \\
  &= g_a(bc) g_a(cd) g_a(db) \\
  &= \coboundary g_a(bcd),
\end{align*}

and in particular    
\begin{equation} \label{eq:inequality-to-coboundary-non-abelian-x-zero}
     \Prob[t]{\exists s \subseteq t \; h(s)=1 \ve r \not \subseteq s} \geq  \Prob[t=a\circ bcd]{ g(bcd) \ne e} = \Prob[bcd \in X_a(2)]{\coboundary g_a(bcd) \ne e}.
\end{equation}

\medskip
The second equality we need to modify is \eqref{eq:dist-to-coboundary-and-local-minimality}. For example, take an \(\eta\)-locally minimal \(g \in C^2(X,\Gamma)\), a vertex \(r \in X(0)\), and \(\coboundary h \in B^{1}(X_r,\Gamma)\) that is a closest coboundary to \(g_r \in C^1(X_r,\Gamma)\). Then 
\[\dist(g_r,\coboundary h) = \prob{g_r(vu)\ne h(v)h(u)^{-1}} = wt(g_r^h) \geq wt(g_r)-\eta.\]
The case where \(r \in X(1)\) is similar.
\end{remark}

\begin{proof}[Proof of \pref{claim:psd-inequality-for-hdx}]
    We begin by showing the following inequality for operators on \(\ell_2(X(k-1))\) on a \(d\)-dimensional simplicial complex for \(d \geq k\):
    \begin{equation} \label{eq:swap-proposition}
        U^{k-1}S_{k-1} \preceq U^{k}D^{k}_{k-1}+  \lambda Id.
    \end{equation}
    where \(S_{k-1}:\ell_2(X(k-1))\to \ell_2(X(0))\) is defined by \(S_{k-1}f(v) = \Ex[s \in X(k-1), s \dunion \set{v} \in X(k)]{f(s)}\).

    Note that \(U^{k}D^{k}_{k-1}\) is the operator that averages a function on all \(k-1\) faces (independently of the starting face), so it sends the constant function to itself, and it sends everything perpendicular to the constant function to \(0\). The operator \(U^{k-1}S_{k-1}\) also sends the constant to itself, and the functions perpendicular to the constant functions are an \(U^{k-1}S_{k-1}\)-invariant subspace. 
    
    Hence, to show \eqref{eq:swap-proposition} need to show that the second largest eigenvalue of \(U^{k-1}S_{k-1}\) is less or equal to \(\lambda\). Recall that for every two linear operators \(Y:A\to B\) and \(Z:B\to A\) the eigenvalues of \(YZ:B\to B\) are equal to the eigenvalues of \(ZY:A\to A\) (up to the multiplicity of the \(0\) eigenvalue). Therefore, applying this to $S_{k-1}:\ell_2(X(k-1))\to \ell_2(X(0))$ and $U^{k-1}_{0}:\ell_2(X(0))\to \ell_2(X(k-1))$,
   \[ \lambda(U^{k-1}_0 S_{k-1}) = \lambda(S_{k-1}U^{k-1}_0)
    \] 
    and we will easily bound the right-hand side. Here the notation \(\lambda(B)\) is the second largest eigenvalue of the operator inside the parentheses.
    The operator $S_{k-1}U^{k-1}_0$ is nothing but the random walk on the $1$-skeleton of the graph: indeed start with a vertex \(v \in X(0)\), go up to \(t \supset v, t\in X(k-1)\) and then traverse to \(u\) so that \(u \dunion s \in X(k)\). This is the same as choosing \(u \sim v\) a random neighbor of $v$ (and $s\dunion \set u$ a random $k$ face containing this edge). By spectral expansion of the \(1\)-skeleton of \(X\), \(\lambda(S_{k-1}U^{k-1}) \leq \lambda\) and \eqref{eq:swap-proposition} is proven. 
 
    Now let us prove the claim. Fix some $f\in \ell_2(X(k))$. We need to show that 
    \begin{align*}
        \iprod{f,U^{k-j}N_{k \to j} f} \leq \iprod {f,U^{k-j+1}D^{k-j+1} f} + \lambda \iprod{f,f}
     \end{align*}
   The lefthand side is equal to \(\Ex[s]{ f(s) \cdot (U^{k-j}N_{k \to j} f)(s)} = \Ex[s,s'] {f(s)f(s')}\) where $s'$ is chosen according to the distribution of \(U^{k-j}N_{k \to j}(s)\). 
   
   We localize this expectation using \(p = r \cap s\). Namely, the above random process of choosing an pair \(s,s'\) is equivalent to the following:
    \begin{enumerate}
        \item Choose \(p \in X(j-1)\).
        \item Choose \(v \in X_p(0)\) and set \(r = \set v\cup p\).
        \item Choose \(s' \supseteq r\).        
        \item In \(X_p\) walk from $v$ to \(q \in X_p(k-j)\) so that \(\set{v} \dunion q \in X_p(k-j+1)\) (Namely, apply the swap walk \(S_{k-j}\) in the link of \(p\), which takes a \((k-j)\)-face to a vertex). Set \(s=p \cup q \in X(k)\).
    \end{enumerate}
    Thus we can write 
    \begin{align*}
        \iprod{f,U^{k-j}N_{k \to j} f} &= \Ex[(s,t=s \dunion \set{v},r,p,s')]{f(s) f(s')}  \\
        &= \Ex[p]{ \Ex[(v,q,s')]{f(p\cup q) f(p \cup (s' \setminus p))}}  \\
        &= \Ex[p]{\Ex[(v,q,s')]{f_p(q) f_p(s' \setminus p)}},
    \end{align*}
    where \(f_p:X_p(k-j) \to \R\) is defined by \(f_p(x)=f(p \dunion x)\). For a fixed \(p\), the choice of \(q, s' \setminus p\) is just using the random walk \(U^{\ell}S_{\ell}\) in the link of \(p\), with \(\ell=k-j\). This is equal to
    \begin{align}
        \Ex[p]{\iprod{f_p, (U^{\ell}S_{\ell})_p f_p}} &\overset{\eqref{eq:swap-proposition}}{\leq}
        \Ex[p]{\iprod{f_p, (U^{k-j+1}D^{k-j+1})_p f_p}} + \Ex[p]{\lambda \iprod{f_p, f_p}} \\
        &=\iprod{f,U^{k-j+1}D^{k-j+1} f} + \lambda \iprod{f,f}.
    \end{align}
    The last equality holds since 
    \[\iprod{f,U^{k-j+1}D^{k-j+1} f} = \Ex[p]{\Ex[s\supseteq p]{f(s)}^2} = \Ex[p]{\Ex[s \setminus p \in X_p(j-k)]{f_p(s)}^2} = \Ex[p]{\iprod{f_p, (U^{k-j+1}D^{k-j+1})_p f_p}}.\]
    Hence the claim is proven.
\end{proof}

%% file: partite_complex_thm.tex
\section{Coboundary expansion via color restriction} \label{sec:coboundary-expansion-of-partite-complexes}

In this section we develop a technique for bounding coboundary expansion in partite complexes with large dimension. 

Recall that for a \(d\)-partite complex \(X\), and some \(F \subseteq [d]\) the complex \(X^F = \sett{s \in X}{col(s) \subseteq F}\). We call these complexes color restrictions of \(X\). Fix a group \(\Gamma\) and integer \(k\). We say that a color restriction \(X^F\) is \((k,\beta)\)-\emph{locally coboundary expanding} (with \(\Gamma\) coefficients) if \(H^k(X^F,\Gamma)=0\), \(h^k(X^F,\Gamma)\geq \beta\) and for every face \(s \in X(j)\) so that \(col(s) \cap F = \emptyset\), it holds that \(H^k(X_s^F,\Gamma)=0\) and \(h^{k-|s|}(X_s^F,\Gamma) \geq \beta\) (for \(j \leq k\)). The next theorem states that when a color restriction \(X^F\) is \(\beta\)-locally coboundary expanding for typical $F$, then \(X\) itself is an \(\Omega(\beta^{k+1})\)-coboundary expander.

\begin{theorem} \label{thm:partite-complex}
Let \(k,\ell, d\) be integers so that \(k+2\leq \ell \leq d\) and let \(\beta,p \in (0,1]\). Let \(\Gamma\) be some group (that is abelian if \(k > 1\)). Let \(X\) be a \(d\)-partite simplicial complex so that
\[\Prob[F \in \binom{[d]}{\ell}]{X^F \text{ is a \((k,\beta)\)-locally coboundary expander}} \geq p.\]
Then \(X\) is a coboundary expander with \(h^{k}(X,\Gamma) \geq \frac{p \beta^{k+1}}{e(k+2)!}\). Here \(e \approx 2.71\) is Euler's number.
\end{theorem}
We use this theorem in \pref{sec:geometric-lattices-coboundary-expansion} to show the coboundary expansion of homogeneous geometric lattices. The reason we need this theorem is that lower bounds obtained in other techniques usually depend on the dimension of the complex (e.g. the cone technique in \pref{sec:geometric-lattices-coboundary-expansion}). Color restricting \(X^F\) reduces the dimension to \(|F|\), which allows us to use the dimension dependent techniques on \(X^F\), as long as \(|F|\) is a function of \(k\) and not of $d$. Hence, this theorem allows us to overcome the dependence on dimension.

\begin{remark}
We shall write this section in additive notation, i.e. we assume that \(\Gamma\) is an abelian group. However, the same proof holds for non-abelian groups when \(k=1\) also. See \pref{rem:non-abelian-changes-color-proof} after the proof of \pref{thm:partite-complex}.
\end{remark}

\subsection{Additional notation} \label{sec:additional-notation-partite-complexes}
\paragraph{Averaging out real valued functions}
 Let \(X\) be a \((d+1)\)-partite complex and let \(I \subseteq J \subseteq [d]\). Let \(A \subseteq X(k)\)  be a set with relative size \(\varepsilon\). For a face \(s \in X(k)\) we denote by \(\varepsilon_s = \cProb{t \in X(k)}{A}{t \supseteq s}\). We denote by \(\varepsilon_{J,I} = \cProb{t \in X(k)}{A}{col(t)\cap J = I}\), the probability of \(A\)  given that we sampled \(t \in X(k)\) so that \(col(t) \cap J = I\). We denote by \(\varepsilon_{J,j} = \cProb{t \in X(k)}{A}{|col(t)\cap J| = j}\).

\paragraph{Conditional distances} Let \(f_1, f_2 : \dir{X}(k) \to \Gamma\) be two asymmetric functions. The distance between \(f_1,f_2\) is 
\[\dist(f_1,f_2) = \Prob[t \in \dir{X}(k)]{f_1(t)\ne f_2(t)}.\]
For a fixed set \(F\) of colors in a partite complex \(X\), we denote by 
\[\dist_{\bar F,i}(f_1,f_2) = \cProb{t \in \dir{X}(k)}{f_1(t)\ne f_2(t)}{|col(t)\cap \bar{F}| = i}.\]
That is, the relative hamming distance between \(f_1\) and \(f_2\) on the set of \(k\)-faces so that exactly \(i\)-of the vertices have colors outside of \(F\). When the set is clear from the context we omit \(\bar{F}\) and denote this by \(\dist_i(f_1,f_2)\).

\subsection{Proof of \pref{thm:partite-complex}}
\begin{proof}
Fix \(\Gamma\) and let \(f:\dir{X}(k) \to \Gamma\) be an asymmetric function. Assume that \(wt(\coboundary f) = \prob{\coboundary f \ne 0} = \varepsilon\). We need to find some \(g:\dir{X}(k-1) \to \Gamma\) so that \(\frac{p\beta^{k+1}}{e(k+2)!} \dist(f,\coboundary g) \leq \varepsilon\).
We start by finding a set of colors \(F \in \mathcal{F}\) so that most \((k+1)\)-faces with $F$-colored vertices are satisfied. We recall that \(\varepsilon_{F,j} = \cProb{t \in X(k+1)}{\coboundary f(t) \ne 0}{\abs{col(t)\cap F} = j}\). 
 \begin{claim} \label{claim:good-F}
    There is some \(F \in \mathcal{F}\) of size $k+2$ so that for any \(j=1,...,k+2\), \(\varepsilon_{F,j} \leq (k+2) p^{-1}\varepsilon\).
\end{claim}
The claim follows by standard averaging and will be proven later below. 
Now construct \(g:\dir{X}(k-1) \to \Gamma\) in \((k+1)\) steps as follows. Fix some global order on the vertices \(X(0)=\set{v_0,v_1,...,v_n}\).
    \begin{enumerate}
        \item In the first step we define \(g\) for \((k-1)\)-faces whose colors are contained in \(F\). Note that \(X^F\) has no cohomology, \(h^k(X^F,\Gamma) \geq \beta\) and \(\varepsilon_{F,k+2} = \varepsilon_{F,F} \leq (k+2)p^{-1}\varepsilon\) by \pref{claim:good-F}. Coboundary expansion of $X^F$ implies that there exists a function \(g_0:\dir{X}^F(k-1)\to \Gamma\) so that \(\beta \dist_{0}(f, \coboundary g_0) \leq \varepsilon_{F,F}\) which implies \
        \begin{equation}\label{eq:defg0}
        \frac{p\beta}{(k+2) }\dist_{0}(f, \coboundary g_0) \leq \varepsilon~
        \footnote{Here we recall that \(\dist_{0}(f, \coboundary g_0)\)
        is the distance between \(f\) and \(\coboundary g_0\) over \(k-1\)-faces in \(F\). See \pref{sec:additional-notation-partite-complexes}.}
        \end{equation}
        We set \(g =g_0\) on faces \(s \in \dir{X}(k-1)\) so that \(col(s) \subseteq F\).
        \item Let \(i>1\). In the \(i\)-th step, we define \(g\) on faces \(t\) so that \(|col(t) \cap F|=k-(i-1)\) as follows. Assume that \(g\) was defined for all faces \(t\) so that \(|col(t) \cap F| = k-(i-2)\). 
        Every face $t$ with $i-1$ colors outside $F$ can be viewed as $t = s\circ r$ for $s\in \dir{X}^{[d] \setminus F}(i-2)$ and  \(r \in \dir{X_s}^F(k-i)\).
        
        Fix  \(s\) and let us define \(g(r \circ s)\) for all \(r\). We assume  \(s= (v_{j_0},v_{j_1},...,v_{j_{i-2}})\) for \(j_0<j_1<...<j_{i-2}\), and defining $g$ for other \(s\) follows directly by asymmetry. Let us define \(h_s:\dir{X}_s^F(k-i+1) \to \Gamma\) by
        \[h_s(a)=f(a \circ s) - (-1)^{|a|}\sum_{\ell=0}^{i-2}(-1)^{\ell} g(a \circ s_\ell).\]
        Here we use the fact that \(g\) has been defined for all the \(a \circ s_\ell\).
        Next we find (an arbitrary) \(g_0^s : \dir{X}_s^F(k-i) \to \Gamma\) that minimizes \(\dist(h, \coboundary g_0^s)\). We set \(g(r \circ s) = g_0^s(r)\).
    \end{enumerate}
As alluded to in the overview, the motivation behind defining \(h_s\) for every \(s \in X(i-2)\) this way is because for every element in \(\sett{a \circ s}{a \in \dir{X}_s^F(k-i)}\), we have an equation
\begin{equation*}
    \coboundary g(a \circ s) = f(a \circ s)
    \end{equation*}
which translates to
\begin{equation} \label{eq:equation-that-a-defines}
f(a \circ s) - (-1)^{|a|}\sum_{\ell=0}^{i-2}(-1)^\ell g(a \circ s_\ell) = \sum_{\ell=0}^{k-i}(-1)^\ell g(a_\ell \circ s).    
\end{equation}
We consider the set of equations \pref{eq:equation-that-a-defines} one per $a$.
The left hand side (which is \(h_s(a)\)) has the values we have defined in the \((i-1)\)-th step, and the right hand side are the ``unknowns'', that is, the values of \(g\) we wish to define in the \(i\)-th step. Translating this to the language of coboundaries:
\begin{enumerate}
    \item \(h_s(a)\) is the ``free coefficient'' in every equation.
    \item Assignments to the ``unknowns'', i.e. \(g(a_\ell \circ s)\) are functions \(g_0^s : \dir{X}_s^F(k-i-1) \to \Gamma\), so that \(g(r \circ s) = g_0^s(r)\).
    \item The equation that \(a\) defines \eqref{eq:equation-that-a-defines} is satisfied by a solution \(g_0\) if and only if \(h_s(a) = \coboundary g_0^s(a)\). That is, we want to find some \(g_0^s\) that minimizes \(\dist(h, \coboundary g_0^s)\) in the link of \(s\) (and as discussed before, we will do this by showing that \(\coboundary h_s \approx 0\) and using coboundary expansion).
\end{enumerate}

We need to prove we indeed constructed a function \(g\) so that \(\coboundary g\) is close to \(f\).
\begin{lemma} \label{lem:g-is-close-to-f-partite-k}
\[\frac{p\beta^{k+1}}{e (k+2)!} \dist (f, \coboundary g) \leq \varepsilon.\]
\end{lemma}

\end{proof}

\begin{remark} \label{rem:non-abelian-changes-color-proof}
    When \(\Gamma\) is non-abelian and \(k=1\), we construct \(g:\dir{X}(0) \to \Gamma\) in a similar manner to the abelian case. First, we find the values of \(g\) on \(\dir{X}^F(0)\) using coboundary expansion as in the first step of the abelian case. Next, for every \(v\) whose color isn't in \(f\), we define \(h_v:\dir{X}_v^F(0) \to \Gamma\) by \(h_v(u)=f(vu)g(u)\) (where \(g(u)\) has been previously defined since \(col(u)\in F\)). Then we find \(\gamma \in \Gamma = C^{-1}(X,\Gamma)\) so that \(\dist(h_v, \coboundary \gamma)\) and set \(g(v)=\gamma\).

    The proof in this case is identical to the abelian case, so we won't repeat it.
\end{remark}

\begin{proof}[Proof of \pref{lem:g-is-close-to-f-partite-k}]
Recall that \(\dist_i\) is the hamming distance on \(k\)-faces where \(i\) out of \(k+1\) vertices are not in \(F\). We show that for \(i \leq k\)
\begin{equation} \label{eq:i-distance}
\dist_i(f, \coboundary g) = \cProb{t \in X(k)}{f(t) \ne \coboundary g (t)}{|col(t) \cap F| = i} \leq (k+2) (i!) \beta^{-(i+1)} p^{-1}\varepsilon \sum_{j=0}^{i} \frac{1}{j!},
\end{equation}
and that for \(i=k+1\)
\begin{equation} \label{eq:k+1-distance}
\dist_{k+1}(f, \coboundary g) \leq (k+2)! \beta^{-(k+1)} p^{-1} \varepsilon\sum_{j=1}^{k+1}\frac{1}{j!} \leq \frac{e(k+2)!}{p\beta^{k+1}} \varepsilon.
\end{equation}
Since 
\[\dist(f, \coboundary g) = \Prob[t \in X(k)]{f(t) \ne \coboundary g(t)} = \sum_{i=0}^{k+1} \Prob[t \in X(k)]{|col(t) \cap \bar{F}|=i} \dist_i(f,\coboundary g),\]
it holds that
\(\dist(f,\coboundary g) \leq \max_i \dist_i(f,\coboundary g)\) and the lemma follows. We show \eqref{eq:i-distance} by induction over \(i\). 
When \(i=0\), by the definition of \(g\) in the first step, see \pref{eq:defg0}, it follows that \(\dist_0(f,\coboundary g) \leq (k+2) \beta^{-1}p^{-1} \varepsilon\). Let us assume that \eqref{eq:i-distance} holds for \(i\), and show it for \(i+1\). 

We want to bound the fraction of \(t=r \circ s \in \dir{X}(k)\), where \(|s|=i+1\), \(col(r)\subseteq F\) and \(col(s) \cap F = \emptyset\), so that 
\[f(r \circ s) \ne \coboundary g(r \circ s) = \sum_{j=0}^{k-i-1}(-1)^j g(r_j \circ s) + \sum_{j=k-i}^k (-1)^{j} g(r \circ s_{j-(k-i)}).\]

\begin{claim} \label{claim:can-pass-to-links}
    For \(i \leq  k\), 
    \[f(r\circ s) \ne \coboundary g(r \circ s) \Leftrightarrow \coboundary g_0^s(r) \ne h_s(r),\]
    where \(g_0^s(r) = g(r\circ s)\).
\end{claim}

Thus we wish to bound \(\Prob[s,r]{\coboundary g_0^s(r) \ne h_s(r)}\). As \(h^{k-|s|}(X_s^F,\Gamma) \geq \beta\), it is enough to bound the probability that \(\coboundary h \ne 0\) (up to multiplying by a factor of \(\beta^{-1}\)). That is, we need to show that
\begin{equation} \label{eq:h-is-almost-a-cosystol}
wt(\coboundary h_s) \leq (k+2) (i+1)! \beta^{-(i+1)} p^{-1} \varepsilon\sum_{j=0}^{i+1} \frac{1}{j!}.
\end{equation}
Combining \eqref{eq:h-is-almost-a-cosystol} with \(\beta\)-coboundary expansion (in the link of every \(s\)) we have that 
\[\Prob[s,r]{g_0^s(r) \ne h_s(r)} = \Ex[s]{\Prob[r]{g_0^s(r) \ne h_s(r)}} \leq \beta^{-1} \Ex[s]{wt(\coboundary h_s)} \leq (k+2) (i+1)! \beta^{-(i+2)} p^{-1} \varepsilon\sum_{j=0}^{i+1} \frac{1}{j!},\]
since when \(i+1 \leq k\), \(g_0^s\) was chosen to minimize the distance between \(h_s\) and any coboundary.
Towards this end, we claim the following:
\begin{claim} \label{claim:when-dh-is-not-0}
    Let \(s,r\) so that \(\coboundary h_s(r) \ne 0\). Then 
    \begin{enumerate}
        \item Either \(r\circ s\) has some sub-face \(r \circ s_j\) so that \(f(r\circ s_j) \ne \coboundary g (r \circ s_j)\).
        \item Or \(\coboundary f (r \circ s) \ne 0\).
    \end{enumerate}
\end{claim}

When $s,r$ are chosen at random so that \(|s|=i+1\), \(col(r)\subseteq F\) and \(col(s) \cap F = \emptyset\), the probability that the first item in \pref{claim:when-dh-is-not-0} occurs is bounded by 
\[(i+1) \cdot \cProb{t \in X(k)}{f(t) \ne \coboundary g (t)}{|col(t) \cap \bar{F}| = i} = (i+1) \dist_i(f,\coboundary g),\]
since sampling \(r\circ s\) as above, and then sampling a random \(t=r \circ s_j\) in it, has the same marginal distribution as just sampling \(t \in X(k-1)\) so that \(|col(t) \setminus F| = i\). By the induction hypothesis this is bounded by 

\begin{equation} \label{eq:derive-bound-dist-i}
    (i+1) \cdot (i!) (k+2) \beta^{-(i+1)} p^{-1} \varepsilon \sum_{j=0}^{i} \frac{1}{j!}.
\end{equation}

The probability that the second item in \pref{claim:when-dh-is-not-0} occurs is \( \varepsilon_{F,k-(i-1)}\) which is less or equal to \((k+2) p^{-1} \varepsilon\) by \pref{claim:good-F}.
In conclusion
\begin{equation*}
    \begin{aligned}
        \Prob[s,r]{\coboundary h_s(r)\ne 0} & \overset{\pref{claim:when-dh-is-not-0}}{\leq} (i+1) \Prob[s_i,r]{g(s_i \circ r) \ne f(s_i \circ r)} +(k+2)p^{-1}\varepsilon \\ 
        &\overset{ \eqref{eq:derive-bound-dist-i}}{\leq} (k+2) (i+1)! \beta^{-(i+1)} p^{-1} \varepsilon \sum_{j=0}^{i} \frac{1}{j!} + (i+1)! \frac{1}{(i+1)!}(k+2)p^{-1}\varepsilon \\
        &\leq (k+2) (i+1)! \beta^{-(i+1)}p^{-1} \varepsilon \sum_{j=0}^{i+1} \frac{1}{j!}.
    \end{aligned}
\end{equation*}
Here the last inequality is just simplification.

Proving \eqref{eq:k+1-distance} is similar to the above (assuming for \eqref{eq:i-distance} holds for \(i=k\)). We need to bound the probability that \(f(s) \ne \coboundary g(s)\). Similar to \pref{claim:when-dh-is-not-0} we note that \(f(s) - \coboundary g(s) \ne 0\) implies that for every \(w \in X_s^F(0)\):
\begin{enumerate}
    \item Either there is some \(i\) so that \(f(w \circ s_i) \ne  g(w \circ s_i)\).
    \item Or \(\coboundary f (w \circ s) \ne 0\).
\end{enumerate}
Otherwise take some \(w \in X_s^F(0)\) so that  \(\coboundary f (w\circ s) = 0\) and so that \(f(w \circ s_i) = g(w \circ s_i)\) for all \(i\). We have that
\begin{equation}
\begin{aligned}
    0 &= \coboundary f(w \circ s)\\
    & f(s) + \sum_{i=1}^{k+1}(-1)^i f((w\circ s)_i) = \\
    & f(s) + \sum_{i=1}^{k+1}(-1)^i \coboundary g((w\circ s)_i) = \\
    & f(s) - \coboundary g(s) + \sum_{i=0}^{k+1}\coboundary g((w\circ s)_i) = \\
    & f(s) - \coboundary g(s) + \coboundary \coboundary g(s) = \\
    & f(s) - \coboundary g(s).
\end{aligned}
\end{equation}
Hence
\[\dist_{k+1}(f, \coboundary g) \leq \Ex[s \in \dir{X}(k)]{\Prob[w \in X_s^F(0)]{\coboundary f(w \circ s) \ne 0} + \sum_{i=0}^{k}\Prob[w \in X_s^F(0)]{f(w\circ s_i) \ne \coboundary g(w \circ s_i)}}.\]
As 
\[\Ex[s \in \dir{X}(k)]{\Prob[w \in X_s^F(0)]{\coboundary f(w \circ s) \ne 0}} = \varepsilon_{F,1} \leq (k+2)p^{-1}\varepsilon\]
and
\[\Ex[s \in \dir{X}(k)]{\Prob[w \in X_s^F(0)]{f(w\circ s_i) \ne \coboundary g(w \circ s_i)}} = \dist_{k}(f, \coboundary g).\]
This is less or equal to
\begin{equation*}
    \begin{aligned}
        \leq (k+2)p^{-1} \varepsilon + (k+1)\dist_k(f, \coboundary g) & \\
        \leq & (k+2)\frac{(k+1)!}{(k+1)!}p^{-1} \beta^{-(k+1)} \varepsilon + (k+2)(k+1)(k!) p^{-1} \beta^{-(k+1)} \varepsilon \sum_{j=0}^k \frac{1}{j!} \\
        =  &(k+2)!\beta^{-(k+1)}p^{-1} \varepsilon \sum_{j=0}^{k+1} \frac{1}{j!}.
    \end{aligned}
\end{equation*}
and \eqref{eq:k+1-distance} follows.
\end{proof}

\subsection{Proof of remaining claims}
\begin{proof}[Proof of \pref{claim:good-F}]
Consider the distribution where we sample some \((i,F,I) \in \set{1,2,...,k+2} \times \binom{[d]}{k+2} \times \binom{[d]}{\ell}\), so that \(F\) is uniform, \(i\) is uniform, and \(I \subseteq F\) is uniform given that \(\abs{I}
=i\). Let \(\psi:X(k+1) \to \RR\) be the indicator of the set \(\set{\coboundary f \ne 0}\). Denote by \(\varepsilon = \ex{\psi}\). Then \(\Ex[(i,F,I)]{\varepsilon_{F,I}} = \varepsilon\) (where \(\varepsilon_{F,I}\) is the expectation of \(\psi\) over faces \(t\) so that \(col(t) \cap F = I\)). Let \(\mathcal{F}\) be the set of colors \(F\) so that \(X^F\) is locally coboundary expanding as defined in \pref{thm:partite-complex}. By assumption \(\prob{\mathcal{F}} \geq p\). Thus we have that
\[\Ex[(i,F,I): F \in \mathcal{F}]{\varepsilon_{F,I}} \leq p^{-1} \varepsilon.\]
In particular,
\[\Ex[F \in \mathcal{F}]{\sum_{i=1}^{k+2} \varepsilon_{F,i}} \leq (k+2) p^{-1} \varepsilon,\]
where \(\varepsilon_{F,i} = \Ex[I \subseteq F, |I| = i]{\varepsilon_{F,I}}.\)
We conclude by taking some \(F \in \mathcal{F}\) so that the sum of \(\varepsilon_{F,i}\) is less than the expectation. This is the \(F\) we need.
\end{proof}

\begin{proof}[Proof of \pref{claim:can-pass-to-links}]
Recall that by definition \(g(r' \circ s)=g_0^s(r')\) for some \(g_0^s\) so that \(\coboundary g_0^s(a)=\sum_{\ell=0}^{k-i} (-1)^\ell g_0^s(a_\ell)\) is the closest coboundary to the function \(h_s\). The function \(h_s\) was defined by
\[h_s(a)=f(a \circ s) - (-1)^{|a|}\sum_{\ell=0}^{i-2}(-1)^{\ell} g(a \circ s_\ell).\]
\(\coboundary g_0^s(a) = h(a)\) if and only if 
\[\coboundary g(a \circ s) = \sum_{j=0}^k (-1)^j g((a\circ s)_j) =\]
\[ \coboundary g_0^s(a) + \sum_{j=k-i+1}^k (-1)^j g((a \circ s)_j) \overset{\coboundary g_0^s(a) = h(a)}{=} \]
\[ h_s(a) + \sum_{j=k-i}^k (-1)^j g((a \circ s)_j) = \]
\[f(s \circ r) - \sum_{\ell=0}^{i-1} (-1)^{\ell+(k-i+1)} g(a \circ s_\ell) +  \sum_{j=k-i+1}^k  (-1)^j g((r \circ s)_j) = f(r \circ s).\]
The last equality is due to a change of variables in \(\sum_{\ell=0}^{i-2} (-1)^{\ell+(k-i+1)} g(a \circ s_\ell)\) from \(\ell\) to \(j:=\ell+(k-i+1)\).
\end{proof}

\begin{proof}[Proof of \pref{claim:when-dh-is-not-0}]
Let \(s,r\) be so that \(\abs{s}=i+1\), both \(\coboundary f(r \circ s) = 0\) and for every \(s_j \subseteq s\), it holds that \(\coboundary g(r \circ s_j) = f(r\circ s_j)\).

Observe that,
\[\coboundary h_s(r) = \sum_{j=0}^{k-i}(-1)^j h_s(r_j) = \sum_{j=0}^{k-i}(-1)^{j} f(r_j \circ s) - (-1)^{k-i} \sum_{j=0}^{k-i} \sum_{\ell=0}^{i} (-1)^{j+\ell} g(r_j \circ s_\ell)\]
where the second equality is by the definition of $h_s$.
By \pref{claim:algebraic-relation} this is equal to
\[ \sum_{j=0}^{k-i}(-1)^{j} f(r_j \circ s) - (-1)^{k-i} \sum_{j=0}^{k-i} \sum_{\ell=0}^{i} (-1)^{j+\ell} g(r_j \circ s_\ell) - (-1)^{k-i} \sum_{j=0}^{i}\sum_{\ell=0}^{i-1} (-1)^{(k-i)+1+j+\ell}g(r \circ (s_\ell)_j). \]
Note that in change of variables above we are using the fact that \(r \circ (s_\ell)\) is equal to \((r \circ s)_{\ell+|r|}\)
We change variables in the rightmost sum \(\ell := \ell + (k-i+1)\), so this is equal to
\[ \sum_{j=0}^{k-i}(-1)^{j} f(r_j \circ s) - (-1)^{k-i} \sum_{j=0}^{k-i} \sum_{\ell=0}^{i} (-1)^{j+\ell} g(r_j \circ s_\ell) - (-1)^{k-i} \sum_{j=0}^{i}\sum_{\ell=k-i+1}^{k} (-1)^{j+\ell}g((r \circ s_j)_\ell). \]

Rearranging, we get that this is equal to
\[\sum_{j=0}^{k-i}(-1)^{j} f((r\circ s)_j) - \sum_{j=0}^i (-1)^{j+(k-i)} \coboundary g(r \circ s_j) = \]
\[\sum_{j=0}^{k-i}(-1)^{j} f((r \circ s)_j) + \sum_{j=k-i+1}^k (-1)^{j} \coboundary g((r \circ s)_j).\]
If for all \(r \circ s_j\) we have that \(\coboundary g(r \circ s_j) = f(r \circ s_j)\), then this is equal to
\[\sum_{j=0}^k (-1)^j f((r\circ s)_j) = \coboundary f(r\circ s). \]
Finally, if \(\coboundary f(r \circ s) = 0\) then indeed \(\coboundary h_s (r) =0.\)
\end{proof}

%% file: k-exp-for-spherical-building.tex
\section[k-coboundary expansion of order complexes of geometric lattices]{\(k\)-coboundary expansion of order complexes of geometric lattices} \label{sec:geometric-lattices-coboundary-expansion}

In this section we analyze the coboundary expansion of $k$-chains in a $d$-dimensional spherical building, and prove lower bounds that depend on $k$ but are independent of the ambient dimension $d$. 
Our analysis holds for any $d$-dimensional order complex of a homogeneous geometric lattice (see \pref{sec:complexes}), a setting that generalizes the spherical building.

\begin{theorem} \label{thm:lattice-k-coboundary-expansion}
Let \(k+2<d\). There exists constants \(\beta_k=\exp(-O(k^5 \log k))>0\) so that for every group \(\Gamma\) and every order complex of a homogeneous geometric lattice \(X\), \(X\) is a coboundary expander with constant
\[ h^k(X,\Gamma) \geq \beta_k.\]  
\end{theorem}
We did not try to optimize the constants \(\beta_k\). 

Recall that $h^k(X,\Gamma)$ is defined for abelian groups $\Gamma$ for all $k\geq 1$, and for $k=0,1$ for all groups $\Gamma$. In this section we assume that \(\Gamma\) is abelian. \pref{sec:non-abelian-geometric-lattices} treats the case of non-abelian $\Gamma$ and $k=1$. The case for $k=0$ is straightforward. \pref{claim:geometric-lattice-is-an-edge-expander} establishes that geometric lattices are edge expanders, and \pref{claim:edge-expander-many-sets} asserts that edge expansion implies constant coboundary expansion.

Note that an order complex of a $d-1$-graded lattice is naturally \(d\)-partite. For any  \(F \subseteq [d]\), we defined the restriction of $X$ to the colors of $F$ to be the complex \(X^F = \sett{s \in X}{col(s) \subseteq F}\). A color restriction \(X^F\) is said to be \((k,\beta)\)-\emph{locally coboundary expanding} (with respect to \(\Gamma\)) if \(h^k(X^F,\Gamma)\geq \beta\) and for every face \(s \in X(j)\) so that \(col(s) \cap F = \emptyset\), \(h^{k-|s|}(X_s^F,\Gamma) \geq \beta\) (for \(j \leq k\)). 

Our main effort will be to show that \(X^F\) is a local coboundary expander for many possible sets of colors \(F\). \pref{thm:partite-complex} will then imply a lower bound on the coboundary expansion of \(X\) itself. 
\begin{lemma} \label{lem:colored-spherical-buildings}
Let \(c_k = \frac{k^2+5k+4}{2}\). Let \(d \geq c_k\) and let \(X\) be a $d$-dimensional homogeneous geometric lattice. There are constants \(p_k = \exp(-O(k^5 \log k))\) and \(\beta_k' = \exp(-O(k^2 \log k))\) depending only on $k$ so that for every abelian group \(\Gamma\),
\[\Prob[F \in \binom{[d+1]}{c_k}]{X^F \text{ is a \((k,\beta_k')\)-locally coboundary expander with respect to }\Gamma} \geq p_k.\]
\end{lemma}
The proof of \pref{thm:lattice-k-coboundary-expansion} is direct given this lemma and \pref{thm:partite-complex}. We have it explicitly at the end of this section.

To prove this lemma, we use the cone machinery developed by \cite{Gromov2010, LubotzkyMM2016, KozlovM2019, KaufmanO2021}. We take a detour to define and explain their machinery in the next subsection.

Throughout the proof, we shall assume that \(d\gg k\) (in particular, \(d \geq \frac{k^2+5k+4}{2}\)). The work of \cite{KozlovM2019}  gives constant coboundary expansion when \(d\) is smaller (see \cite[Section 3.3]{KozlovM2019} ).
\subsection{Boundaries and cones}
We fix \(X\) to be a \(d\)-dimensional simplicial complex for \(d \geq k\). We consider \(C_k(X,\mathbb{Z})\). It will be convenient to write members of \(C_k(X,\mathbb{Z})\) as a formal sum, so we prepare some conventions first. We identify \(\dir{X}(k) \subseteq C_k(X,\mathbb{Z})\) where \(t \in \dir{X}(k)\) is identified with the function 
\[ f_t(s) = \begin{cases}
1 & s = \pi(t) \text{ for some permutation \(\pi\) with an even sign}\\
-1 & s = \pi(t) \text{ for some permutation \(\pi\) with an odd sign}\\
0 & \text{otherwise}
\end{cases}.\]
These functions span \(C_k(X,\mathbb{Z})\) \footnote{by choosing an ordering \(\dir{t}\) for every \(t \in X(k)\), the functions \(\sett{\dir{t}}{t\in X(k)}\) are a basis for \(C_k(X,\Z)\).}. Thus we can write every \(f \in C_k(X,\Z)\) as \(f = \sum_{t \in \dir{X}(k)} \alpha_t t\).

The support of a function \(supp(f) \subseteq X(k)\) is the set of all \(t \in X(k)\) so that \(f(\dir{t}) \ne 0\) (for any ordering \(\dir{t}\) of \(t\)). The vertex support \(vs(f) \subseteq X(0)\) is the set of all vertices that are contained in some \(t \in supp(f)\).

\paragraph{The boundary operator} Let \(\boundary_k: C_{k}(X,\mathbb{Z}) \to C_{k-1}(X,\mathbb{Z})\) be the operator defined by
\[\boundary f = \sum_{t \in \dir{X}(k)} \alpha_t \sum_{i=0}^k (-1)^i  t_i.\]
This is the \(k\)-th boundary operator. It is a direct calculation to verify that this is well defined, i.e. that it does not depend on the choice of orientations of the faces in the sum of \(f\).

\paragraph{Restriction to a vertex} Let \(w \in X(0)\) and let \(f \in C_k(X,\mathbb{Z})\) be \(f = \sum_{t \in X(k)} \alpha_t t\). The restriction to \(w\) is the following function \(f_w := \sum_{t \in X(k), t\ni w } \alpha_t t\).

\paragraph{Appending a vertex} Let \(w \in X(0)\) and let \(f \in C_k(X,\mathbb{Z})\) be a function that is supported in the link of \(w\). Namely, \(f = \sum_{t \in X_w(k)} \alpha_t t\). We denote by \(f^{w}\in C_{k+1}(X,\mathbb{Z})\) the function \(f^w = \sum_{t \in X_w(k)}  \alpha_t (w\circ t)\). We note that this too does not depend on the representation of \(f\). We record the following equality that will be useful later.
\begin{equation}\label{eq:star-boundary}
    \boundary_{k+1} (f^w) = f - (\boundary_k f)^w
\end{equation}
For example, if $f = uv$, and $w$ is some vertex, then $f^w = wuv$ and
\begin{align*}
    \boundary f^w &= uv - wu + wv \\
    &= f - ((u-v)^w) 
    = f - (\boundary f)^w.
\end{align*}
\paragraph{Cones} Let \(\ell < d\). An \emph{\(\ell\)-cone} \(\boldsymbol{\psi} = (\psi_i:C_i(X,\mathbb{Z})\to C_{i+1}(X,\mathbb{Z}))_{i=-1}^\ell\) is a sequence of functions so that:
\begin{enumerate}
    \item Every \(\psi_i\) is \(\mathbb{Z}\)-linear, that is \(\psi (a \cdot s + b \cdot s') = a\psi(s) + b\psi(s')\).
    \item The main property: for every \(j \leq \ell\) and every \(s \in X(j)\), it holds that
\[\boundary \psi_j(s) = s - \sum_{i=0}^j (-1)^i \psi_{j-1}(s_i).\]
\end{enumerate}
We illustrate the first levels of cones.
\begin{enumerate}
\item A \((-1)\)-cone is just given by a single vertex \(\psi_{-1}:C_{-1}(X, \mathbb{Z}) \to C_{0}(X,\mathbb{Z}); \; \psi_{-1}(\emptyset) = v \in X(0)\) (up to multiplying by an integer).
\item Let \(v_0 =  \psi_{-1}(\emptyset)\). To extend \(\psi\) to a \(0\)-cone we choose walks \((v_0,v_1,...,v_m=u)\) from \(v_0\) to \(u\) for every \(u \in X(0)\) and set \(\psi(u) = \sum_{i=0}^{m-1} v_i v_{i+1} \). It is direct that \(\boundary \psi(u) = u - v_0 = u- \psi_{-1}(\emptyset)\). We note that it is always possible to add a cycle to \(\psi(u)\) (i.e. any \(R \in C_{2}(X,\mathbb{Z})\) so that \(\boundary R = 0\)).
\item To extend \(\psi\) to a \(1\)-cone, we need to find some \(\psi(uw) \in C_2(X,\mathbb{Z})\) such that \( \boundary \psi(uw) = uw - \psi(w) + \psi(u)\). By the above, this is a cycle containing \(v_0\) and the edge \(uw\), so we need to ``fill'' this cycle with triangles.
\end{enumerate}

When it is clear from context, we omit the index and just write \(\psi(s)\).
We note (and later use) the following. 
\begin{claim} \label{claim:cone-equation-has-no-boundary}
    Let \(\boldsymbol{\psi}\) be any cone and let \(s \in X(j)\). Then
\begin{equation} \label{eq:cone-equation-has-no-boundary}
\boundary (s - \sum_{i=0}^j (-1)^i \psi(s_i) ) = 0.
\end{equation}
\(\qed\)
\end{claim}
The reader may prove this by computing directly.

For \(\ell' < \ell\) and an \(\ell\)-cone \(\boldsymbol{\psi} = (\psi_i:C_i(X,\mathbb{Z})\to C_i(X,\mathbb{Z}))_{i=-1}^\ell\), the partial sequence \(\boldsymbol{\psi}_{\ell'} = (\psi_i:C_i(X,\mathbb{Z})\to C_i(X,\mathbb{Z}))_{i=-1}^{\ell'}\) is an \(\ell'\)-cone.

The radius of a cone is 
\[rad(\boldsymbol{\psi}) := \max_{s \in X} \abs{supp \set{\psi(s)}}.\]
The following theorem connects cones with small radius to coboundary expansion.
\begin{theorem}[\cite{KaufmanO2021} Theorem 3.8, see also \cite{KozlovM2019}, Theorem 2.5] \label{thm:cones}
Let \(X\) be a \(k\) dimensional simplicial complex so that there exists a group \(G\) that acts transitively on its \(k\)-dimensional faces. Assume that the \(\ell+1\)-skeleton has an \(\ell\) cone with radius \(B\). Then \(X\) is a coboundary expander and \(h^{\ell}(X,\Gamma) \geq \frac{1}{B\binom{k+1}{\ell+1}}\) for any abelian group \(\Gamma\).
\end{theorem}

\begin{remark}
In previous works this theorem was proven for \(\Gamma = \mathbb{F}_2\). Following the same steps for arbitrary \(\Gamma\) gives exactly the same statement. We omit it from the paper.
\end{remark}

Let \(X=X_P\) be the order complex of a homogeneous geometric lattice $P$, and let \(F\) any set of \(\ell\) colors. Then \(Aut(P)\) acts transitively on the \(\ell\)-faces of \(X^F\). The same is true for top-level faces of color restrictions of links \(X_s^F\). Hence \pref{lem:colored-spherical-buildings} will follow from \pref{thm:cones} if we show there are enough colors \(F\) so that \(X^F\) and its local views \(X_s^F\) have small radius \(k\)-cones. 

\subsection{Proof of \pref{lem:colored-spherical-buildings}}
Let \(P\) be a homogeneous geometric lattice, and let $X$ be its $d$-dimensional order complex.

Recall that the colors in an order complex of a graded lattice correspond to the ranks of the elements. We will construct cones for certain sets of colors that correspond to ranks that are roughly exponentially increasing. For $i=0,\ldots,k$ we will keep track of the parameters   \(c_i, n_i, D_i\) which we now define inductively.
\begin{itemize}
    \item $c_i$ is the number of colors we need for constructing an $i$-dimensional cone, 
\[c_0 = 2 \ve c_i = c_{i-1} + (i+2) = \frac{i^2+5i+4}{2}.\]
\item $D_i$ upper bounds the radius of the $i$-cones, \[D_0 = 3 \ve D_i=(i+2)(i+1)(D_{i-1} + 1).\]
We record that \(D_i = \exp(O(i^2 \log i))\).
\item $n_i$ upper bounds the size of the vertex support of the $i$-cones,
\[ n_0 = 4 \ve n_i = 2(i+2)-(i+1)^2 +(i+1)n_{i-1}\footnote{One can show by induction that \(n_i \geq i+1\) and in particular that this sequence is positive.}. \]
We record that \(n_i = \exp(O(i \log i))\).
\end{itemize} 
Let \(\ell \geq c_k\). A set of colors \(F = \set{i_1 <i_2<...<i_{\ell}}\) is called \emph{\(k\)-suitable} if the following holds \(i_2\geq 2 i_1\), and for every \(j=0,1,...,k-1\) and \(m=1,2,...,j+2\) it holds that \(i_{c_j + m} \geq n_j \cdot i_{c_j} + \sum_{m'=1}^{m-1} i_{c_j + m'}\). 
While the exact inequalities may seem opaque, intuitively \(F\) is \(k\)-suitable if every \(i_{j+1}\) is sufficiently larger than its previous color \(i_j\). For example, if for every \(j\), \(i_{j+1} > (k+3)n_k i_j\) then \(F\) is \(k\)-suitable.

\begin{proposition}[Key Proposition] \label{prop:small-cones}
Let \(F = \set{i_1 <i_2<...<i_{\ell}}\)  be a set of \(k\)-suitable colors. Let \(X\) be either an order complex of a geometric lattice, or a link of said complex. Then \(X^F\) has a \(k\)-cone of radius \(\leq D_k\).
\end{proposition}

\begin{proof}[Proof of \pref{lem:colored-spherical-buildings} (given \pref{prop:small-cones})]
By \pref{prop:small-cones}, for every \(k\)-suitable set of colors \(F\) that has size \(c_k\), and every \(s \in X(j-1)\) so that \(col(s) \cap F = \emptyset\), it holds that \(X_s^F\) has a \((k-j)\)-cone of radius \(D_{k-j} \leq D_k\). It follows from \pref{thm:cones} that that \(h^{k-j}(X_s^F,\Gamma) \geq \frac{1}{\binom{c_k+1}{k+1} D_k}:= \beta_k'\). Note that \(D_k = \exp(O(k^2 \log k))\) and \(\binom{c_k}{k+1} \leq \binom{k^2}{k}=\exp(O(k \log k))\) so indeed \(\beta'_k = \exp(-O(k^2 \log k))\).

We need to show that the set of suitable colors \(\mathcal{F} \subseteq \binom{d}{c_k}\) is a constant fraction of all colors (that depends on \(k\) only). The intuitive idea is that a set is suitable if each color is in a strip \([\alpha d, \beta d)\) that guarantees they satisfy a sequence of inequalities of the form $i_{j+1} \geq B\cdot i_j$. The probability that this happens for a random set of colors is sufficient for our purpose.
Indeed, let \(F=\set{i_1<i_2<...<i_{c_k}}\) be a set of colors. Let \(B = (k+3)n_k\). If for every \(j=1,2,...,c_k\) it holds that 
\[i_j \in \left [ \frac{d}{(2B)^{c_k+1-j}},\frac{2d}{(2B)^{c_k+1-j}}\right ), \]
then in particular \(i_{j+1} \geq B \cdot i_j\). One can easily verify that this implies that \(F\) is \(k\)-suitable, since \(i_2 \geq B i_1 \geq 2 i_1\) and \(i_{c_j+m} \geq (k+3)n_k \cdot i_{c_j + m-1} \geq n_j i_{c_j} + \sum_{m'=1}^{m-1}i_{c_j + m'}\). On the other hand, there are at least

\[\prod_{j=1}^{c_k} \left ( \frac{2d}{(2B)^{c_k+1-j}} - \frac{d}{(2B)^{c_k+1-j}} \right ) = d^{c_k}\cdot \frac{1}{(2B)^{c_k^2 + c_k -\sum_{j=1}^{c_k}j}}.\]
such sets \(F\). 
As \(\binom{d}{c_k} \leq d^{c_k}\), it holds that \(\Prob[F \in \binom{[d]}{c_k}]{\mathcal{F}} \geq \frac{1}{(2B)^{c_k^2 + c_k -\sum_{j=1}^{c_k}j}} :=p_k\) as required.

Finally, we draw our readers attention to the fact that \(B = \exp(O(k \log k))\) and \(c_k = O(k^2)\) so \(p_k=\exp(-O(k^5 \log k))\).
\end{proof}

It remains to prove \pref{prop:small-cones}. We will prove a stronger statement that is amenable to an inductive proof:
\begin{proposition} \label{prop:small-cones-stronger}
Let \(F = \set{i_1 <i_2<...<i_{\ell}}\)  be a set of \(k\)-suitable colors. Let \(X\) be either an order complex of a geometric lattice, or a link of said complex. Then \(X^F\) has a \(k\)-cone of radius \(\leq D_k\). Denote this cone \(\boldsymbol{\psi}\). It has the following properties:
\begin{enumerate}
    \item For every \(j \leq k\) and \(s \in X(j)\), it holds that \(\abs{vs(\psi_j(s))} \leq n_j\).
    \item For every \(j \leq k\) and \(s \in X(j)\) and every \(u \in vs(\psi_j(s)) \setminus s\), it holds that \(col(u) \leq i_{c_j}\).
\end{enumerate} 
\end{proposition}

The proof of \pref{prop:small-cones-stronger} will use the following properties of our complex.
\begin{claim} \label{claim:properties-we-need-from-spherical-building}
    Let \(X\) be either an order complex of a geometric lattice, or a link in said order complex. Then \(X\) has the following properties.
    \begin{enumerate}
        \item The bipartite graph \((X[i_1], X[i_2])\) has diameter \(2\).
        \item If \(col(v) > col(u)\) and \(v \sim u\) then \(X_v[\set{i \leq col(u)}] \supseteq X_u[\set{i \leq col(u)}]\).
        \item (Bound on join rank) For every \(j=0,1,...,k-1\) and \(m=1,2,...,j+2\) the following holds.
        Let \(D \subseteq X(0)\) be a set of at most \(n_j\) vertices of colors \(\leq i_{c_j}\). Let \(M = \set{v_1,v_2,...,v_{r}}\) be so that \(i_{c_j + 1} \leq col(v_1) < col(v_2) < ... < col(v_r) \leq i_{c_j+m-1}\). 
        Then there exists a vertex \(w \in X[i_{c_j+m}]\) so that the complex induced by \(D \cup M\) is contained in \(X_w\).
        \item The property in item 3 holds in \(X_s\) for every \(s \in X\) so that all colors in \(s\) are all greater than \(i_{c_j+m}\).
\end{enumerate}
\end{claim}

We prove this claim after proving the proposition.

\begin{proof}[Proof of \pref{prop:small-cones-stronger}]  
We construct a cone \(\boldsymbol{\psi}\) inductively.
We note that as all the \(\psi_j\) are \(\mathbb{Z}\)-linear and \(C_k(X,\mathbb{Z})\) is generated by \(X(k) \subseteq C(X,\mathbb{Z})\), it is enough to define \(\psi_j\) on \(s \in \dir{X}(k)\) so that it respects anti-symmetry (i.e. \(\psi_j(\pi(s))=\sign(\pi) \psi_j(s)\)).

The first two steps of our construction, corresponding to $\ell=-1,0$, are as follows:
    \begin{enumerate}
        \item \(\psi_{-1}(\emptyset) = v_0\), where \(v_0 \in X[i_1]\) is chosen arbitrarily.
        \item For \(u \in X(0)\) we construct \(\psi_{0}(u)\). 
        \begin{itemize}
            \item If \(u=v_0\) then \(\psi_0(v_0) = 0\).
            \item If \(v_0u \in X(1)\) then \(\psi_0(u) = v_0 u\).
            \item If \(u \in X[i_1]\) then by the assumption that \((X[i_1], X[i_2])\) has diameter \(2\), they have a common neighbor \(w \in X[i_2]\). We assign \(\psi_0(u) = v_0 w + wu\).
            \item Finally, for other \(u \in X[i_j]\) first find a neighbor of $u$ \(w_1 \in X[i_1]\). Select some \(w_2 \in X[i_2]\) that is a common neighbour of \(v_0\) and \(w_1\). Finally we set \(\psi_0(u) = v_0 w_2 + w_2 w_1 + w_1 u\).
        \end{itemize}
    \end{enumerate}
    As we can see from the description above, there are many choices for \(\psi\), so we choose arbitrarily. Notice that the first two parts of the cone \(\boldsymbol{\psi}_0 = \set{\psi_{-1},\psi_0}\) have radius \(D_0=3\). The number of vertices in \(\psi(u)\) is at most \(n_0=4\). Finally, the vertices participating in \(\psi(u)\) are either \(u\) or vertices whose colors are \(i_1,i_2\).

\clearpage

\input{spherical-building-lemma-split.tex}

\end{proof}

\begin{proof}[Proof of \pref{claim:properties-we-need-from-spherical-building}]
Here when we write \(u \leq v\) we mean by the order of the lattice, and when we write \(col(u)>col(v)\) we mean the usual order of integers.
    \begin{enumerate}
        \item Let \(u_1,u_2 \in X[i_1]\) the join \(u_1 \vee u_2\) has rank \(\leq 2 i_1\). By properties of the lattice, there exists some \(v \in X[i_2]\) so that \(v \succeq u_1 \vee u_2 \succeq u_1,u_2\). In particular, there is a path \(u_1,v,u_2\) and the diameter is \(2\).
        \item If \(col(v) > col(u)\) and \(v \sim u\) this implies that \(v \succeq u\). In particular, \(w \in X_u[{\set{i \leq col(u)}}]\) if and only if \(w \preceq u \preceq v\) and thus \(w \in X_v[{\set{i \leq col(u)}}]\).
        \item Let \(u\) be the join of all the elements in \(M \cup D\). The rank of \(u\) is at most the sum of the ranks of the elements in \(M \cup D\). That is, \(col(u) \leq n_j \cdot i_{c_j} + \sum_{m'=1}^{m-1} i_{c_j + m'}\). By assumption on the colors \(i_{c_j+m} \geq col(u)\). Therefore, there is some \(w \in X[i_{c_j+m}]\) that contains \(u\) (if $col(u) < i_{c_j+m}$ we can always add atoms to $u$ one by one, each increasing the rank by $1$, until we reach some \(w \succ u\) whose color is $i_{c_j+m}$ as required). In particular the complex induced by \(M \cup D\) is in the link of \(w\).
        \item Let \(v \in s\) be the vertex with the smallest rank. By the second property proven above, it holds that \(X_s[\set{i_1,...,i_{c_j+m}}] = X_v[\set{i_1,...,i_{c_j+m}}]\). Note that $X_v[\set{i_1,...,i_{c_j+m}}]$ is a partially ordered set whose elements are strictly less than \(v\) in the order of the lattice. This is also a geometric lattice (or a link of said lattice), thus the proof for item \(3\) holds there as well.
    \end{enumerate}    
\end{proof}

This following claim shall be used in the proof of \pref{claim:properties-of-step-2}.
\begin{claim} \label{claim:boundary-of-restriction}
Let \(R \in C_k(X,\Z)\) so that \(\boundary R = 0\). Then \(w \notin vs(\boundary R_{w})\) or equivalently \((\boundary R_{w})_{w} = 0\).
\end{claim}
\begin{proof}[Proof of \pref{claim:boundary-of-restriction}]
We write \(R = R_{w} + R' = \sum_{w \in t} \alpha_t t + R'\) and get that
\(0 = \boundary R = \boundary R_{w} + \boundary R'.\) 
In particular
\( (\boundary R_{w})_{w} = (-\boundary R')_{w}.\) 
As \(w \notin vs(R')\) it is also not in \(vs(\boundary R')\), so \((\boundary R')_{w_2} = 0\). In conclusion \((\boundary R_{w})_{w} = 0\).
\end{proof}

\begin{proof}[Proof of \pref{claim:properties-of-step-2}]
    We prove this by strong induction on \(j\). Assume that the claim holds for all \(j' < j\). 
    
    Let us begin with the first item. Our goal is to show that the conditions in the last item of \pref{claim:properties-we-need-from-spherical-building} hold on \((R_{j-1})_{v_j}\). By the induction hypothesis, we note that the vertex support of \((R_{j-1})_{v_j}\) is contained in the union of:
    \begin{enumerate}
        \item The vertices of \(s' = \sett{v \in s}{col(v) \geq col(v_j)} \subseteq s\).
        \item A subset \(M \subseteq \set{v'_1,...,v'_{j-1}}\) which were added to the support in one of the previous steps.
        \item Other vertices \(D\) that came from the support of \(R_0\). There are at most \(n_\ell\) of those, and they all have colors \(\leq i_{c_j}\) by our assumption on the cone in the previous steps.
    \end{enumerate}
    Let \(Q = \sett{t \setminus s'}{t \in supp((R_{j-1})_{v_j})}\).
    The complex \(Q\) is a subcomplex of the complex induced by \(M \cup D\), and is contained in \(X_{v_j}[{\set{i \leq col(v_j)}}]\). Also, we note that by item 2 in \pref{claim:properties-we-need-from-spherical-building}, this is equal to \(X_{s'}[{\set{i \leq col(v_p)}}]\) (since \(v_j\) is the vertex that has the smallest color out of \(s'\)).
    By the last item in \pref{claim:properties-we-need-from-spherical-building}, there indeed exists some \(v'_j \in X_{s'}\) so that the complex induced by \(M \cup D\) is contained in \(X_{v'_j \cup s'}\). In particular, this implies that the support of \(Q\) is contained in the \(X_{v'_j}\). As all vertices \(v \in s'\) that are not \(v_j\) are greater than \(v_j\) (in the lattice order), it also holds that the support of \((R_j)_{v_j}\) is in \(X_{v'_j}\).

    We turn towards the second item. Assume that \(R_j \ne R_{j-1}\). By construction, \(vs(R_p) \subseteq vs(R_{j-1}) \cup \set{v'_j}\). So to show that \(vs(R_p) \subseteq (vs(R_{j-1})\cup \set{v'_j} ) \setminus \set{v_{j}}\) we need to show that \(v_j \notin vs(R_j)\).
    \begin{align*}
    R_j &= R_{j-1} - \boundary T_j\\
    &=  R_{j-1} - \boundary( ((R_{j-1})_{v_j})^{v'_j})\\
    &\stackrel{\eqref{eq:star-boundary}}= R_{j-1} - (R_{j-1})_{v_j} + (\boundary (R_{j-1})_{v_j})^{v'_j}.
    \end{align*}
    In particular, restricting both sides to the support of \(v_j\) we have that
    \[(R_j)_{v_j} = (R_{j-1})_{v_j} - (R_{j-1})_{v_j} + (\boundary (R_{j-1})_{v_j})^{v'_j}_{v_j} = (\boundary (R_{j-1})_{v_j})^{v'_j}_{v_j}.\]
    Note that the right-hand side is equal to \(((\boundary (R_{j-1})_{v_j})_{v_j})^{v'_j}\), i.e. we can first restrict to \(v_j\) and then append \(v'_j\). However, this is \(0\) by \pref{claim:boundary-of-restriction} (and the fact that \(\boundary R_{j-1} = \boundary R_0 - \sum_{r=0}^{j-1} \boundary \boundary T_r =  0\)). So in conclusion we have that \((R_j)_{v_j} = 0\) and \(v_j \notin vs(R_j)\).

    Finally we explain why the last item holds. Indeed, we notice that \(supp(R_{j-1}) \setminus supp(R_{j})\) contains all faces in \(supp((R_{j-1})_{v_j})\).  On the other hand, the faces in \(supp(R_j) \setminus supp(R_{j-1})\) are faces of the form \((s \setminus v_j) \cup \set{v_{j'}}\) for faces \(s \in supp((R_{j-1})_{v_j})\). This is because faces in \(supp(R_j) \setminus supp(R_{j-1})\) can come from \(\partial T_j\). Faces in \(\partial T_j\) that are not of the form \((s \setminus v_j) \cup \set{v_{j'}}\) for faces \(s \in supp((R_{j-1})_{v_j})\), must contain \(v_j\) (since \(T_j = ((R_{j-1})_{v_j})^{v'_j}\) and thus all its faces contain \(v_j\)) but as we saw \(v_j \notin vs(R_j)\) so all these faces are not in the support of \(R_j\). Hence \(\abs{supp(R_j)} \leq \abs{supp(R_{j-1})}\).
\end{proof}

\subsection{Proof of \pref{thm:lattice-k-coboundary-expansion}}
\begin{proof}
\pref{lem:colored-spherical-buildings} says that the homogeneous geometric lattice satisfies the assumptions of \pref{thm:partite-complex}, with \(\beta = \beta'_k = \exp(-O(k^2 \log k))\) and \(p=p_k = \exp(-O(k^5 \log k))\). Thus concluding that \(X\) is a coboundary expander and that 
\[h^k(X,\Gamma) \geq \frac{p_k\beta_k'^{k+1}}{e (k+2)!} = \exp(-O(k^5 \log k)).\]
\end{proof}

%% file: spherical-building-lemma-split.tex
For \(\ell \geq 1\) given \(\boldsymbol{\psi}_{\ell}\) we construct \(\psi_{\ell+1}\), as described in the algorithm appearing in \pref{fig:cone-construction}. 

\begin{figure}[H]
     \centering
     \begin{algorithm}
    Input: \(s \in X(\ell+1)\)

    \begin{enumerate}
        \item Set 
        \begin{equation*}
        \begin{aligned}
          R_0 &= s - \sum_{i=0}^{\ell+1} (-1)^i \psi(s_i),\\
          s_0 &= \sett{v \in s}{col(v) > i_{c_j}},  \\
          \end{aligned}
        \end{equation*}
        \item Order the vertices in \(s_0 = (v_0,v_2,...,v_t)\) so that \(col(v_1) < col(v_2) < ... < col(v_t)\).
        \item (Shifting step) For \(j=0\) to \(t\):
            \begin{enumerate}
                \item If \(col(v_j) \ne i_{c_\ell + j+1}\):
                \begin{enumerate}
                    \item Find a vertex \(v'_j\) so that \(supp ((R_j)_{v_j} )\subseteq X_{v'_j}\) and so that \(col(v'_j) = i_{c_\ell + j+1}\). 
                \item Set \(T_j= ((R_j)_{v_j})^{v'_j}\), and set \(R_{j+1} = R_{j} - \boundary T_j\).
                \end{enumerate}
                \item Else: set \(T_j = 0\) and set \(R_{j+1} = R_j\).
            \end{enumerate}
        \item (Star step) Find some \(u \in X[i_{c_{\ell+1}}]\) so that \(R_{t+1}\) is in the link of \(u\). 
        \item Output \(\psi(s) = R_{t+1}^u + \sum_{r=0}^{t}\boundary T_r\).
    \end{enumerate}
    \end{algorithm}
     \caption{Constructing \(\psi(s)\)}
     \label{fig:cone-construction}
 \end{figure}

Fix \(s \in \dir{X}(\ell+1)\) and let \(R_0 = s - \sum_{i=0}^\ell (-1)^i \psi(s_i)\). 
Before constructing \(\psi\) formally, we describe the main idea.

We will first find a sequence of ``shifted chains'' \(R_0,R_1,...,R_{t+1}\) so that all vertices in the support of \(R_{t+1}\) are of colors \(< i_{c_{\ell+1}}\) and each $R_j$ is $R_{j-1}$ ``shifted'' by $T_{j-1}$. Namely, there is a sequence of \(T_0,T_1,...,T_t \in C_{\ell+1}(X,\mathbb{Z})\) where \(R_{j+1} = R_0 - \sum_{r=0}^j \boundary T_r\). We will explain below how to construct such chains.

Next we will use item three in \pref{claim:properties-we-need-from-spherical-building} to argue that due to the fact that \(R_{t+1}\) is supported only on vertices of colors \(< i_{c_{\ell+1}}\), there exists some \(u \in X[i_{c_{\ell+1}}]\) so that the complex induced by the vertices of \(R_{t+1}\) is contained in \(X_u\).  Finally we will set 
\begin{equation} \label{eq:cone-def}
    \psi(s) := \sum_{r=0}^t T_r + R_{t+1}^u.
\end{equation}
A direct calculation shows that \(\boundary \psi(s) = R_0\):
\begin{align*}
   \boundary \psi(s) &= \boundary(\sum_{r=0}^t T_r + R_{t+1}^u) \\
   & \stackrel{\eqref{eq:star-boundary}}= \sum_{r=0}^t \boundary T_r + R_{t+1} + (\boundary R_{t+1})^u \\
   & = \sum_{r=0}^t \boundary T_r + R_0 - \sum_{r=0}^t \boundary T_r + (\boundary R_0 + \sum_{r=0}^t \boundary \boundary T_r)^u\\
   & = R_0  + (\boundary R_0 + \sum_{r=0}^t \boundary \boundary T_r)^u\\
   & \overset{\boundary \boundary = 0}{=} R_0  + (\boundary R_0)^u = R_0.
\end{align*}
The last equality is due to \pref{claim:cone-equation-has-no-boundary} which states \(\boundary R_0 = 0\).

Let us understand how to perform the shifting step. Note that by the assumption on \(\boldsymbol{\psi}_\ell\) any vertex in \(vs(R_0)\) of color \(>i_{c_\ell}\) must come from \(s\) itself. So let \(v_0,v_1,...,v_t \subseteq s\) be the vertices in \(vs(R_0)\) of colors \(> i_{c_\ell}\) (the vertices of color \(\geq i_{c_{\ell+1}}\) are a subset of these vertices). The sets \(T_j\) we want will have the property that \(vs(R_{j+1}) = vs(R_j - \boundary T_j) = (vs(R_j) \setminus \set{v_j}) \cup \set{v_{j}'}\), where the replacing vertices \(v_0',v_1',...,v_t'\) are of low-colors (\(col(v_j') = i_{c_\ell+j+1}\)).

We construct \(T_j\) as follows. Assume without loss of generality that \(v_0 \prec v_1 \prec... \prec v_t\) according to the order of the lattice. This implies that \(col(v_j) \geq i_{c_\ell+j+1}\). If \(col(v_j) = i_{c_\ell+j+1}\) then by setting \(T_j=0\) we are done. Otherwise, \(col(v_j) > i_{c_\ell+j+1}\). We will find a vertex \(v_j'\) of color \(i_{c_\ell+j+1}\) so that for every face \(s' \in supp(R_j)\) that contains \(v_j\), \(s' \cup \set{v_j'} \in X\). Then we take \(T_j = ((R_{j})_{v_j})^{v_j'}\) (i.e. \(T_j\) takes all faces that contain \(v_j\) in \(R_{j}\) and adds \(v_j'\) to them). We will show below that \(v_j\) is no longer in the support of \(R_{j+1} = R_{j} - \boundary T_j\).

The reason we can find such a \(v_j'\) is the fourth item of \pref{claim:properties-we-need-from-spherical-building}; this item promises us existence of some \(v_j' \prec v_j\) so that all the vertices in \(vs((R_{j})_{v_j})\) of color \(< col(v_j)\) are also neighbors of \(v_j'\). Moreover, as \(v_j' \prec v_j\) it also holds that \(v_j' \prec v_{j+1} \prec... \prec v_t\), this promises us that \emph{all} \(vs((R_{j-1})_{v_j})\) are also neighbors of \(v_j'\). Thus \((R_{j-1})_{v_j}\) is contained in the link of \(v_j'\) and \(T_j\) is well defined as a chain in \(X\).

We summarize this process in the following claim. Its proof, which technically formalizes this discussion, is given below.
\begin{claim} \label{claim:properties-of-step-2}~
    \begin{enumerate}
        \item The shifting step is always possible. That is, there exists \(v'_j\) so that \(supp((R_j)_{v_j}) \subseteq X_{v'_j}\) and so that \(col(v'_j) = i_{c_\ell + j+1}\). Moreover,
        \item If \(R_j \ne R_{j-1}\) then \(vs(R_j) \subseteq (vs(R_{j-1}) \cup \set{v'_j}) \setminus \set{v_j}\) and
        \item \(|supp(R_j)| \leq |supp(R_{j-1})| \leq ... \leq |supp(R_0)|\).
    \end{enumerate}
\end{claim}
Note that the last item will be necessary when we will bound \(\abs{\psi(s)}\).

 As alluded to in the discussion, after obtaining \(R_{t+1}\) we need to show that there exists some \(u \in X[i_{c_{\ell+1}}]\) so that \(R_{t+1} \subseteq X_u\).
 
 Let \(M \subseteq vs(R_{t+1})\) be the vertices of colors \(> i_{c_\ell}\). By the shifting step there is at most one vertex of each color in \(M\). Let \(B \subseteq vs(R_{t+1})\) be the rest of the vertices, all of which are of color \(\leq i_{c_\ell}\). By the last item of \pref{claim:properties-of-step-2} there are at most \(|vs(R_{t+1})| \leq |vs(R_0)|\) vertices in \(B\). This is at most \(n_{\ell+1}\) by induction hypothesis on the cone already constructed. Hence by item 3 in \pref{claim:properties-we-need-from-spherical-building}, there exists some \(u \in X[i_{c_{\ell+1}}]\) so that the complex induced by \(M \cup B = vs(R_{t+1})\) is in the link of \(u\). This shows that \(\psi(s)\) is well defined.

Now that we established that \(\boldsymbol{\psi}\) is a cone, we bound its radius and verify its other properties.
    \paragraph{The radius} Obviously \(|supp(R_{t+1})^u| = |supp (R_{t+1})|\), and by last item in \pref{claim:properties-of-step-2}, this is \( \leq |\supp R_0|\). By the inductive assumption on \(\boldsymbol{\psi}_\ell\),
    \[\abs{supp(R_0)} \leq 1 + \sum_{i=0}^{\ell+1} \abs{supp(\psi(s_i))} \leq (\ell+2)rad(\boldsymbol{\psi_{\ell}}) + 1 \leq (\ell+2)(1+D_\ell).\]
    Moreover, in every iteration of the shifting step we added \(T_j\) to \(\psi(s)\). The faces in \(T_j = ((R_j)_{v_j})^{v'_j}\), the support of this function is again of size at most \((\ell+2)(1+D_\ell)\) (since by the last item of the size of \pref{claim:properties-of-step-2}, the support of \(R_j\) is less or equal to the size of \(R_0\)'s support). There are at most \(\ell+2\) iterations in the shifting step. So the support of \(\sum_{j=0}^{t} T_j\) is at most \((\ell+2)^2 (1+D_\ell)\). In total, for every \(s \in X(\ell+1)\),
    \[ \abs{supp(\psi(s))} = \abs{supp \left( (R_{t+1})^u + \sum_{j=0}^{t}\boundary T_j \right)} \leq (\ell+3)(\ell+2)(1+D_\ell) = D_{\ell+1}.\]
    
    \paragraph{Colors of \(vs(\psi(s)) \setminus s\)} The vertices in \((vs(\psi(s)) \setminus s)\) that come from \(R_0 = s - \sum_{i=0}^{\ell+1} \psi(s_i)\) have colors \(\leq i_{c_\ell}\) by the induction hypothesis on \(\boldsymbol{\psi_\ell}\). The other vertices in \(vs(\psi(s)) \setminus s\) either come from the shifting step, in which case these are the \(v'_j\) that have colors \(\leq i_{c_{\ell+1}+j+1}\), or the vertex of color \(i_{c_{\ell+1}}\) from the star step. To conclude, all vertices of \(vs(\psi(s)) \setminus s\) have colors \(\leq i_{c_{\ell+1}}\).
    
    \paragraph{The size of \(vs(\psi(s))\)} The vertex support of every \(T_j\) consists only of the vertex support of \(R_0\) and the new \(v'_j\) that we added. There are at most \(\ell+2\) vertices \(v'_j\) introduced by the sets \(T_j\). The vertex support of \((R_{t+1})^u\) is \(u\) together with the vertex support of \((R_{t+1})\). Thus \(vs(\psi (s)) \subseteq  vs(R) \cup \set{v'_1,...,v'_m} \cup \set{u}\). The vertex support of \(R_0\) satisfies
    \[|vs(R_0)| = \abs{vs(s - \sum_{i=0}^j \psi(s_i))} \leq |s|+1+\sum_{i=0}^{\ell+1} \left (|vs(\psi(s_i))|-|s_i|-1 \right ).\]
    Here we subtract \(|s_i|+1\) from the sums since we need to count the vertices of \(s\) and \(u_0\) only once. Hence
    \(|vs(R_0)| \leq \ell+3 - (\ell+2)^2 + (\ell+2)n_{\ell} \).
    Plugging this back in we get that
    \[|vs(\psi(s))| \leq \ell+3 - (\ell+2)^2 + (\ell+2)n_{\ell} + (\ell+3) = n_{\ell+1}.\]

%% file: concrete-complexes.tex
\section{Applications to known bounded degree complexes} \label{sec:concrete-complexes}
\subsection{Cosystolic expansion of known complexes}
\subsubsection[LSV complexes]{\cite{LubotzkySV2005b} complexes}
We now put everything together and show that the complexes constructed by Lubotzky, Samuels and Vishne have degree and dimension independent cosystolic expansion (that is, provided that they are sufficient local spectral expanders). Recall the following properties of their construction.
\restatetheorem{thm:lsv-complexes}

\begin{theorem} \label{thm:lsv-two-sided}
For every \(k > 0\) there is some constant \(\beta_k = \exp(-O(k^6 \log k)) >0\) and integer \(q_0\) so that for every prime power \(q> q_0\), integer \(d > k+2\), group \(\Gamma\), and \(X \in \mathcal{X}_{q,d}\) it holds that
\[h^k(X,\Gamma) \geq \beta_k.\]
\end{theorem}

\begin{proof}[Proof of \pref{thm:lsv-two-sided}] 
We wish to apply \pref{thm:cosystolic-expansion-clear} to the complexes of \pref{thm:lsv-complexes}. By \pref{thm:lattice-k-coboundary-expansion} there is a sequence of constants \(\set{\beta_\ell = \exp(-O(\ell^5 \log k))}_{\ell=0}^k\) so that for every \(q,d\) the spherical building \(SL_d(\mathbb{F}_q)\) satisfies \(h^{\ell}(SL_d(\mathbb{F}_q),\Gamma)\geq \beta_\ell\) for every group \(\Gamma\). There is some \(q_0\) so that for every \(q>q_0\), \(X_n\) are also $\lambda$-expanders for \(\lambda = \exp(-O(k^6 \log k))\). Applying \pref{thm:cosystolic-expansion-clear}, we get that \(h^\ell(X_n,\Gamma) \geq \exp(-O(\ell^6 \log \ell))\) for every \(\ell \leq k\), and in particular, this holds for $h^k$ as claimed.
\end{proof}
Observe that our theorem (as well as previous bounds of \cite{KaufmanKL2014, EvraK2016}) holds only for LSV complexes $\mathcal{X}_{q,d}$ with sufficiently large $q>q_0$. It seems reasonable that even for $q=2$ the theorem should hold, but we leave this as an open question.
\subsubsection[KO complexes]{\cite{KaufmanO2021} complexes}
Kaufman and Oppenheim showed that links of their complexes are spectral expanders, and also coboundary expanders for $1$-chains. 

\restatetheorem{thm:ko-complexes}
Applying our \pref{thm:cosystolic-expansion-clear} implies stronger bounds on the cosystolic expansion of $Y_n$:
\begin{theorem} \label{thm:ko-complexes-cosystolic-expansion}
    For every \(d\) there exists some \(\lambda>0\) and some \(\beta' = \frac{\beta^2}{2}(1-O(\lambda))\) so that for every abelian group \(\Gamma\) and every \(Y_n \in \mathcal{Y}_{\lambda,d}\)
    \[h^1(Y_n,\Gamma) \geq \beta'.\]
\end{theorem}
The proof of \pref{thm:ko-complexes-cosystolic-expansion} is just applying \pref{thm:cosystolic-expansion-clear} on every \(Y_n \in \mathcal{Y}_{\lambda,d}\), and is therefore omitted.

\subsection{Topological overlap of LSV complexes}
Let \(X\) be \(d\)-dimensional simplicial complex. \(X\) is also a topological space constructed by taking a \(d\)-simplex for every \(t \in X(d)\) and gluing two simplexes \(t_1,t_2\) over their intersection \(t_1 \cap t_2\). 
\begin{definition}
Let \(X\) be a \(d\)-dimensional simplicial complex and let \(c > 0\). We say that \(X\) has \((c,k)\)-topological overlap if for every continuous map \(f:X \to \R^k\) there exists a point \(p \in \R^k\) so that 
\[\Prob[t \in X(d)]{p \in f(t)} \geq c.\]
\end{definition}
We call such a point \(p\) a \(c\)-heavily covered point (with respect to \(f\)), since if the measure on \(X(d)\) is uniform, this is proportional to the number of \(d\)-faces covering \(p\).
\begin{theorem}[\cite{Gromov2003}] \label{thm:top-overlap}
    Let \(X\) be a simplicial complex so that
    \begin{enumerate}
        \item For every \(\ell \leq k\), \(h^\ell(X,\mathbb{F}_2) \geq \beta\).
        \item For every \(\ell \leq k\) and every \(g \in Z^k(X,\mathbb{F}_2) \setminus B^k(X,\mathbb{F}_2)\), \(wt(g) \geq \nu\).
        \item \(\max_{v \in X(0)} \prob{v} \leq \varepsilon\).
    \end{enumerate}
    Then the \(k\)-skeleton of \(X\) has \((c,k)\)-topological overlap (i.e. for continuous functions to \(R^k\)) for \(c= \frac{\nu \beta^{k+1}}{2(k+1)!} - \varepsilon k^2 \beta^{-(2k+1)}\).
\end{theorem}
The constant \(c\) was estimated by \cite{DotterrerKW2018}.
In addition, it turns out that one can replace \(R^k\) with any \(k\)-dimensional manifold that admits a piecewise-linear triangulation, and this theorem still holds.

The work by \cite{EvraK2016} used this theorem to show that \cite{LubotzkySV2005b} complexes have the topological overlap properties. Namely, they showed that these complexes have
\[h^{k}(X,\beta)  = \Omega(\min \set{\frac{1}{Q}, (d!)^{-O(2^k)} }),\]
where \(d\) is the dimension of the complex, and \(Q\) is the maximal number of edges adjacent to a vertex. They also show that \(g \in Z^k(X,\mathbb{F}_2) \setminus B^k(X,\mathbb{F}_2)\) has weight at least \(\nu = (d!)^{-O(2^k)}\). Plugging this in \pref{thm:top-overlap} gives a topological overlap constant of \(c= Q^{-k} \exp(-O(2^k d \log d))\).

A direct application of our cosystolic expansion bounds, together with \pref{prop:heavy-cosystols}, gives us bounds that are independent of the degree of the vertices, and the ambient dimension of the complex. In addition, it gives a better dependence in \(k\) (exponential instead of doubly exponential).
\begin{corollary}
Let \(\set{X_n}\) be the simplicial complexes in \pref{thm:lsv-two-sided}. Then \(X_n\) have the \((c,k)\)-topological overlap with \(c=\exp(-O(k^7 \log k)) - \varepsilon \cdot \exp(O(k^7 \log k)),\) where \(\varepsilon = \frac{1}{|X_n(0)|}\) (and goes to \(0\) independent of \(k\)).
\end{corollary}
We note that a similar corollary holds for the spherical building and other homogenuous lattices.
\begin{proof}
    The corollary follows from plugging in the parameters of the complexes in \pref{thm:lsv-two-sided}, to \pref{thm:top-overlap}. \pref{thm:lsv-two-sided} gives us \(h^k(X,\mathbb{F}_2) = \exp(-O(k^6 \log k))\) cosystolic expansion. \pref{prop:heavy-cosystols} give us a bound of \(\nu = \exp(-O(k^6 \log k))\) on the weight of all \(g \in Z^{k}(X,\mathbb{F}_2) \setminus B^{k}(X,\mathbb{F}_2)\).
\end{proof}

Our bounds also directly imply that the \(2\)-skeletons of all complexes constructed in \cite{KaufmanO2021} have \(\Omega(1)\)-topological overlap (where previously the bound depended on the maximal degree of a vertex). We omit the direct proof.
\begin{corollary}
    Let \(\set{Y_n'}\) be the two skeletons of complexes in \pref{thm:ko-complexes-cosystolic-expansion} with a sufficiently large number of vertices. Then every \(Y_n\) is a \((c,1)\)-topological overlap for some universal constant \(c>0\). \(\qed\)
\end{corollary}

\subsection{Cover stability}
Dinur and Meshulam studied local testability of covers \cite{DinurM2019}, and showed that covering maps of a simplicial complex \(X\) are locally testable if and only if \(X\) is a cosystolic expander on \(1\)-chains. We briefly describe their result below, and show that our new bounds on cosystolic expansion of \(1\)-chains of \cite{LubotzkySV2005b} and \cite{KaufmanO2021} complexes, show that covering maps to these complexes are locally testable.

\begin{definition}[covering map]
    Let \(X,Y\) be pure simplicial complexes. A covering map is a surjective simplicial map\footnote{that is, for every \(i \leq d\) and every \(s \in Y(i)\), \(\rho(s) \in X(i)\).}  \(\rho: Y(0) \to X(0)\) such that for every \(\tilde{u} \in Y(0)\) that maps to \(\rho(\tilde{u})=u\in X(0)\), it holds that \(\rho|_{Y_{\tilde{u}}(0)}: Y_{\tilde{u}}(0) \to X_u(0)\) is an isomorphism. If there exists such a map \(\rho\) we say that \(Y\) is a cover of \(X\).
\end{definition}
While covering maps are described combinatorially in simplicial complexes, they are a well-known topological notion in general topological spaces. They are classified by the fundamental group of the complex \(X\) \cite{Surowski1984}. This is an interesting example for a non-trivial topological property that is locally testable.

The one-dimensional case, i.e. graph covers, have been useful in construction of expander graphs. Bilu and Linial showed that random covers of Ramanujan graphs are almost Ramanujan \cite{BiluL2006}. A celebrated result by \cite{MarkusSS2015} used these techniques to construct bipartite Ramanujan graphs of every degree. Recently, \cite{Dikstein2023} showed that random covers could also be applied for constructing new local spectral expanders.

\paragraph{Local testability of covering maps} Given a map \(\rho:Y(0) \to X(0)\), we wish to test whether \((Y,\rho)\) is close to a covering map (in Hamming distance), while querying only a few of the values \((u,\rho(u))\).

To describe such a test restrict ourselves to the following family of maps. Fix \(X\) to be some pure \(d\)-dimensional simplicial complex. Let \(S\) be a set, and \(\Gamma \leq Sym(S)\) be a group acting on \(S\). The family of functions we consider are \(M(\Gamma,S)\) (suppressing \(X\) in the notation). These are all \((\rho,Y)\) such that 
\begin{enumerate}
    \item \(Y(0) = X(0) \times S\),
    \item \(\rho(u,s)=u\) and
    \item For every \(uv \in X(1)\) the complex induced by the vertices \(\sett{(u,s),(v,s)}{s \in S} \subseteq Y(0)\) is a perfect matching where \((u,s) \sim (v,\gamma_{uv} . s)\) for some \(\gamma_{uv} \in \Gamma\).
\end{enumerate}  

The distance between two maps is
\[\dist ((Y_1,\rho_1), (Y_2,\rho_2)) = \Prob[uv \in X(1)]{\gamma^{Y_1}_{uv} \ne \gamma^{Y_2}_{uv}},\]
where \(\gamma^{Y_i}_{uv}\) is the member in \(\Gamma\) that describes the bipartite graph induced by \(\sett{(u,s),(v,s)}{s \in S} \subseteq Y_i\). We note that when \(S\) is finite this is proportional to the Hamming distance between \(Y(1),Y(2)\).

While this restriction seems arbitrary, it turns out that for every pair \((Y,\rho: Y(0) \to X(0))\) where \(\rho\) is a covering map, there exists some \(\Gamma,S\) and an identification \(Y(0) \cong X(0) \times S\) such that \(\rho(u,s)=u\) \cite{Surowski1984}. However, being in \(M(\Gamma,S)\) is not sufficient to being a covering map. A map in \(M(\Gamma,S)\) above is a covering map if and only if for every \(uvw\in \dir{X}(2)\), 
\begin{align} \label{eq:cover-local-conditions}
    \gamma_{uv} \gamma_{vw} \gamma_{wu}=e.
\end{align}
We denote by \(M_0(\Gamma,S) \subseteq M(\Gamma,S)\) the family of covering maps \((Y,\rho)\) satisfying \eqref{eq:cover-local-conditions}.

The local condition \eqref{eq:cover-local-conditions} gives rise to a local test. 
\begin{test} Input: \((Y,\rho) \in M(\Gamma,S)\).
\begin{enumerate}
    \item Sample \(uvw \in \dir{X}(2)\).
    \item Accept if \eqref{eq:cover-local-conditions} holds for \(uvw\).
\end{enumerate}
\end{test}
We note that this test could be realized by sampling \(3|S|\) points in \(Y(0)\) and their values.
We denote by 
\(c(Y,\rho) = \Prob[uvw\in X(2)]{\text{test fails}},\)
and for a complex \(X\) we define its \((G,S)\)-\emph{cover-stability} to be
\[C(X,\Gamma,S) = \min_{(Y,\rho) \in M(\Gamma,S) \setminus M_0(\Gamma,S)} \frac{C(Y,\rho)}{\dist( (Y,\rho),M_0(\Gamma,S))}.\]
We say that \(X\) is \(c\)-cover stable, if for all \(\Gamma,S\) it holds that \(C(X,\Gamma,S) \geq c\).

\begin{theorem}[\cite{DinurM2019}]
    Let \(X\) be a \(d\)-dimensional simplicial complex for \(d \geq 2\). Then \(C(X,\Gamma,S) = h^1(X,\Gamma)\).
\end{theorem}
This test has been used as a component in the agreement test by \cite[Lemma 3.26]{GotlibK2022}\footnote{While they don't use this language to define their test, one may verify that their test is equivalent to this one.}. Our results show directly that the complexes of \cite{LubotzkySV2005b} and \cite{KaufmanO2021} are cover-stable.
\begin{corollary} \label{cor:cover-stability}
    There exists some absolute constant \(c>0\), such that 
    \begin{itemize}
        \item Order complexes of homogenuous lattices,
        \item the complexes in \pref{thm:lsv-two-sided},
        \item and the complexes in \pref{thm:ko-complexes-cosystolic-expansion}.
    \end{itemize} 
    are all \(c\)-cover stable. \(\qed\)
\end{corollary}
We point out in particular the meaning of this theorem for the case of the spherical building. The fundamental group of the spherical building is trivial, so it has no non-trivial covers. In other words, any cover is just a bunch of disjoint copies of the original complex. The theorem says that any approximate cover of the spherical building must approximately split into a bunch of disjoint copies. 

%% file: appendixAnew.tex
\section[Non-abelian 1-coboundary expansion in geometric lattices]{Non-abelian \(1\)-coboundary expansion in geometric lattices} \label{sec:non-abelian-geometric-lattices}

In this appendix we give a proof that homogenious geometric lattices have constant \(1\)-coboundary epxnasion that applies to coefficients coming from non-abelian groups. 
\begin{theorem}\label{thm:spherical-building}
Let \(d \geq 3\) be and integer. Let \(X\) be the poset complex for some homogeneous lattice \(P\) of rank \(d\). Let \(\Gamma\) be any group (including non-abelian groups). Then \(H^1(X,\Gamma)=0\) and \[h^{1}(X,\Gamma) \geq \frac{1}{198288}.\]
\end{theorem}
The constant in this proof is by no means tight.

The proof for non-abelian groups closely follows the proof for the abelian case, but we use a modified version of cones instead of the version defined in \pref{sec:geometric-lattices-coboundary-expansion}. This definition of cones for the non-abelian case appeared in \cite[Section 4]{DiksteinD2023swap}, and we give it here again in the next subsection to stay self contained. We closely follow the definitions and notation therein.

\subsection{Non-abelian cones}
Fix \(X\), a simplicial complex and some \(v_0 \in X(0)\). We define two symmetric relations on loops around \(v_0\):
\begin{enumerate}
    \item[(BT)]~ We say that \(P_0 \overset{(BT)}{\sim} P_1\) if \(P_i = Q_0 \circ (u,v,u) \circ Q_1\) and \(P_{1-i} = Q_0 \circ (u) \circ Q_1\) for \(i=0,1\) (i.e. going from \(u\) to \(v\) and then backtracking).
    \item[(TR)]~ We say that \(P_0 \overset{(TR)}{\sim} P_1\) if \(P_{i} = Q_0 \circ (u,v) \circ Q_1\) and \(P_{1-i} = Q_0 \circ (u,w,v) \circ Q_1\) for some triangle \(uvw \in X(2)\) and \(i=0,1\).
\end{enumerate}

Let \(\sim\) be the smallest equivalence relation that contains the above relations (i.e.\ the transitive closure of two relations).

We denote by \(P \sim_1 P'\) if there is a sequence of loops \((P_0=P,P_1,...,P_m=P')\) and \(j \in [m-1]\) such that:
\begin{enumerate}
    \item \(P_j \overset{(TR)}{\sim} P_{j+1}\) and
    \item For every \(j' \ne j\), \(P_{j'} \overset{(BT)}{\sim} P_{j'+1}\).
\end{enumerate}
I.e. we can get from \(P\) to \(P'\) by a sequence of equivalences, where exactly one equivalence is by \((TR)\).

Let \(P = (u_0,u_1,...,u_m)\) be a walk in \(X\). We denote by $P^{-1}$ the walk $(u_m,\ldots,u_1,u_0)$.

\begin{definition}[Non-abelian cone]
    A non-abelian cone is a triple \(C=(v_0,\set{P_u}_{u \in X(0)}, \set{T_{uw}}_{uw \in X(1)})\) such that
\begin{enumerate}
    \item \(v_0 \in X(0)\).
    \item For every \(v_0 \ne u \in X(0)\) \(P_{u}\) is a walk from \(v_0\) to \(u\). For \(u = v_0\), we take \(P_{v_0}\) to be the loop with no edges from \(v_0\).
    \item For every \(uw \in X(1)\), \(T_{uw}\) is a sequence of loops \((P_0,P_1,...,P_m)\) such that:
    \begin{enumerate}
        \item \(P_0 = P_u \circ (u,w) \circ P_w^{-1}\), 
        \item For every \(i=0,1,...,m-1\), \(P_i \sim_1 P_{i+1}\) and
        \item \(P_m\) is equivalent to the trivial loop by a sequence of \((BT)\) relations.
    \end{enumerate} 
    We call \(T_{uw}\) a \emph{contraction}, and we denote $|T_{uw}| = m$.
\end{enumerate}
\end{definition}
The definition of \(T_{uw}\) depends on the direction of the edge \(uw\). We take as a convention that \(T_{wu}\) has the sequence of loops $(P_0^{-1},P_1^{-1},\ldots,P_m^{-1})$, and notice that $P_0^{-1} = (P_u \circ (u,w) \circ P_w^{-1})^{-1} = P_w \circ (w,u) \circ P_u^{-1}$. Thus for each edge it is enough to define one of \(T_{uw},T_{wu}\). The diameter of the cone is
\[diam(C) = \max_{uw \in \dir{X}(1)} \abs{T_{uw}}.\]
Equivalently, this is the maximal number of triangle relations required in the contraction of any \(T_{uw}\).

There is a direct connection between the diameter of the cone to coboundary expansion.
\begin{lemma}[\cite{DiksteinD2023swap}] \label{lem:group-and-cones}
    Let \(X\) be a simplicial complex such that \(Aut(X)\) is transitive on \(k\)-faces. Suppose that there exists a cone \(C\) with diameter \(R\). Then \(X\) is a \(\frac{1}{\binom{k+1}{3}\cdot R}\)-coboundary expander.
\end{lemma}

\subsection[Proof of Theorem A.1]{Proof of \pref{thm:spherical-building}}
\begin{proof}[Proof of \pref{thm:spherical-building}]
    First, We prove the theorem under the assumption that the complex \(X\) is \(d\)-partite for \(d \geq 18\). We modify the arguments for the case \(d< 18\) in the end of the proof. 
    
    Towards applying \pref{thm:color-restriction}, we show that \(X^I\) is a coboundary expander on \(1\)-cochains for a constant fraction of the triples of colors \(I \subseteq \binom{[d]}{3}\). Let \(C = \sett{\set{i_0 < i_1 < i_2} \subseteq [d]}{2i_0 \leq i_1, 3i_1 \leq i_2}\) be the set of colors. We note that \[C \supseteq  \sett{\set{i_0,i_1,i_2}}{i_0 \in [0,\frac{1}{18}d], i_1 \in [\frac{1}{9}d,\frac{2}{9}d], i_2 \in [\frac{2}{3}d,d]}\]
    so in particular \(\frac{|C|}{\binom{d}{3}} \geq \frac{1}{100}\) (here we use the fact that \(d \geq 18\)). By \pref{claim:geometric-lattice-is-an-edge-expander}, the links of vertices in \(X\) are also \(\frac{1}{\sqrt{2}}\)-spectral expanders. Thus the links of vertices are \(\frac{1}{2\sqrt{2}}\)-edge expanders, or equivalently, \(\frac{1}{2\sqrt{2}}\)-coboundary expanders on \(0\)-cochains. Hence by \pref{thm:color-restriction},
    \[h^1(X) \geq \frac{1}{100 e} \cdot \min_{I \in C} (h^1(X^I),\frac{1}{2\sqrt{2}})^2\]
    provided that every \(X^I\) has \(H^1(X^I,\Gamma) = 0\). By showing that for every \(I\), \(X^I\) is a \(\frac{1}{9}\)-coboundary expander we get that \(X\) is a \(\frac{1}{24300}\)-coboundary expander.
    
    Indeed, fix \(I = \set{i_0 < i_1 < i_2} \in C\) and let us construct a non-abelian cone \((v_0,\set{P_u}_{u \in X(0)},\set{T_{uw}}_{uw \in X(1)})\). Fix an arbitrary \(v_0 \in X[i_0]\) as a base vertex. For every \(u \in X\) we take \(P_u\) to be as in \pref{prop:small-cones-stronger}:
    \begin{enumerate}
        \item If \(u \in X[i_0]\), we set \(P_u=(v_0,v_1,u)\) where \(v_1 \in X[i_1]\) is such that \(u \vee v_0 \preceq v_1\) (this is possible since \(rk(u \vee v_0) \leq 2i_0 \leq i_1\)).
        \item Otherwise, let \(v_2 \preceq u\) be any \(v_2\) such that \(rk(v_2) = i_0\). Let \(v_1 \in X[i_1]\) be such that \(v_0 \vee v_2 \preceq v_1\) \(\dim(v_1) = i_1\). We take the path \(P_u=(v_0,v_1,v_2,u)\).
    \end{enumerate}

    Fix any cycle \(P_0 = P_u \circ (u,w) \circ P_w^{-1}\). Let us construct a contraction \(T_{uw}\). Let us assume first that \(u \in X[i_1], w \in X[i_0]\) (we will show how to reduce to this case below). In this case, the cycle is of the form \(P_0 = (v_0,v_1,v_2,u,w,v_3,v_0)\), where \(v_0,v_2,w \in X[i_0]\) and \(v_1,u,v_3 \in X[i_1]\). The elements \(v_0,v_2,w\) are all \(\preceq\) to some other element in the cycle, therefore the element \(v_1 \vee u \vee v_3\) is the maximum of all elements in the cycle. As \(rk(v_1 \vee u \vee v_3) \leq 3i_1 \leq i_2\) there exists some \(y \in X[i_2]\) such that every vertex in the cycle is \(\preceq y\). Hence for every edge \(\set{x,x'}\) in the cycle, \(\set{x,x',y} \in X^I(2)\). Let us now introduce \(y\) into the cycle, by doing backtracking relations from \(P_0\) to \(P_0' = (v_0,y,v_0,v_1,y,v_1,\dots,v_3,y,v_3,v_0)\). Note that any \(P' \sim_1 P_0'\) is also \(P' \sim P_0\). Now we can contract every sub-path of \(P_0'\) of the form \((y,x,x')\) to \((y,x')\). That is, we define 
    \[P_1 = (v_0,y,v_1,y,v_1,v_2,y,v_2,u,\dots,v_3,v_0),\] then \(P_2 = (v_0,y,v_1,y,v_2,y,v_2,u,\dots,v_3,v_0)\) and so on until reaching
    \[P_6 = (v_0,y,v_1,y,\dots,v_3,y,v_0).\]
    Note that \(P_0' \sim_1 P_1\) and \(P_i \sim_1 C_{i+1}\) for all \(i \geq 1\). Finally, \(P_6\) is equivalent to the trivial loop by a sequence of backtracking relations. Thus, we found a contraction \(T_{uw}\) where \(|T_{uw}| = 6\). 

    Now let us see how to reduce the general case where \(u \in X[i_1], w \in X[i_0]\).
    \begin{enumerate}
        \item If \(w \in X[i_0]\) and \(u \in X[i_2]\), then denote by \(v_2 \preceq u\) the other neighbor of \(u\) in the cycle \(P_0\). Let \(u' \in X[i_1]\) be an element such that \(w \vee v_2 \preceq u' \leq u\). By using the two triangles \(\set{u,u',w},\set{v_2,u',u} \in X(2)\), we can replace \(u\) by \(u'\) in the cycle \(P_0\) thus reducing to the case above. More explicitly, we replace the sub-path \((v_2,u,w)\) in \(P_0\), with \((v_2,u',u,w)\) obtaining a cycle \(P_0'\), and then replace \(((v_2,u',u,w)\) with \((v_2,u',w)\) to obtain \(P_0''\) which is a similar cycle only with \(u' \in X[i_1]\) instead of \(u \in X[i_2]\). We used \(2\) triangles in this reduction.
        \item Otherwise \(u \in X[i_2]\) and \(w \in X[i_1]\). In this case, denote by \(v_3 \preceq w\) the other neighbor of \(w\) in the cycle \(P_0\). With the triangle \(\set{v_2,u,w} \in X(2)\) we can replace the sub-path \((u,w,v_3)\) in \(P_0\) with \((u,v_3)\) to obtain \(P_0'\). The cycle \(P_0'\) is a cycle where there is a single \(u \in X[i_2]\) and all others are in \(X[i_0]\) or \(X[i_1]\), so we can continue by doing the same reduction as in the first item. The total number of triangles necessary is \(\leq 3\).
    \end{enumerate}

    Thus we found a non-abelian cone where \(diam(C) \leq 9\). By \pref{lem:group-and-cones} \(H^1(X,\Gamma) = 0\) and \(h^1(X^I) \geq \frac{1}{9}\). The theorem is proven.

    \medskip

    Finally, let us explain why the same proof holds for \(d < 100\). In this case we can just take \(C = \set{\set{i_0=1,i_1=2,i_2=3}}\). Therefore \(\frac{|C|}{\binom{d}{3}} \leq \frac{1}{816}\) because \(d<18\), and replacing the constant \(\frac{1}{100}\) with \(\frac{1}{816}\) gives us a bound of \(h^1(X) \geq \frac{1}{198288}\).
    
    Let us explain why \(h^1(X^I) \geq \frac{1}{9}\) when \(I=\set{1,2,3}\) as well. Following the proof closely, one observes that the same steps hold for the same reasoning, except for the following step: it is no longer clear why in the case where \(u \in X[i_0], w \in X[i_1]\), there exists \(y \in X[i_2]\) such that all vertices in \(P_0=(v_0,v_1,v_2,u,w,v_3,v_0)\), are \(\leq y\). In this case we claim that \(v_1 \vee u \vee v_3\) is a join of three \emph{rank \(1\) elements}. 
    
    Indeed, assume without loss of generality that \(v_0,v_2,w\) are not all identical (if they are all the same, then we can just contract the cycle using backtracking relations).  
    If \(v_0,v_2,w\) are all distinct then \(v_1 = v_0 \vee v_2, u= v_2 \vee w\) and \(v_3 = v_0 \vee w\). Thus \(v_1 \vee u \vee v_3 = v_0 \vee v_2 \vee w\) that has rank \(3\) at most. Otherwise, if (say) \(v_0 = v_2 \ne w\), then \(v_1 = v_0 \vee v^*\) in which case \(v_1 \vee u \vee v_3 = v_0 \vee v^* \vee w\). Hence, the theorem is also proven in this case.
\end{proof}

%% file: coboundary-expansion-never-greater-than-1.tex
\section[Cosystolic expansion of dense complexes is at most 1+o(1)]{Cosystolic expansion of dense complexes is at most $1+o(1)$} \label{app:upper-bounding-cosystolic-expansion}
In this section we give a simple upper bound on cosystolic expansion of dense complexes.
\begin{proposition} \label{prop:coboundary-expansion-upper-bound}
Let \(\varepsilon > 0\). Let \(X\) be a \(d\)-dimensional simplicial complex for \(d\geq k+1\). Assume that for every \(j\), the probability of \(s \in X(j)\) is \(\frac{1}{|X(j)|}\). Let \(\varepsilon = \max (\sqrt{\frac{8|X(k-1)|}{|X(k)|}}, \sqrt{\frac{9}{|X(k+1)|}})\) and assume that \(\varepsilon \leq \frac{1}{2}\). Then \(h^{k}(X,\mathbb{F}_2) \leq 1+8\varepsilon\).
\end{proposition}
For example, this proposition implies the following. If a \(\set{X_n}\) are a family of simplicial complexes whose vertex set is growing to infinity, and so that \(\lim_{n \to \infty} \frac{|X(k-1)|}{|X(k)|} = 0\) then for every \(\varepsilon > 0\), all but finitely many \(X_n\) have that \(h^{k}(X_n,\mathbb{F}_2) \leq 1+ \varepsilon\).

There are many complexes so that these conditions hold. For example, the complete complex, the complete bipartite complex, and the spherical building (when \(d > \frac{k}{2}\)).

\begin{proof}
We sample \(f \in C_k(X,\mathbb{F}_2)\) uniformly at random (i.e. we fix some global order of the vertices. For every \(s \in \dir{X}(k)\) that is ordered according to this global order we sample \(f(s) \in \mathbb{F}_2\), and extend this asymmetrically). We note that every \(s\) was chosen independently.

Fixed some \(g \in C_{k-1}(X,\mathbb{F}_2)\). Let \(X_g(f) = \Prob[s \in X(k)]{f(s) = \coboundary g(s)}\). Obviously, \(X_g(f) = \sum_{s \in X(k)} \prob{s}\one_{f(s) =\coboundary g(s)}\). As all the random variables \(\one_{f(s) = g(s)}\) are independent, and have expectation \(\frac{1}{2}\) it holds by Chernoff's bound that
\[\Prob[f]{X_g(f) \geq \frac{1}{2} + \varepsilon} \leq e^{-\frac{\varepsilon^2}{1+\varepsilon} |X(k)|}.\]
There are \(\abs{B^k(X,\mathbb{F}_2)} \leq \abs{C_{k-1}(X,\mathbb{F}_2)} \leq 2^{|X(k-1)|}\) possible coboundaries. Hence, the probability that there exists some \(g\) so that \(X_g(f) \geq \frac{1}{2} + \varepsilon\) is at most
\[e^{|X(k-1)| \ln 2} e^{-\frac{\varepsilon^2}{1+\varepsilon} |X(k)|},\]
which is less than \(\frac{1}{2}\) by the assumption on \(\varepsilon\). Note that when \(X_g(f) \leq \frac{1}{2} + \varepsilon\) for all \(g\) this implies that \(\dist(f,Z_k(X,\mathbb{F}_2)) \geq \frac{1}{2}-\varepsilon\).

On the other hand, let \(Y(f) = \Prob[t \in X(k+1)]{f(t) = 0}\). As above it holds that \(\ex{Y} = \frac{1}{2}\). Furthermore, \(Y=\sum_{t \in X(k+1)} \prob{t}\one_{\coboundary f(t) =0}\) where  the set \(\set{\one_{\coboundary f(t) =0}}_{t \in X(k+1)}\) are pairwise independent. This is because any two distinct \(t,t' \in X(k+1)\) share only one \(k\)-face.
By pairwise independence, \(Var(Y(f)) = \sum_{t \in X(k+1)} Var(\prob{t}\one_{\coboundary f(t) =0}) = \frac{1}{4|X(k+1)|}.\)
In particular, 
\[\Prob[f]{Y(f) \leq \frac{1}{2} - \varepsilon} \leq \frac{1}{4|X(k+1)|\varepsilon^2}.\]
Which is also strictly less than \(\frac{1}{2}\) by assumption.

In particular, as the events \(\set{ \forall g ; X_g(f) \geq \frac{1}{2} + \varepsilon}\) and \(\set{Y(f) \leq \frac{1}{2} - \varepsilon}\) happen with probability less than \(\frac{1}{2}\), there exists some \(f\) so that for every \(g \in C_{k-1}(X,\Gamma)\), \(\dist(f,\coboundary g) \geq \frac{1}{2}-\varepsilon\) but \(wt(\coboundary f) \geq \frac{1}{2} + \varepsilon.\)
Hence \(h^{k}(f,\Gamma) \leq \frac{\frac{1}{2}+\varepsilon}{\frac{1}{2}-\varepsilon} \leq 1+8\varepsilon\).
\end{proof}